\documentclass{amsart}
\newcommand{\draftorfinal}{final}
\usepackage{mystyle}
\usepackage{img/circles}
\usepackage{img/growth}

\title{Optimal control of martingales in a radially symmetric environment}
\author{Alexander M.\ G.\ Cox}
\address{alexander, Bath, U.K.}
\email{a.m.g.cox@bath.ac.uk}
\author{Benjamin A.\ Robinson}
\address{Universit\"at Wien, Vienna, Austria}
\email{ben.robinson@univie.ac.at}
\date{\today}
\thanks{BR is supported by a scholarship from the EPSRC Centre for Doctoral Training in Statistical Applied Mathematics at Bath (SAMBa), under the project EP/L015684/1, and by the Austrian Science Fund (FWF) project Y782-N25.}

\begin{document}
	\maketitle
	
\begin{abstract}
	We study a stochastic control problem for continuous multidimensional martingales with fixed quadratic variation. In a radially symmetric environment, we are able to find an explicit solution to the control problem and find an optimal strategy. We show that it is optimal to switch between two strategies, depending only on the radius of the controlled process. The optimal strategies correspond to purely radial and purely tangential motion. It is notable that the value function exhibits smooth fit even when switching to tangential motion, where the radius of the optimal process is deterministic. Under sufficient regularity on the cost function, we prove optimality via viscosity solutions of a Hamilton-Jacobi-Bellman equation. We extend the results to cost functions that may become infinite at the origin. Extra care is required to solve the control problem in this case, since it is not clear how to define the optimal strategy with deterministic radius at the origin. Our results generalise some problems recently considered in Stochastic Portfolio Theory and Martingale Optimal Transport.
\end{abstract}

\section{Introduction}

In this paper we study a stochastic control problem for continuous multidimensional martingales with fixed quadratic variation. We work in a radially symmetric environment, where we are able to find optimal strategies and give the value function explicitly. We find that an optimal strategy is to switch between two behaviour regimes depending only on the current radius of the controlled martingale. Under one of the optimal strategies, which we will call \textit{tangential motion}, the controlled martingale has a deterministically increasing radius. This property leads to two notable features of the control problem. First, we make the observation that the value function exhibits smooth fit everywhere. When it is optimal to switch into the regime of tangential motion, continuous fit is sufficient to uniquely specify the value function, since the radius is deterministic here. Therefore it is surprising that smooth fit holds. Moreover, it is not obvious how to define tangential motion at the origin. As a result, solving the control problem at the origin requires extra care and we only find approximate optimisers here. We will see that the value function can remain finite when the cost function is allowed to be infinite at the origin. Under a particular growth condition on the cost function, approximation arguments break down and it is necessary to understand how to define tangential motion at the origin.

\subsection{Problem statement}

Fix a domain $D \subseteq \RR^d$, for some $d \geq 2$. We study the control problem of finding
\begin{equation}
	\inf_{\PP}\EE^{\PP} \left[\int_0^\tau f(X_t) \D t + g(X_\tau) \right],
\end{equation}
where $\tau$ is the first exit time of $X$ from $D$, and the infimum is taken over a set of probability measures under which $X$ is a continuous martingale with quadratic variation given by
\begin{equation}
	\D \, \langle X \rangle_t = \D t.
\end{equation}
We specialise to the \emph{radially symmetric} case, taking $D$ to be a $d$-dimensional ball, $f$ a function of the form $f(x) = \tilde{f}(\abs{x})$, and $g$ constant --- without loss of generality, we will assume that $g \equiv 0$. This structure allows us to work with the radius of the controlled processes.

\subsection{Related literature}
Problems of stochastic control for martingales with fixed quadratic variation have appeared recently in the context of Stochastic Portfolio Theory, in the two papers \cite{larsson_minimum_2020} and \cite{larsson_relative_2020} of Larsson and Ruf. In \cite{larsson_minimum_2020} the authors consider the problem of finding the greatest almost sure lower bound on the exit time of a martingale from some domain. They apply this control problem in \cite{larsson_relative_2020} to find the minimal time horizon over which relative arbitrage can be achieved for a market with at least two stocks. One of the strategies that we find to be optimal in the present paper is also studied in \cite{larsson_relative_2020}. While the control set in \cite{larsson_minimum_2020} and \cite{larsson_relative_2020} is the same as in the problem that we study, we consider a significantly different class of cost functions. In \cite[Theorem 1.1]{larsson_minimum_2020}, the value function is characterised as a viscosity solution of a Hamilton-Jacobi-Bellman (HJB) equation on a general compact set. In contrast to \cite{larsson_minimum_2020}, the HJB equation \eqref{eq:hjb-intro} for our control problem has no first order term, and our viscosity solution characterisation \Cref{thm:unique-viscosity} requires a strong notion of convexity on the domain. This convexity condition is of course satisfied when the domain is a ball in $\RR^d$; this is the case for the explicit results that we obtain in this paper, and so we do not investigate extending \Cref{thm:unique-viscosity} to more general domains here.

The HJB equation that arises from the control problem in this paper is
\begin{equation}\label{eq:hjb-intro}
	- \frac{1}{2} \inf_{\sigma \in U} \trace\left(D^2 u \sigma \sigma^\top\right) = f,
\end{equation}
where $U := \left\{\sigma \in \RR^{d, d} \colon \trace(\sigma \sigma^\top) = 1\right\}$. We can see this by considering a martingale that is adapted to the natural filtration of some Brownian motion $B$ and has the representation $\D X_t = \sigma_t \D B_t$. Then, under the quadratic variation constraint $\D \, \langle X \rangle_t = \D t$, we have $\sigma_t \in U$, for any $t \geq 0$. This HJB equation takes a similar form to the Black-Scholes-Barenblatt (BSB) equation, as studied in \cite{vargiolu_existence_2001}. Compared with the PDE \eqref{eq:hjb-intro}, the BSB equation has an additional time derivative term, and the infimum can be taken over a more general compact control set. The BSB equation is an HJB equation corresponding to a time-inhomogeneous control problem of the type discussed in Section 3.3 of \cite{touzi_optimal_2013}. In \cite{gozzi_superreplication_2002}, the BSB equation is applied to find a super-hedging strategy for European multi-asset derivatives.

\subsection{Motivation --- martingale optimal transport}\label{sec:mot}

In one dimension, it is well known that any continuous martingale is a time-change of a standard Brownian motion. Martingales with fixed quadratic variation are a natural generalisation of Brownian motion to higher dimensions. Imposing this constraint will allow us to study the structure of the optimal martingales in the control problems that we consider.

Our motivation for studying the problems in this paper comes from a connection with Martingale Optimal Transport (MOT), motivated by the paper \cite{tan_optimal_2013} of Tan and Touzi (see also \cite{lim_optimal_2020}). In \cite{tan_optimal_2013}, the authors formulated the martingale optimal transport problem through penalisation. Formally, MOT is the problem of finding a martingale $(M_t)_{t \in [0,1]}$ with marginal laws $M_0 \sim \lambda$ and $M_1 \sim \mu$ such that the joint distribution minimises some given quantity, for example $\EE\left[ |M_1-M_0|\right]$. In this setting, \cite{tan_optimal_2013} showed that the one-dimensional MOT problem could be reformulated as an optimal stopping problem for Brownian motion, where the connection to the Brownian motion is established by requiring $B_0 \sim \lambda$ and $B_{\tau} \sim \mu$ for some stopping time $\tau$. The penalisation arises as a cost function of the form $H(B_\tau)$ appearing in the optimisation criterion, or equivalently, via an It\^o argument, a cost function of the form $\int_0^\tau H''(B_s) \D s$.

For the MOT problem in general dimensions, the structure of the transports is much more complex than in the one-dimensional case, as described by De March \cite{de_march_local_2018}, De March and Touzi \cite{de_march_irreducible_2019}, and Ghoussoub, Kim and Lim \cite{ghoussoub_structure_2019}. Lim \cite{lim_optimal_2020} also considers MOT in higher dimensions, focussing on radially symmetric marginals, as we do here. 

We provide the following motivating example, in which we heuristically pass from an MOT problem to a control problem of the form that we study in this paper.

\begin{ex}
	Fix $d \in \NN$ and let $\lambda, \mu$ be radially symmetric probability measures on $\RR^d$. Consider the MOT problem of finding
	\begin{equation}
		MT(\lambda, \mu) := \inf_\tau \inf_{\substack{M \; \text{martingale}\\M_0 \sim \lambda, \; M_1 \sim \mu}} \EE [(H(M_\tau) - H(M_0))^2],
	\end{equation}
	for some function $H: \RR^d \to \RR$ depending only on the radius, where $\tau$ is a stopping time.
	
	By penalisation, we can rewrite
	\begin{equation}\label{eq:mot-penalised}
		\begin{split}
			MT(\lambda, \mu) & = \inf_\tau \inf_{\substack{M\; \text{martingale}\\M_0 \sim \lambda}} \sup_{f \in C_b(\RR^d)} \left\{ \EE [|H(M_\tau) - H(M_0)|^2 + f(M_\tau) ] - \int f(x) \mu (\D x)\right\}\\
			& =  \sup_{f \in C_b(\RR^d)} \left\{ \inf_\tau \inf_{\substack{M\; \text{martingale}\\M_0 \sim \lambda}} \EE [|H(M_\tau) - H(M_0)|^2 + f(M_\tau) ]  - \int f(x) \mu (\D x)\right\},
		\end{split}
	\end{equation}
	assuming that a min-max principle holds.
	
	Let $\mathcal U := \{U\text{-valued progressive processes}\}$ --- in particular, for any $\sigma \in \mathcal U$, we have $\trace(\sigma_t \sigma^\top_t) = 1$, for all $t \geq 0$. We restrict to martingales of the form $X^\sigma_t = X^\sigma_0 + \int_0^t \sigma_s \D B_s$, for some $\sigma \in \mathcal U$, with $X^\sigma_0 \sim \lambda$, and $B$ a $d$-dimensional Brownian motion. For a fixed $f$ and $\tau$, we are then interested in the problem of finding
	\begin{equation}
		V(\lambda) := \inf_{\substack{\sigma \in \mathcal{U}\\X^\sigma_0 \sim \lambda}}\EE[|H(X^\sigma_\tau) - H(X^\sigma_0)|^2 + f(X^\sigma_\tau)].
	\end{equation}
	By radial symmetry of the problem, we need only consider radially symmetric test functions $f$. Applying It\^o's formula, and writing the radius process as $R^\sigma_t = |X^\sigma_t|$, $t \geq 0$, we find functions $\tilde h, \tilde g: \RR_+ \to \RR$ such that
	\begin{equation}\label{eq:control-mot}
		V(\lambda) = \inf_{\substack{\sigma \in \mathcal{U}\\X^\sigma_0 \sim \lambda}}\EE\left[\int_0^\tau \left\{\tilde{h}(R^\sigma_t) \trace (\sigma_t \sigma^\top_t) + \tilde g(R^\sigma_t) \trace(X^\sigma_t {X^\sigma}^\top_t \sigma_t \sigma^\top_t)\right\}\D t\right].
	\end{equation}
	The corresponding HJB equation for this control problem is then
	\begin{equation}\label{eq:hjb-mot}
		\inf_{\sigma \in U} \trace(\sigma \sigma^\top [ \partial_{xx} u(x) - \tilde g(|x|) xx^\top ]) = \tilde h(|x|).
	\end{equation}
	
	In the present paper, we specialise to the case of $\tilde g \equiv 0$, so that we arrive at a control problem of the form
	\begin{equation}\label{eq:control-simple}
		\inf_{\substack{\sigma \in \mathcal{U}\\X^\sigma_0 \sim \lambda}}\EE\left[\int_0^\tau \tilde{h}(R^\sigma_t)\D t\right].
	\end{equation}
	In order to solve this control problem, we appeal to the theory of viscosity solutions for the HJB equation \eqref{eq:hjb-mot} with $\tilde g \equiv 0$ in \Cref{sec:proof-optimality}. In \Cref{app:dpp-comparison}, we prove a comparison result for this PDE, which to our knowledge is not contained in standard results in the literature.
	
	For a general $\tilde g$, the equation \eqref{eq:hjb-mot} does not fit into the framework of \Cref{app:dpp-comparison}. However, by modifying our proof of \Cref{thm:unique-viscosity}, we believe that it would be possible to extend these results to \eqref{eq:hjb-mot} in full generality.
	Moreover, we expect that an optimal control for \eqref{eq:control-mot} is to switch between radial and tangential motion, as we prove for \eqref{eq:control-simple}; see \Cref{rem:general-control}.
\end{ex}

Herein, we focus specifically on the challenge of understanding the optimal martingales for the simplified control problem \eqref{eq:control-simple}, as a first step towards the full problem of finding $V$ in \eqref{eq:control-mot} and the penalised value $MT(\lambda, \mu)$ in \eqref{eq:mot-penalised} above.

\subsection{Optimal behaviour}
A key result of this paper is to show that, under sufficient regularity on the cost function, an optimal strategy can be constructed by switching between the following two behaviours. We say that a martingale $X$ follows \emph{radial motion} if it can be written as 
\begin{equation}
	X_t = \frac{x}{\abs{x}}W_t, \quad t \geq 0,
\end{equation}
for some $x \in D \setminus \{0\}$ and $W$ a one-dimensional Brownian motion, so that $X$ acts as a one-dimensional Brownian motion on the line connecting its starting point to the origin, as illustrated in \Cref{fig:rad-tang}. Such a process maximises the expected time spent close to the origin and is therefore optimal for any cost function whose radial part is monotonically increasing. On the other hand, we say that $X$ follows \emph{tangential motion} in dimension two if it solves the SDE
\begin{equation}\label{eq:intro-tangential-sde}
	\D X_t = \frac{X^\perp_t}{\abs{X_t}} \D W_t,
\end{equation}
where $W$ is a one-dimensional Brownian motion, and we define a vector $x^\perp \in \RR^2$ orthogonal to $x \in D \setminus \{0\}$ by $x^\perp = (-x_2, x_1)^\top$. In higher dimensions, tangential motion has analogous behaviour to the two-dimensional case, but the formal definition is more delicate, as described in \Cref{def:tangential}. As we discuss below and illustrate in \Cref{fig:rad-tang}, tangential motion has a deterministically increasing radius. As a result, such a process minimises the expected time spent close to the origin and is optimal for cost functions with monotonically decreasing radial part. Tangential motion in dimension two already appears in \cite{fernholz_volatility_2018} and \cite{larsson_relative_2020}, and it is shown to be optimal for a simple example of a control problem in \cite[Example 1.6]{larsson_minimum_2020} --- see also \Cref{ex:step-decr} below.

In our main results, \Cref{prop:radial-symmetric-value} and \Cref{thm:value-relaxed}, we solve the control problem explicitly and give conditions under which switching between radial and tangential motion is optimal. Our approach to proving \Cref{prop:radial-symmetric-value} is first to construct a candidate value function by making the ansatz that a switching strategy of the above form is optimal, and then to verify that this is the correct value function by using the theory of viscosity solutions for the HJB equation. In \Cref{thm:value-relaxed} we extend \Cref{prop:radial-symmetric-value} by a series of approximation arguments. We will find optimal controls in a weak sense and show, in all but one case, that weak and strong formulations are equivalent.

\subsection{A switching problem with smooth fit}
The most interesting optimal behaviour is so-called tangential motion, described by the SDE \eqref{eq:intro-tangential-sde}. This behaviour is studied by Fernholz, Karatzas and Ruf in \cite{fernholz_volatility_2018}, and by Larsson and Ruf in \cite{larsson_relative_2020}. It is observed that, for a solution $X$ of \eqref{eq:intro-tangential-sde} with initial condition $x \in D$, the radius of $X$ is the deterministically increasing process
\begin{equation}
	t \mapsto \abs{X_t} = \sqrt{\abs{x} + t}.
\end{equation}
This property implies optimality of tangential motion for cost functions whose radial part is monotonically decreasing. For more general cost functions, we will show that an optimal strategy can be found by switching between this behaviour and radial motion, as described above. Due to the radial symmetry of the control problem, the optimal strategy at any time depends only on the current radial position of the controlled process. Therefore we have a one-dimensional switching problem between two regimes. It is common in such switching problems to observe smooth fit criteria on the free boundaries, and indeed, we do observe such behaviour in our case. Notably, we can only justify smooth fit \emph{heuristically} in one switching regime. 
  To determine the optimal radius at which to switch from radial to tangential motion, we only need to impose continuous fit for the value function. However, interestingly, smooth fit also holds at such a point. We discuss this phenomenon further in \Cref{sec:smooth-fit}.

\subsection{Behaviour at the origin}
Another interesting feature of the control problem is the optimal behaviour at the origin. In the case that the cost function has increasing radial part at the origin, then we can define an optimal strategy by taking a one-dimensional Brownian motion in any given direction, analogously to radial motion described above. On the other hand, in the case of a decreasing cost at the origin, it seems that tangential motion, as defined by the SDE \eqref{eq:intro-tangential-sde}, should be optimal, since such a process minimises the expected time spent at the origin. However, it is not immediately clear how to define tangential motion started from the origin.

In \cite{fernholz_volatility_2018}, the authors prove that a weak solution of \eqref{eq:intro-tangential-sde} exists in dimension $d = 2$, with initial condition $X_0 = 0$. Using this result, we can solve the control problem in a weak sense, which we define in \Cref{sec:problem-formulation}, following \cite{el_karoui_capacities_2013-1}. Under sufficient conditions on the cost, El Karoui and Tan show in \cite{el_karoui_capacities_2013-1} that this weak formulation is equivalent to a strong formulation of the control problem. However, we do not always assume that such conditions hold. In \Cref{sec:exploding-cost}, we will recover this equivalence result for more general costs, using some approximation arguments, but we do not identify an optimal control started from the origin. Under the particular growth regime where the cost $f$ satisfies
\begin{equation}
	\int_0^r \tilde{f}(s) \D s = \infty \qandq \int_0^r s \tilde{f}(s) \D s < \infty, \quad \text{for all} \quad r > 0,
\end{equation}
our approximation schemes diverge, but the weak solution of \eqref{eq:intro-tangential-sde} gives rise to a finite expected cost. This leads us to consider the existence of strong solutions of \eqref{eq:intro-tangential-sde} started from the origin. We leave open this question of existence of strong solutions, and the question of equivalence of weak and strong control problems, and we will address these issues in a forthcoming paper.

\subsection{Outline of the paper}
In \Cref{sec:problem-formulation}, we define the control problems precisely and prove some preliminary results on properties of the controlled processes and equivalence of the strong and weak control problems.

In \Cref{sec:discontinuities}, we present two motivating examples and introduce radial and tangential motion as candidates for the optimal behaviour. In the simple examples of step function costs, we prove directly that either radial or tangential motion is optimal, depending on whether the step function is increasing or decreasing.

In \Cref{sec:general-rad-symm}, we solve the control problem explicitly for radially symmetric cost functions that are sufficiently regular. We first conjecture that an optimal strategy is to switch between radial and tangential motion, and we heuristically derive a one-dimensional switching problem in \Cref{sec:switching-problem}. We then identify optimal switching points in \Cref{sec:switching-points} and go on to construct a candidate for the value function in \Cref{sec:construction}. We observe smooth fit at the points at which the conjectured optimal radius process switches to the deterministic regime of tangential motion, and we discuss this surprising phenomenon in \Cref{sec:smooth-fit}. In \Cref{sec:proof-optimality}, we prove \Cref{prop:radial-symmetric-value}, which states that the value function is equal to the candidate function that we constructed and that switching between radial and tangential motion is indeed optimal. To prove this result, we show that the candidate function satisfies a Hamilton-Jacobi-Bellman equation in the viscosity sense, and use the facts that the value function is also a viscosity solution of this equation, and that such a solution is unique. The required results on viscosity solutions are presented without proof in \Cref{app:dpp-comparison}.

In \Cref{sec:exploding-cost}, we relax the regularity conditions on the cost function, allowing for cost functions that are infinite at the origin. In this case, the theory of viscosity solutions in no longer applicable, and we do not know a priori that the weak and strong formulations of the control problem are equivalent. Under various growth conditions on the cost function, we use approximation arguments to show that all formulations of the control problem are equivalent and that the value function takes the same form as the candidate constructed in \Cref{sec:construction}. We also identify the conditions under which the value function remains finite. These results are summarised in \Cref{thm:value-relaxed}. We conclude by discussing one exceptional case under a specific growth condition in dimension $d = 2$. Here we find the weak value function, but equivalence with the strong formulation of the control problem is left open.

\section{Problem formulation}\label{sec:problem-formulation}
Fix $d \in \NN$. We introduce the control set
\begin{equation}
	U := \left\{\sigma \in \RR^{d,d} \colon \trace(\sigma\sigma^\top) = 1\right\}.
\end{equation}
Let $D \subset \RR^d$ be a domain and define the functions $f:D \to \RR$ and $g:\partial D \to \RR$, which we call the \emph{running cost} and \emph{boundary cost}, respectively. We make the following assumptions.

\begin{ass}\label{ass:main}
Suppose that
	\begin{enumerate}
		\item The domain $D$ is bounded;
		\item The cost functions $f$ and $g$ are upper semicontinuous;
		\item The running cost $f$ is bounded above; i.e. $f \leq M$, for some $M \geq 0$;
		\item The boundary cost $g$ is bounded above; i.e. $g \leq K$ for some $K \geq 0$.
	\end{enumerate}	
\end{ass}
Note that, in later sections, we will also impose radial symmetry on the problem.

We now introduce both weak and strong formulations of the control problem. We will show in \Cref{prop:weak-strong} that, under \Cref{ass:main}, the weak and strong formulations are equivalent. In \Cref{sec:exploding-cost}, we will relax our assumptions, so that \Cref{prop:weak-strong} no longer applies. However, we will show that equivalence of weak and strong formulations holds in all but one exceptional case.

\subsection{Strong formulation}
The strong formulation of the control problem is to find the strong value function $v^S: D \to \RR$, which we now define as in \cite{touzi_optimal_2013}. In order to define the value function, we introduce the set of controls, which will be $U$-valued processes, and we describe the dynamics of the controlled martingales via the stochastic integral \eqref{eq:sde} below.

Let $(\Omega_0, \mathcal{F}, \PP_0)$ be a probability space on which a $d$-dimensional Brownian motion $B$ is defined with natural filtration $\FF = (\mathcal{F}_t)_{t\geq 0}$.
\subsubsection*{Control:}
Define the set of controls $\mathcal{U}:=\left\{U\text{-valued }\FF\text{-progressively measurable processes} \right\}$.
\subsubsection*{Dynamics:}
For any $x \in D$ and $\nu = (\nu_t)_{t\geq0} \in \mathcal{U}$, define $X^\nu$ by the stochastic integral
\begin{equation}\label{eq:sde}
	X^\nu_t = x + \int_0^t \nu_s \D B_s, \quad t \geq 0.
\end{equation}
\begin{ex}\label{ex:lip-markov-strong}
	Let $\sigma: D \to U$ be Lipschitz. Then there is a unique strong solution $X^\sigma$ of the SDE
	\begin{equation}\label{eq:sde-lip}
		\D X_t = \sigma(X_t) \D B_t, \quad X_0 = x.
	\end{equation}
	Define $\nu_t = \sigma(X^\sigma_t)$, for all $t \geq 0$. Then $\nu \in \mathcal{U}$ and, for any $t \geq 0$, $X^\sigma_t = x + \int_0^t \nu_s \D B_s$.
\end{ex}

\subsubsection*{Value function:}
We define the \emph{strong value function} $v^S: D \to \RR$ by
\begin{equation}\label{eq:strong-value}
	v^S(x) := \inf_{\nu \in \mathcal{U}}\EE^x\left[\int_0^{\tau}f(X^\nu_s) \D s + g(X^\nu_{\tau^\nu})\right],
\end{equation}
where $\tau$ is the exit time of $X^\nu$ from the domain $D$, and $\EE^x$ denotes expectation with respect to the law of $X^\nu$, conditioned on $X^\nu_0 = x$.

\begin{remark}
	Note that, for any $\nu \in \mathcal{U}$, the quadratic variation of a controlled martingale $X^\nu$ is given by
	\begin{equation}
		\langle X^\nu \rangle_t = \int_0^t \trace(\nu_s \nu_s^\top) \D s = t,
	\end{equation}
	for any $t \geq 0$, by the definition of the control set $U$.
\end{remark}
	
\begin{defn}
	We say that a process $X$ has \emph{unit quadratic variation} if its quadratic variation is given by
	\begin{equation}
		\langle X \rangle_t = t, \quad \text{for all} \quad t \geq 0.
	\end{equation}
\end{defn}

A martingale with unit quadratic variation has the property that the expected exit time of the martingale from a ball is fixed. This gives a bound on the expected exit time from the domain $D$, as given in \cite[Lemma 1.7]{larsson_minimum_2020}, as follows.

\begin{prop}\label{rem:exit-time}
	Let $X$ be a continuous martingale with initial condition $X_0 = x \in D$, and suppose that $X$ has unit quadratic variation. Fix $R > 0$ and denote by $\tau_R$ the first exit time from $B_R(0) := \{y \in \RR^d \colon \abs{y} < R\}$. Then
	\begin{equation}\label{eq:exit-time}
		\EE^x [\tau_R] = R^2 - \abs{x}^2,
	\end{equation}
	Moreover, defining $\tau$ to be the first exit time from $D$, we have the bound
	\begin{equation}\label{eq:dom-exit-time}
		\EE^x[\tau] \leq \diam(D)^2 < \infty.
	\end{equation}
\end{prop}

\begin{proof}
 Applying It\^o's formula to $\abs{X_{\tau_R}}^2$ and taking expectations, we find that
 \begin{equation}
 	\EE^x\left[\abs{X_{\tau_R}}^2\right] - \abs{x}^2 = \EE^x \left[\langle X \rangle_{\tau_R}\right] = \EE^x[\tau_R],
 \end{equation}
 since $X$ is a martingale and has unit quadratic variation. Therefore, by continuity of the paths of $X$, we have
 \begin{equation}
 	\EE^x[\tau_R] = R^2 - \abs{x}^2.
 \end{equation}
Now set $R = \diam(D)$ so that $D \subseteq B_R(x)$. Then the inequality $\tau \leq \tau_R$ holds pointwise and, in particular,
\begin{equation}
	\EE^x[\tau] \leq \EE^x[\tau_R] = \diam(D)^2 < \infty,
\end{equation}
as required.
\end{proof}

\subsection{Weak formulation}
We now introduce the weak formulation of the control problem, following El Karoui and Tan in \cite{el_karoui_capacities_2013-1}. The problem is to find the weak value function $v^W: D \to \RR$, which we define below. In the weak formulation, the controls will take values in a set of probability measures, and the dynamics of the controlled martingales will be described as solutions of a local martingale problem.

Define the space of continuous paths $\Omega := C([0, \infty), \RR^d)$ and denote the set of Borel measurable functions $\nu: \RR_+ \to U$ by $\mathcal{B}(\RR_+, U)$. Then set $\overline{\Omega} = \Omega \times \mathcal{B}(\RR_+, U)$ and denote an element of $\overline{\Omega}$ by $\overline{\omega} = (\omega, u)$. Define the canonical process $\overline{X} = (X, \nu)$ on $\overline{\Omega}$ by $X_t(\overline{\omega}) = \omega_t$, for each $t \geq 0$, and $\nu(\overline{\omega}) = u$. We define the canonical filtration as in \cite{el_karoui_capacities_2013-1}. For $\phi \in C_b(\RR_+ \times U)$, $s \geq 0$, define
\begin{equation}
	M_s(\phi) := \int_0^s \phi(r, \nu_r) \D r.
\end{equation}
Then define the canonical filtration $\overline{\FF} = (\overline{\mathcal{F}}_t)_{t \geq 0}$ by
\begin{equation}
	\overline{\mathcal{F}}_t := \sigma\left\{(X_s, M_s(\phi)) \colon \phi \in C_b(\RR_+ \times U), s \leq t\right\}, \quad t \geq 0.
\end{equation}

\subsubsection*{Control:}
Let $\mathbb{M}$ be the set of probability measures on the set $\overline{\Omega}$. For each $x \in D$, let
\begin{equation}
	\mathbb{M}_x = \left\{\PP \in \mathbb{M} \colon \PP(X_0 = x) = 1\right\}.
\end{equation}
\subsubsection*{Dynamics:}
For $x \in D$, define
\begin{equation}
\begin{split}
	\mathcal{P}_x := \{\PP \in \mathbb{M}_x \, \colon \quad & t \mapsto \phi(X_t) - \phi(X_0) - \frac{1}{2}\int_0^t \trace\left(D^2\phi(X_s)\nu_s \nu_s^\top\right)\D s\\
	& \qquad \text{is a }(\overline{\FF}, \PP)\text{-local martingale for all }\phi\in C^2(\RR^d)\},
\end{split}
\end{equation}
and let $\tau = \inf\left\{t \geq 0 \colon X_t \notin D\right\}$.

\subsubsection*{Value function:}
Define the \emph{weak value function} $v^W: D \to \RR$ by
\begin{equation}\label{eq:weak-value}
	v^W(x) = \inf_{\PP \in \mathcal{P}_x} \EE^\PP\left[\int_0^{\tau}f(X_s) \D s + g(X_{\tau})\right],
\end{equation}
where $\EE^\PP$ denotes expectation with respect to the measure $\PP$.

\begin{remark}\label{rem:weak-control}
	A measure $\PP \in \mathcal{P}_x$ is a solution of a local martingale problem, as defined in Definition 4.5 of \cite[Chapter 5]{karatzas_brownian_1998}. As shown in Problem 4.3 and Proposition 4.6 of \cite[Chapter 5]{karatzas_brownian_1998}, there is a correspondence between solutions of a local martingale problem and weak solutions of an SDE. In our set up, a measure $\PP \in \mathcal{P}_x$ corresponds to a weak solution of an SDE $\D X_t = \sigma_t(X) \D W_t$ with initial distribution $\delta_x$, for some $U$-valued process $\sigma$. To save notation, we may also denote such a weak solution by $X^\sigma$ and refer to $\sigma$ as the control.
\end{remark}

We will now show that, under \Cref{ass:main}, the weak and strong value functions are equal, by referring to Theorem 4.5 of \cite{el_karoui_capacities_2013-1}.

\begin{prop}\label{prop:weak-strong}
	Suppose that \Cref{ass:main} holds. Then the weak and strong formulations of the control problem are equivalent; i.e.\ $v^S = v^W$ in $D$.
\end{prop}

\begin{proof}
	We apply Theorem 4.5 of \cite{el_karoui_capacities_2013-1}, which gives conditions for equality of the weak and strong value functions. Define a function $\Phi: \Omega \to \RR$ by
	\begin{equation}
		\Phi(\omega) = \int_0^{\tau(\omega)} f(X_s(\omega)) \D s + g(X_{\tau(\omega)}(\omega)),
	\end{equation}
	and fix $x \in D$. Then, by Theorem 4.5 of \cite{el_karoui_capacities_2013-1}, it is sufficient to show that $\Phi$ is upper semicontinuous and bounded above by some random variable $\xi$ that is uniformly integrable under the family of probability measures $\mathcal{P}_x$.
	
	Under our assumptions, $f: D \to \RR$ and $g: \partial D \to \RR$ are upper semicontinuous and so $\Phi$ is also upper semicontinuous. Since we have also assumed that $f$ and $g$ are bounded above, we have the bound
	\begin{equation}
		\Phi(\omega) \leq M \tau(\omega) + K =: \xi(\omega).
	\end{equation}
	Fix $\PP \in \mathcal{P}_x$ and let $(X, \nu)$ have joint law $\PP$. Then the process $X$ has unit quadratic variation, and so by \Cref{rem:exit-time}, $\EE^\PP [\tau] \leq \diam(D)^2$. Hence
	\begin{equation}
		\EE^\PP [\xi] \leq M \diam(D)^2 + K < \infty,
	\end{equation}
	independently of the choice of measure $\PP$. Therefore $\xi$ is uniformly integrable under $\mathcal{P}_x$.
	
	We apply Theorem 4.5 of \cite{el_karoui_capacities_2013-1} to conclude that $v^S(x) = v^W(x)$.
\end{proof}

With the result of \Cref{prop:weak-strong} in hand, we will write $v = v^W = v^S$ and refer to $v$ as the \emph{value function}. We henceforth choose to work with the strong formulation of the control problem, unless we explicitly refer to the weak formulation.

In the following section, we consider two simple examples of minimising and maximising the expected time spent in a ball about the origin.

\section{Occupation times}\label{sec:discontinuities}

From now on, we specialise to radially symmetric cost functions defined on a ball. Fix $d \geq 2$, $R > 0$ and let the domain be $D = B_R(0) \subset \RR^d$. Without loss of generality, we fix the boundary cost $g \equiv 0$. In this section, we will consider two examples that illustrate the optimality of the controlled processes that we call radial and tangential motion, which we later show to be optimal for more general cost functions. Note that, since the cost functions in this section have a discontinuity in $D \setminus \{0\}$, these examples are not covered by the more general results in this paper. However, some of the observations made in the proofs of optimality for these examples are used in later sections.

Since the problem is radially symmetric, we expect the value function $v$ to depend only on the radius. In fact, in the following examples, it will be convenient to work with the squared radius of any controlled process. We now derive an SDE for this squared radius process.

\begin{lemma}\label{lem:squared-radius}
	Let $x \in D$, $\sigma \in \mathcal{U}$, and define $X^\sigma$ by the stochastic integral
	\begin{equation}
		X^\sigma_t = x + \int_0^t\sigma_s \D B_s, \quad t \geq 0.
	\end{equation}
	Define the squared radius process $Z^\sigma$ by $Z^\sigma_t := \abs{X^\sigma_t}^2$, $t \geq 0$. Then $Z^\sigma$ satisfies the SDE
	\begin{equation}\label{eq:squared-radius}
		\D Z^\sigma_t = 2 X_t^{\top}\sigma_t \D B_t + \D t, \quad Z^\sigma_0 = \abs{x}^2.
	\end{equation}
\end{lemma}

\begin{proof}
	We apply It\^o's formula, along with the constraint that $\sigma_t \in U$, to find
	\begin{equation}
	\begin{split}
		\D Z^\sigma_t & = 2 {X^\sigma}^{\top}\sigma_t \D B_t + \trace(\sigma_t\sigma^{\top}_t)\D t\\
		& = 2 {X^\sigma}^{\top}\sigma_t \D B_t + \D t.\qedhere
	\end{split}
	\end{equation}
\end{proof}

\subsection{Tangential motion}\label{sec:tangential}

We first consider the following example of minimising the expected time spent in a ball about the origin that is contained in $D$. We note the similarities with \cite[Example 1.6]{larsson_minimum_2020}, for which the same control is shown to be optimal.

\begin{ex}\label{ex:step-decr}	
	Let $\rho \in (0, R)$ and define $f: D \to \RR$ by
	\begin{equation}
		f(x) = \begin{cases}
	 	0, & \abs{x} \leq \rho,\\
	 	-1, & \abs{x} \in (\rho, R).
	 \end{cases}
	\end{equation}
	We seek the value function
	\begin{equation}
		v(x) = \inf_{\sigma \in \mathcal{U}}\EE^x\left[\int_0^{\tau}f(X^{\sigma}_s)\D s\right] = \inf_{\sigma \in \mathcal{U}}\EE^x\left[\int_0^{\tau}-\ind{\abs{X^{\sigma}_s}\in(\rho, R)}\D s\right].
	\end{equation}
	That is, we wish to maximise the expected time that the radius process $\abs{X^{\sigma}}$ spends in the interval $(\rho, R)$. The cost function is shown in \Cref{fig:cost-decr}.
\end{ex}

\begin{figure}[h]
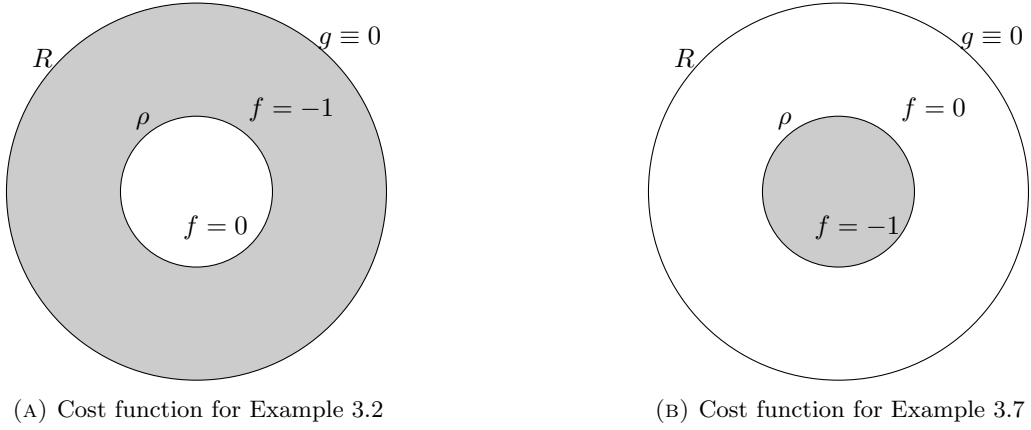

	\centering
	\begin{subfigure}{0.49\textwidth}
		\centering
		\circles
		\caption{Cost function for \Cref{ex:step-decr}}
		\label{fig:cost-decr}\par
	\end{subfigure}
	\begin{subfigure}{0.49\textwidth}
		\centering
		\circlesdecr
		\caption{Cost function for \Cref{ex:step-incr}}
		\label{fig:cost-incr}\par
	\end{subfigure}
	\caption{Cost functions for the examples of \Cref{sec:discontinuities}}
\end{figure}

We conjecture that, on the boundary $\{x \in D \colon \abs{x} = \rho\}$, any optimal control for \Cref{ex:step-decr} must enforce motion tangential to this internal boundary. We now define a process that exhibits this behaviour. 

\begin{defn}[Tangential motion]\label{def:tangential}
	$\mathit{d = 2}$: Fix $x = (x_1, x_2)^\top \in D \setminus \{0\}$, define an orthogonal vector $x^\perp := (-x_2, x_1)$, and define the function $\sigma^0: D \to \RR^{2, 2}$
	\begin{equation}\label{eq:tangential-function}
		\sigma^0(x) := \frac{1}{\abs{x}}
		\begin{bmatrix}
			x^\perp; & 0
		\end{bmatrix}.
	\end{equation}
	Fix $x \in D \setminus\{0\}$ and suppose that $X^{\sigma^0}$ is the strong solution of the SDE
	\begin{equation}
		\D X_t = \sigma^0(X_t)\D B_t, \quad X_0 = x;
	\end{equation}
	see \Cref{rem:tang-strong} for existence and uniqueness.
	To save notation, also denote by $\sigma^0 \in \mathcal U$ the control process defined by $\sigma^0_t := \sigma^0(X^{\sigma^0}_t)$, $t \geq 0$, so that
	\begin{equation}
		X^{\sigma^0}_t = x + \int_0^t \sigma^0_s \D B_s, \quad t \ge 0.
	\end{equation}
	We say that the process $X^{\sigma^0}$ follows \emph{tangential motion}.
	
	We cannot define $\sigma^0$ in the same way for $d \geq 3$, since it is not possible to continuously define an orthogonal vector in general dimensions. We therefore adapt the definition to obtain a process with analogous behaviour to the two-dimensional case.
	
	$\mathit{d \geq 3}$: Define the control $\sigma^0 \in \mathcal{U}$ locally as follows. Fix $x \in D \setminus \{0\}$ and a halfspace $H_x \ni x$. Set $X^{\sigma^0}_0 = x$. Then, for all $y \in H_x$, we can define continuously an orthogonal vector $y^\perp \in \RR^d \setminus\{0\}$ such that $y^\top y^\perp = 0$ and $|y^\perp| = |y|$. Define the function $\sigma^0: D \cap H_x \to \RR^{d, d}$ by $\sigma^0(y) := \frac{1}{|y|}\begin{bmatrix}
			y^\perp; & 0 & \dots; & 0
	\end{bmatrix}$, and let $X^{\sigma^0}$ be the strong solution of
	\begin{equation}\label{eq:tang-high-dim}
		\D X_t = \sigma^0(X_t)\D B_t,
	\end{equation}
	for $s \in (0, \tau_x)$, where $\tau_x := \inf \{s \ge 0 \colon \; X^{\sigma^0}_s \notin H_x\}$. Note that, once again, justification for existence and uniqueness of strong solutions is given in \Cref{rem:tang-strong}. Also define the control $\sigma^0$ by $\sigma^0_s = \sigma^0(X^{\sigma^0}_s)$ for $s \in (0, \tau_x)$. Now fix a halfspace $H_{X^{\sigma^0}_{\tau_x}} \ni X^{\sigma^0}_{\tau_x}$ on which we can define continuously an orthogonal vector once again. Let us redefine the function $\sigma^0:D \cap H_{X^{\sigma^0}_{\tau_x}} \to \RR^{d, d}$ accordingly on this halfspace and let $X^{\sigma^0}$ be the strong solution of \eqref{eq:tang-high-dim} until first exit of the halfspace. Again define the control $\sigma^0$ on this time interval by $\sigma^0_s = \sigma^0(X^{\sigma^0}_s)$.
	
	Continuing this construction iteratively for all times $s \leq \tau_D$, we define a control $\sigma^0 \in \mathcal U$ and a process $X^{\sigma^0}$ which we say follows \emph{tangential motion in dimension $d \geq 3$}.
	\end{defn}

Note that tangential motion is not defined at the origin. In later sections we will need to take care when considering cases where tangential motion is optimal close to the origin.

For $\sigma^0$ defined in \Cref{def:tangential}, we can find a formula for the squared radius process $Z^{\sigma^0}$ via \Cref{lem:squared-radius}, as follows.

\begin{lemma}\label{lem:deterministically-increasing}
	Suppose that $X^{\sigma^0}$ follows tangential motion, as defined in \Cref{def:tangential}, with $X^{\sigma^0}_0 = x \neq 0$. Then the radius process is deterministically increasing and, for any $t \geq 0$,
	\begin{equation}
		Z^{\sigma^0}_t = \abs{X^{\sigma^0}_t}^2 = \abs{x} + t.
	\end{equation}
\end{lemma}

\begin{proof}
	For $t \geq 0$, provided that $\big \lvert X^{\sigma^0}_t\big\rvert \neq 0$, we see that there is some orthogonal vector $\left(X^{\sigma^0}_t\right)^\perp$, as defined in \Cref{def:tangential}, such that
	\begin{equation}
		\left(X^{\sigma^0}_t\right)^{\top}\sigma^0_t = \frac{1}{\abs{X^{\sigma^0}_t}}
		\begin{bmatrix}
	 		\left(X^{\sigma^0}_t\right)^\top \left(X^{\sigma^0}_t\right)^\perp, & 0, & \dotsc, & 0
		\end{bmatrix} = 
		\begin{bmatrix}
	 		0, & \dotsc, & 0
	 	\end{bmatrix}.
	\end{equation}
	Therefore, by \Cref{lem:squared-radius}, $Z^{\sigma^0}$ satisfies
	\begin{equation}
		\D Z^{\sigma^0}_t = \D t.	
	\end{equation}
	Let $\xi = \abs{x}^2 \neq 0$, so that $Z^{\sigma^0}_0 = \xi$. Then $Z^{\sigma^0}$ is the deterministically increasing process given by
	\begin{equation}
	\begin{split}
		t \mapsto Z^{\sigma^0}_t = \xi + t. \qedhere
	\end{split}
	\end{equation}
\end{proof}

\begin{remark}\label{rem:tang-strong}
	As a consequence of the above lemma, supposing that $X^{\sigma^0}_0 \neq 0$, we have that $\abs{X^{\sigma^0}_t} > 0$ for all $t \geq 0$. Note also that each of the functions $\sigma^0$ in \eqref{eq:tangential-function} and \eqref{eq:tang-high-dim} is Lipschitz continuous in the set on which it is defined. By adapting standard results of It\^o on existence and uniqueness of strong solutions of SDEs with Lipschitz coefficients, we can show that existence and uniqueness holds for \eqref{eq:tangential-function} and \eqref{eq:tang-high-dim}. Therefore the control $\sigma^0$ in \Cref{def:tangential} is well-defined.
\end{remark}

The observation that the process $X^{\sigma^0}$ has deterministically increasing radius was made by Fernholtz, Karatzas and Ruf in Section 6.2 of \cite{fernholz_volatility_2018} and again by Larsson and Ruf in Section 4.2 of \cite{larsson_relative_2020}, where they consider a problem of relative arbitrage.

In \Cref{fig:tangential-motion}, we show a simulated trajectory of a process following tangential motion in dimension $d = 2$. We note that a similar simulation is produced in Figure 2 of \cite{larsson_relative_2020}.

Having defined tangential motion and proved a key property of this process, we now construct a candidate for the value function in \Cref{ex:step-decr}.

Fix $\xi \geq \rho^2$. Then we conjecture that the control $\sigma^0$ defined in \Cref{def:tangential} is optimal, and we compute the expected cost
\begin{equation}
\begin{split}
	\EE^x\left[\int_0^{\tau}-\mathds{1}_{\{\abs{X^{\sigma^0}_s}\in(\rho, R)\}}\D s\right] & = \EE^\xi\left[\int_0^{\tau}-\mathds{1}_{\{Z^{\sigma^0}_s\in(\rho^2, R^2)\}}\D s\right] = - \int_0^\infty \mathds{1}_{\left\{s \in (0, R^2 - \xi)\right\}}\D s\\
	& = - \int_{0}^{R^2 - \xi}\D s = \xi - R^2.
\end{split}
\end{equation}

Now suppose that $\xi < \rho^2$. This includes the case where the process starts at the origin, where the control $\sigma^0$ is not well-defined. However, since the cost is zero in the ball $\{x \in \RR^d \colon \abs{x} < \rho\}$, we will see that any strategy is optimal in this region. For a fixed $r \in (\sqrt{\xi}, \rho)$ and an arbitrary $\sigma \in \mathcal{U}$, define the control $\sigma^\star$ by
\begin{equation}
	\sigma^\star_t =
	\begin{cases}
		\sigma_t, & \abs{X^{\sigma^\star}_t} < r,\\
		\sigma^0_t, & \abs{X^{\sigma^\star}_t} \in [r, R).
	\end{cases}
\end{equation}
Then we compute the expected cost
\begin{equation}
\begin{split}
	\EE^x\left[\int_0^{\tau}-\mathds{1}_{\{\abs{X^{\sigma^\star}_s}\in(\rho, R)\}}\D s\right] & = \EE^\xi\left[\int_0^{\tau}-\mathds{1}_{\{Z^{\sigma^\star}_s\in(\rho^2, R^2)\}}\D s\right] = \EE^{r^2}\left[\int_0^{\tau}-\mathds{1}_{\{Z^{\sigma^0}_s\in(\rho^2, R^2)\}}\D s\right]\\
	& = - \int_{r^2 - \xi}^\infty \mathds{1}_{\left\{s \in (\rho^2 - \xi, R^2 - \xi)\right\}}\D s = - \int_{\rho^2}^{R^2}\D s = \rho^2 - R^2.
\end{split}
\end{equation}

This calculation gives us a conjecture for the value function in \Cref{ex:step-decr}. Using the It\^o-Tanaka formula, we will show that our candidate function satisfies a dynamic programming principle, as described in \Cref{app:dpp-comparison}, and we can then deduce that this function must be the value function.

\begin{prop}\label{prop:step-value}
Let $w: [0, R^2) \to \RR$ be defined by
\begin{equation}
		w(\xi) =
		\begin{cases}
				\rho^2 - R^2, & \xi \leq \rho^2,\\
				\xi - R^2, & \xi \in (\rho^2, R^2),
		\end{cases}
\end{equation}
and define $\overline{v}: D \to \RR$ by $\overline{v}(x) = w(\abs{x}^2)$, for $x \in D$. Then the value function for \Cref{ex:step-decr} is given by $v = v^S = v^W = \overline{v}$.
\end{prop}

\begin{figure}[h]
  \centering
  \includegraphics[width = 0.4\textwidth]{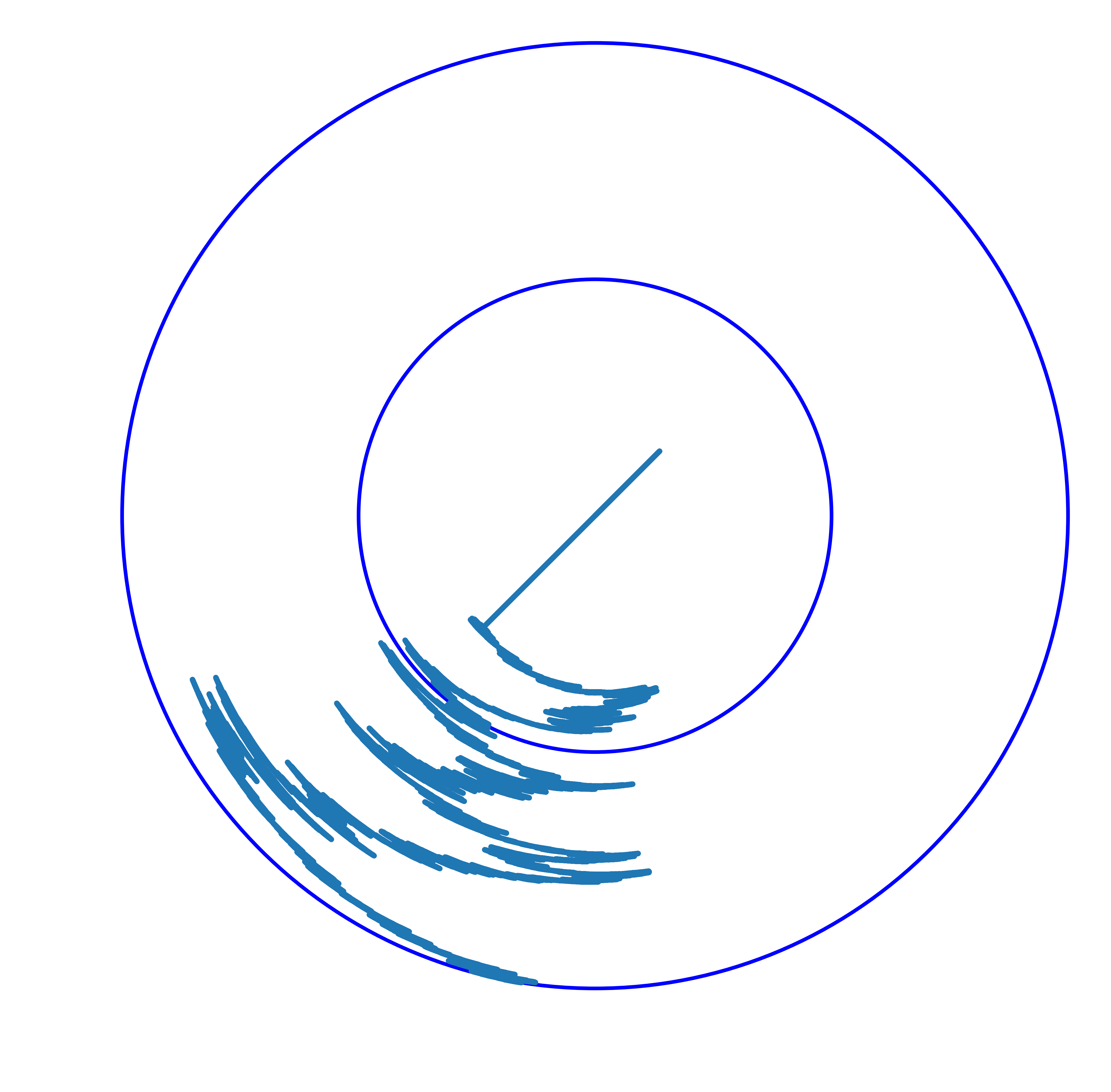}
  \caption{A possible trajectory for an optimal strategy in \Cref{ex:step-decr}}
\end{figure}

\begin{proof}
	We first show that $\overline{v}$ satisfies the form of the dynamic programming principle given in \Cref{rem:dpp}.
	
	Define $\tilde{f}: [0, R^2) \to \RR$ such that $f(x) = \tilde{f}(\abs{x})$, for $x \in D$; i.e.
	\begin{equation}
		\tilde{f}(\xi) = - \ind{\xi \in (\rho^2, R^2)}, \quad \xi \in [0, R^2).
	\end{equation}
	We seek to prove that $w(Z^\sigma_t) + \int_0^t \tilde{f}(Z^\sigma_s)\D s$ is a submartingale for all $\sigma \in \mathcal{U}$, and that $w(Z^{\sigma^\star}_t) + \int_0^t \tilde{f}(Z^{\sigma^\star}_s)\D s$ is a martingale for an optimal strategy $\sigma^{\star}\in \mathcal{U}$.
	
	Let $\sigma \in \mathcal{U}$. We note that $w$ is not continuously differentiable at $\xi = \rho^2$, so we apply the It\^o-Tanaka formula to write down an SDE for $w(Z^\sigma_t)$. We calculate that the left derivative of $w$ is given by
	\begin{equation}
		w_{-}^{\prime}(\xi) =
		\begin{cases}
			0, & \text{for} \quad \xi \leq \rho^2,\\
			1, & \text{for} \quad \xi \in (\rho^2, R^2),	
		\end{cases}
	\end{equation}
	and the distributional derivative of $w^\prime_-$ is
	\begin{equation}
		w^{\prime\prime}(\D a) = \delta_{\rho^2} (\D a).
	\end{equation}
	Hence, by the It\^o-Tanaka formula,
	\begin{equation}\label{eq:ex-ito-tanaka}
	\begin{split}
		w(Z^\sigma_t) - w(\xi) + \int_0^t \tilde{f}(Z^\sigma_s)\D s & = 2 \int_0^t \mathds{1}_{\{Z^\sigma_s > \rho^2\}} X_s^{\top} \sigma_s \D B_s + \int_0^t \mathds{1}_{\{Z^\sigma_s > \rho^2\}}\D s + \frac{1}{2}L_t^{\sigma, \rho^2} - \int_0^t \mathds{1}_{\{Z^\sigma_s > \rho^2\}}\D s\\
		& = 2 \int_0^t \mathds{1}_{\{Z^\sigma_s > \rho^2\}} X_s^{\top} \sigma_s \D B_s + \frac{1}{2}L_t^{\sigma, \rho^2},
	\end{split}
	\end{equation}
	where $L_t^{\sigma, \rho^2}$ denotes the local time of $Z^\sigma$ at $\rho^2$ at time $t$.
	Since local time is always non-negative, we have shown that
	\begin{equation}
		w(Z^\sigma_t) + \int_0^t\tilde{f}(Z^\sigma_s)\D s
	\end{equation}
	is a submartingale for any $\sigma \in \mathcal{U}$.
	
	Now we note that, for any $\sigma \in \mathcal{U}$,
	\begin{equation}
		\overline{v}(X_\tau^\sigma) = w(Z_\tau^\sigma) = w(R^2) = 0,
	\end{equation}
	by continuity of the paths of $X^\sigma$. Therefore, we can use the submartingale property and the optional sampling theorem to find that
	\begin{equation}
	\begin{split}
		\EE^x \left[\int_0^\tau f(X_s^\sigma) \D s \right] & = \EE^x \left[\int_0^\tau f(X_s^\sigma) \D s + \overline{v}(X^\sigma_\tau)\right]\\
		& = \EE^\xi \left[\int_0^\tau \tilde{f}(Z^\sigma_s)\D s + w(Z^\sigma_\tau)\right] \geq w(\xi) = \overline{v}(x).
	\end{split}
	\end{equation}
	
	Now, supposing that $\xi \neq 0$, consider the control $\sigma^\star = \sigma^0$, so that $Z^{\sigma^\star}_t = \xi + t$, for any $t \geq 0$. Then
	\begin{equation}
		\EE^\xi \left[- \int_0^\tau \ind{Z^{\sigma^\star}_s \in (\rho^2, R^2)} \D s\right] = - \int_0^\infty \ind{s \in (\rho^2 - \xi \wedge \rho^2, R^2 - \xi)} \D s =
		\begin{cases}
		 \rho^2 - R^2, & \xi \in (0, \rho^2],\\
		 \xi - R^2, & \xi \in (\rho^2, R^2).
		\end{cases}
	\end{equation}
	In the case that $\xi = 0$, fix $r \in (0, \rho)$ and $\sigma \in \mathcal{U}$, and take
	\begin{equation}
		\sigma^\star_t =
		\begin{cases}
			\sigma_t, & \abs{X^{\sigma^\star}_t} < r,\\
			\sigma^0_t, & \abs{X^{\sigma^\star}_t} \in [r, R).
		\end{cases}
	\end{equation}
	Then
	\begin{equation}
		\EE^0\left[ - \int_0^\tau \ind{Z^{\sigma^\star}_s \in (\rho^2, R^2)} \D s\right] = \EE^{r^2}\left[ - \int_0^\tau \ind{Z^{\sigma^\star}_s \in (\rho^2, R^2)} \D s\right] = \rho^2 -R^2.
	\end{equation}
	
	We conclude that, for any $x \in D$,
	\begin{equation}
		\overline{v}(x) = \inf_{\sigma \in \mathcal{U}}\EE^x \left[\int_0^\tau f(X_s^\sigma) \D s \right] = v^S(x) = v^W(x).
	\end{equation}
	Hence the conjectured function $\overline{v}$ is indeed the value function.
\end{proof}

\subsection{Radial motion}

We now turn to a second example of maximising the expected time spent in a ball around the origin.

\begin{ex}\label{ex:step-incr}
	Fix $\rho \in (0, R)$ and define the cost $f: D \to \RR$ by
	\begin{equation}
	f(x) =
	\begin{cases}
		-1, & \abs{x} < \rho\\
		0, & \abs{x} \in [\rho, R).
	\end{cases}
	\end{equation}
	We seek the value function
\begin{equation}
	v(x) = \inf_{\sigma\in \mathcal{U}} \EE^x\left[\int_0^\tau f(X^\sigma_s) \D s\right] = \inf_{\sigma \in \mathcal{U}}\EE^x\left[\int_0^\tau - \mathds{1}_{\left\{\abs{X^\sigma_s} < \rho\right\}}\D s \right].
\end{equation}
	That is, we wish to maximise the expected time that the martingale spends in the ball $B_\rho(0)$. The cost function is shown in \Cref{fig:cost-incr}.
\end{ex}

We propose that an optimal strategy is to run as a Brownian motion on the radius of the domain. We now define a process that follows this strategy.
\begin{defn}[Radial motion]\label{def:radial}
	Define a function $\sigma^1: D \to \RR$ by
	\begin{equation}\label{eq:radial-function}
		\sigma^1(x) =
		\begin{cases}
			\frac{1}{\abs{x}}
			\begin{bmatrix}
				x; & 0; & \cdots; & 0
			\end{bmatrix}, & x \neq 0,\\
			\begin{bmatrix}
				e_1; & 0; & \cdots; & 0 
			\end{bmatrix}, & x = 0,
		\end{cases}
	\end{equation}
	where $e_1$ is the unit vector in the first coordinate direction. Fix $x \in D$ and define $\sigma^1$ to be the constant control given by $\sigma^1_t = \sigma(x)$, for all $t \geq 0$. Define $X^{\sigma^1}$ by
	\begin{equation}
		X^{\sigma^1}_t = x + \int_0^t \sigma^1_s \D B_s = x + \sigma^1(x) B_t, \quad t \geq 0.
	\end{equation}
	We say that the process $X^{\sigma^1}$ follows \emph{radial motion}.
\end{defn}

A simulated trajectory of a process following radial motion is shown in \Cref{fig:radial-motion}, along with the sample path of its radius in \Cref{fig:radial-motion-rad}.

\begin{figure}[h]
\centering
	\begin{subfigure}{0.4\textwidth}
	\centering
		\includegraphics[width = 0.8\textwidth]{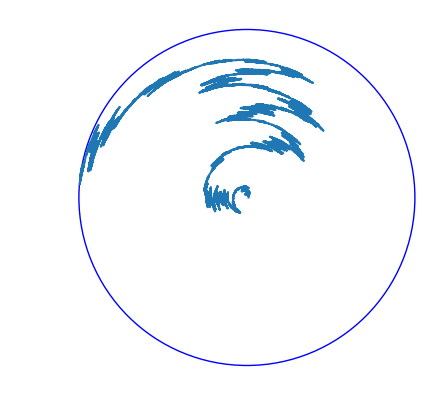}\vspace{-4.5ex}\caption{}\label{fig:tangential-motion}\par
	\end{subfigure}
	\begin{subfigure}{0.4\textwidth}
	\centering
		\includegraphics[width = 0.8\textwidth]{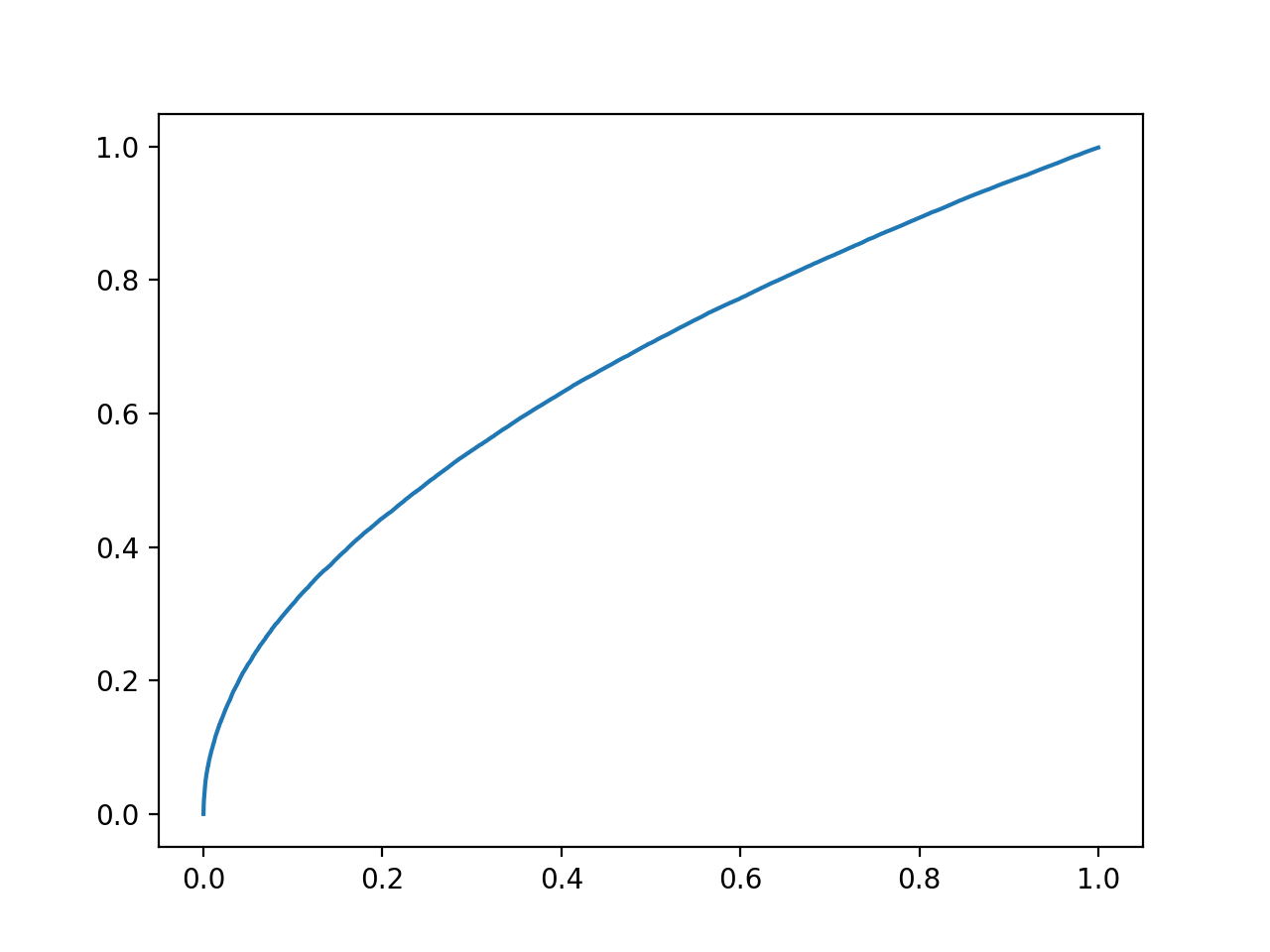}\vspace{1ex}\caption{}\label{fig:tangential-motion-rad}\par
	\end{subfigure}\\
	\begin{subfigure}{0.4\textwidth}
		\centering
			\includegraphics[width = 0.8\textwidth]{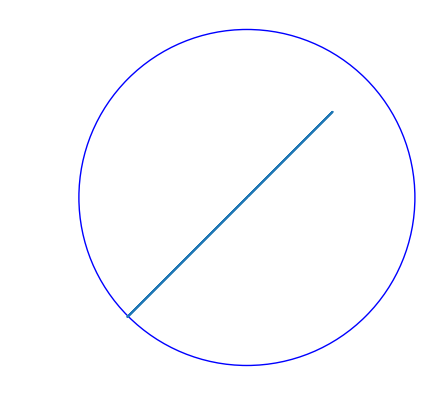}\vspace{-4.5ex}\caption{}\label{fig:radial-motion}\par
		\end{subfigure}
		\begin{subfigure}{0.4\textwidth}
		\centering
			\includegraphics[width = 0.8\textwidth]{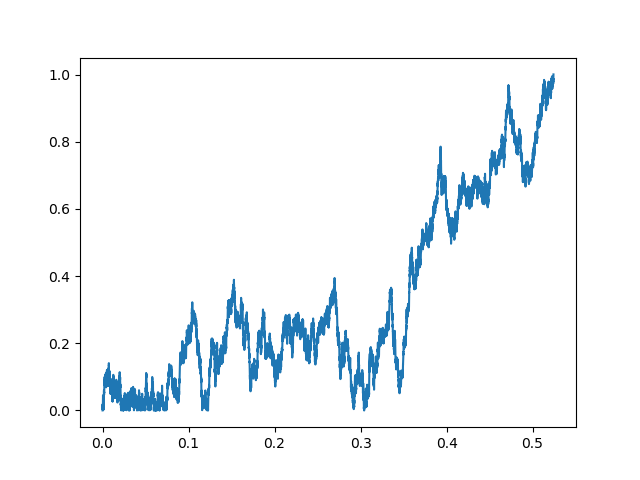}\vspace{1ex}\caption{}\label{fig:radial-motion-rad}\par
		\end{subfigure}
	\caption{Sample paths of processes $(X^{\sigma^0}_t)_{t \geq 0}$ following tangential motion and $(X^{\sigma^1}_t)_{t \geq 0}$  following radial motion are shown in (\textsc{a}) and (\textsc{c}), respectively. The corresponding radius processes are shown in (\textsc{b}) and (\textsc{d}) respectively.}\label{fig:rad-tang}
\end{figure}

Let $W$ be the first component of $B$, and note that $W$ is a one-dimensional Brownian motion. Then, defining $\sigma^1$ as in \Cref{def:radial}, we see that
\begin{equation}
	X^{\sigma^1}_t =
	\begin{cases}
 		x + \int_0^t \sigma^1_s \D B_s = x + W_t \frac{x}{\abs{x}}, & x \neq 0,\\
 		W_t e_1, & x = 0.
 	\end{cases}
\end{equation}
Hence $\big\lvert X^{\sigma^1}_t\big \rvert = \abs{\abs{x} + W_t}$, and so
\begin{equation}
	\EE^x\left[\int_0^\tau -\mathds{1}_{\left\{\abs{X^{\sigma^1}_t} < \rho\right\}} \right] = \EE^{\abs{x}}\left[-\int_0^{\tau_{R}} \mathds{1}_{\left\{W_t \in (-\rho, \rho)\right\}} \right],
\end{equation}
where $\tau_{R} = \inf \left\{t \geq 0 \colon \abs{W_t} = R\right\}$.

We can compute this expected cost by using the Green's function and speed measure for one-dimensional Brownian motion, as defined in Section 3 of \cite[Chapter VII]{revuz_continuous_1999}. The speed measure $m$ is given by
$\int m(\D y) = 2 \int \D y$, and the Green's function $G$ on the interval $[-\rho, \rho]$ is given by
\begin{equation}
	G(r, y) =
	\begin{cases}
		\frac{\left(y + R\right)\left(R - r\right)}{2 R}, & y \leq r,\\
		\frac{\left(r + R\right)\left(R - y\right)}{2 R}, & y \geq r,
	\end{cases}
\end{equation}
for $r, y \in [-\rho, \rho]$. By Corollary 3.8 of \cite[Chapter VII]{revuz_continuous_1999}, we have
\begin{equation}
\begin{split}
	\EE^{\abs{x}}\left[-\int_0^{\tau_{R}} \mathds{1}_{\left\{W_t \in (-\rho, \rho)\right\}} \right] & = -\int_{-\rho}^{\rho}G(\abs{x}, y) m(\D y) =
	\begin{cases}
		\abs{x}^2 + \rho^2 - 2\rho R, & \abs{x} < \rho,\\
		2 \rho \abs{x} - 2 \rho R, & \abs{x} \geq \rho.
	\end{cases}
\end{split}
\end{equation}
This gives us a candidate for the value function in \Cref{ex:step-incr}. Again, we present this function in terms of the radius squared. We can then apply It\^o's formula, using the SDE for the squared radius process that we derived in \Cref{lem:squared-radius}.

\begin{prop}\label{prop:step-incr}
	Let $w: [0, R^2) \to \RR$ be defined by
	\begin{equation}
		w(\xi) =
		\begin{cases}
			\xi + \rho^2 - 2\rho R, & \xi \leq \rho^2,\\
			2\rho \xi^{\frac{1}{2}} - 2\rho R, & \xi \in (\rho^2, R^2),
		\end{cases}
	\end{equation}
	and define $\overline{v}: D \to \RR$ by $\overline{v}(x) = w(\abs{x}^2)$, for $x \in D$. Then the value function for \Cref{ex:step-incr} is given by $v = v^S = v^W = \overline{v}$.
\end{prop}

\begin{proof}[Proof of \Cref{prop:step-incr}]
	Again we will show that $\overline{v}$ satisfies the form of the dynamic programming principle given in \Cref{rem:dpp}.
	
	Note first that $w$ is continuously differentiable and twice piecewise continuously differentiable, with
	
	\begin{equation}
		w^\prime (\xi) =
		\begin{cases}
			1, & \xi \leq \rho^2,\\
			\rho \xi^{- \frac{1}{2}}, & \xi \in (\rho^2, R^2),
		\end{cases}
	\quad \text{ and } \qquad
		w^{\prime \prime}(\xi) =
		\begin{cases}
			0, & \xi \leq \rho^2,\\
			- \frac{1}{2} \rho \xi^{- \frac{3}{2}}, & \xi \in (\rho^2, R^2).
		\end{cases}
	\end{equation}
	Hence we can apply It\^o's formula to $w(Z^\sigma_t)$, for any $\sigma \in \mathcal{U}$, recalling that $Z^\sigma_t = \abs{X^\sigma_t}^2$. Let $Z^\sigma_0 = \xi \in [0, R^2)$. Then, for $t > 0$,
	\begin{equation}\label{eq:ex-ito-square}
	\begin{split}
		w(Z^\sigma_t) - w(\xi) & = \int_0^t \mathds{1}_{\left\{Z^\sigma_s \leq \rho^2\right\}}\D Z^\sigma_s + \rho \int_0^t \mathds{1}_{\left\{Z^\sigma_s \in (\rho^2, R^2)\right\}}{(Z^\sigma_s)}^{-\frac{1}{2}}\D Z^\sigma_s\\
		& \quad - \frac{\rho}{4} \int_0^t \mathds{1}_{\left\{Z^\sigma_s \in (\rho^2, R^2)\right\}} {(Z^\sigma_s)}^{-\frac{3}{2}} \D \, \langle Z^\sigma \rangle_s.
	\end{split}
	\end{equation}
	Substituting in the SDE \eqref{eq:squared-radius} for $Z^\sigma$, we find that there is a square-integrable martingale $M^\sigma$ such that
	\begin{equation}
	\begin{split}
		w(Z^\sigma_t) - w(\xi) & = \int_0^t \D M^\sigma_s + \int_0^t \mathds{1}_{\left\{Z^\sigma_s \leq \rho^2\right\}}\D s + \rho \int_0^t \mathds{1}_{\left\{Z^\sigma_s \in (\rho^2, R^2)\right\}}{(Z^\sigma_s)}^{-\frac{1}{2}}\D s\\
		& \quad - \rho \int_0^t \mathds{1}_{\left\{Z^\sigma_s \in (\rho^2, R^2)\right\}}{(Z^\sigma_s)}^{-\frac{3}{2}} \trace\left(X^\sigma_s {X^\sigma_s}^\top \sigma_s \sigma_s^\top\right) \D s\\
		& = \int_0^t \D M^\sigma s - \int_0^t f(X^\sigma_s) \D s + \rho \int_0^t \mathds{1}_{\left\{\abs{X^\sigma_s} \in (\rho, R)\right\}} \abs{X^\sigma_s}^{-3}\trace\left(\left[\abs{X^\sigma_s}^2 I - X^\sigma_s {X^\sigma_s}^\top\right]\sigma_s \sigma_s^\top\right) \D s,
	\end{split}
	\end{equation}
	where $I$ denotes the identity matrix. Noting that the matrix $\abs{x}^2 I - xx^\top$ is positive semi-definite for any $x \in \RR^d$, we see that the final integral in the above equation is always non-negative, and so
	\begin{equation}
		\overline{v}(X^\sigma_t) + \int_0^t f(X^\sigma_s) \D s
	\end{equation}
	is a submartingale for any $\sigma \in \mathcal{U}$.
	
	Now take $\sigma = \sigma^1$ and let $W$ be the first component of the Brownian motion $B$. Then, from the SDE \eqref{eq:squared-radius} for the squared radius process, we see that $Z := Z^{\sigma^1}$ is a one-dimensional squared Bessel process satisfying
	\begin{equation}
		\D Z_t = 2\sqrt{Z_t} \D W_t + \D t.
	\end{equation}
	Substituting this SDE for $Z$ into our calculation \eqref{eq:ex-ito-square}, and defining $X := X^{\sigma^1}$, we find that there is a square-integrable martingale $M$ such that, for any $t > 0$,
	\begin{equation}
	\begin{split}
		w(Z_t) - w(\xi) & = \int_0^t \D M_s + \int_0^t \mathds{1}_{\left\{Z_s \leq \rho^2\right\}}\D s + \rho \int_0^t \mathds{1}_{\left\{Z_s \in (\rho^2, R^2)\right\}}Z_s^{-\frac{1}{2}}\D s - \frac{\rho}{4}\int_0^t \mathds{1}_{\left\{Z_s \in (\rho^2, R^2)\right\}}Z_s^{-\frac{3}{2}} \cdot 4 Z_s \D s\\
		& = \int_0^t \D M_s - \int_0^t f(X_s)\D s.
	\end{split}
	\end{equation}
	Hence $\overline{v}(X_t) + \int_0^t f(X_s) \D s$ is a martingale.
	
	From the above submartingale and martingale properties, we can conclude in a similar manner as in the proof of \Cref{prop:step-value} that $v = \overline v$.
\end{proof}	

For the step cost functions considered above, optimal controls involve tangential and radial motion, as defined in \Cref{def:tangential} and \Cref{def:radial}, respectively. A consequence of \Cref{prop:radial-symmetric-value} in the next section will be that either tangential or radial motion is optimal for any continuous monotone cost function. More generally, we will show that, for a continuous radially symmetric cost function with sufficient regularity, an optimal control is to switch between radial and tangential motion.

\section{Explicit solution for radially symmetric costs}\label{sec:general-rad-symm}
	
In this section, we consider the control problem for more general radially symmetric cost functions. We make the ansatz that the optimal strategy is to switch between two extreme behaviours in the control set, namely the strategies of tangential and radial motion defined in \Cref{def:tangential} and \Cref{def:radial}, respectively. In this way, we reduce the control problem to a one-dimensional optimal switching problem for the radius process. We use the principles of smooth and continuous fit to identify the optimal switching points, and we provide an algorithm to construct a candidate for the value function. We are able to write this function explicitly in \Cref{def:candidate-value}. We refer to the theory of viscosity solutions, which we summarise in \Cref{app:dpp-comparison}, in order to verify that the candidate function is indeed equal to the value function.

We make the following assumptions.

	\begin{ass}\label{ass:rad-symm}
	Suppose that
		\begin{enumerate}
		\item The domain is $D = B_R(0) \subset \RR^d$, for some $R > 0$ and $d \geq 2$;
		\item The cost function $f$ is radially symmetric; i.e.\ $f(x) = \tilde{f}(\abs{x})$, for some function $\tilde{f}: [0, R) \to \RR$;
		\item The boundary cost $g$ is constant --- without loss of generality, we suppose that $g \equiv 0$;
		\item The cost function $f$ is continuous;
		\item There exists $\eta > 0$ such that the cost function $\tilde{f}$ is monotone on the interval $(0, \eta)$;
		\item The one-sided derivative $\tilde{f}^\prime_+(r)$ exists for all $r > 0$ and changes sign only finitely many times;
		\item There exists $\delta > 0$ such that $\tilde{f}$ is continuously differentiable on $(0, \delta)$ and $\lim_{r\to0}r\tilde{f}^\prime(r)=0$.
	\end{enumerate}
	\end{ass}
	
	\begin{remark}
		In \Cref{sec:exploding-cost}, we will relax the fourth condition on continuity and the seventh condition on differentiability.
		
		We rule out the case that the cost function oscillates at the origin by imposing the fifth condition on monotonicity. We will see in the following sections that the fifth and sixth conditions allow us to find an optimal strategy that switches between two regimes finitely many times. We believe that we would still be able to solve the control problem explicitly if we relax the fifth and sixth conditions, but in this case an optimal strategy may not exist. To simplify our exposition, we do not treat this case here.
	\end{remark}
	
	\begin{remark}
		Throughout this paper, we consider control problems on a bounded domain $D$. One could also consider an infinite time horizon control problem, by adding some discount term to the running cost. Specifically, consider
		\begin{equation}
			\tilde v(t, x) := \inf_{\nu \in \mathcal U}\EE^x \left[\int_t^\infty e^{-(s - t)} f(X^\nu_s) \D s\right].
		\end{equation}
		In this case, rather than the HJB equation \eqref{eq:hjb-intro}, $\tilde{v}$ should be a viscosity solution of the PDE
		\begin{equation}
			\partial_t u - \frac{1}{2} \inf_{\sigma \in U} \trace(D^2 u \sigma \sigma^\top) = f - u,
		\end{equation}
		adapting \cite[Theorem 6.11, Theorem 7.4]{touzi_optimal_2013} and our results of \Cref{app:dpp-comparison}. We do not give further details here, nor do we try to identify the value function explicitly for this case.
	\end{remark}
	
	Recall the functions $\sigma^0$ and $\sigma^1$ defined in \Cref{def:tangential} and \Cref{def:radial}, which are associated to tangential and radial motion, respectively. We conjecture that, in the case that $\tilde{f}$ is increasing at the origin, there exists a sequence of points $0 = s_0 < r_1 < s_1 < \dotsc < R$ such that an optimal controlled process $X^{\sigma^\star}$ for the weak value function is defined as follows. Let $X^{\sigma^\star}$ be a weak solution of the SDE $\D X_t = \sigma^\star_t(X) \D B_t$, where
\begin{equation}\label{eq:conj-optimal-control-incr}
	\sigma^\star_t(X^{\sigma^\star}) = 
	\begin{cases}
			\sigma^1\big(X^{\sigma^\star}_0\big),
			& \abs{X^{\sigma^\star}_t} \in [0, r_1),\\
			\sigma^0\big(X^{\sigma^\star}_t\big),
			& \abs{X^{\sigma^\star}_t} \in [r_i, s_i], \quad i \geq 1,
			\vspace{1ex}\\
			\sigma^1\big(X^{\sigma^\star}_{\tau_{s_{i}}}\big),
			& \abs{X^{\sigma^\star}_t} \in (s_{i}, r_{i+1}), \quad i \geq 1,
	\end{cases}
\end{equation}
where, for each $i \geq 1$, we define the hitting time $\tau_{s_{i}} := \inf\left\{t \geq 0 \colon \abs{X^{\sigma^\star}_t} = s_i\right\}$. Note that $t \mapsto \abs{X^{\sigma^\star}_t}$ is deterministically increasing when $\abs{X^{\sigma^\star}_t} \in [r_i, s_i]$, for any $i \geq 1$, by \Cref{lem:squared-radius}. Therefore, if $\abs{X^{\sigma^\star}_0} \geq r_1$, then $\abs{X^{\sigma^\star}_t} \geq r_1$ for all $t \geq 0$. By \Cref{rem:weak-control}, $\sigma^\star$ defines a weak control.

Similarly, if $\tilde{f}$ is decreasing at the origin, we conjecture that there is a sequence of points $0 = r_0 < s_0 < r_1 < \dotsc < R$ such that an optimal controlled process $X^{\sigma^\star}$ for the weak value function is defined as follows. Let $X^{\sigma^\star}$ be a weak solution of the SDE $\D X_t = \sigma^\star_t(X) \D B_t$, where
\begin{equation}\label{eq:conj-optimal-control-decr}
	\sigma^\star_t(X^{\sigma^\star}) = 
	\begin{cases}
			\sigma^0\left(X^{\sigma^\star}_t\right), & \abs{X^{\sigma^\star}_t} \in (0, s_0],
			\vspace{1ex}\\
			\sigma^1\big(X^{\sigma^\star}_{\tau_{s_i}}\big),
			& \abs{X^{\sigma^\star}_t} \in (s_{i - 1}, r_{i}), \quad i \geq 1,
			\vspace{1ex}\\
			\sigma^0\left(X^{\sigma^\star}_t\right),
			& \abs{X^{\sigma^\star}_t} \in [r_i, s_i], \quad i \geq 1.
	\end{cases}
\end{equation}
Once again, by \Cref{rem:weak-control}, $\sigma^\star$ defines a weak control. Note that, in this second case, we do not make any claim about the optimal behaviour at the origin. Since $\sigma^0(0)$ is not defined, we will require some approximation at the origin in this case. We explore this further in \Cref{sec:exploding-cost} where we relax \Cref{ass:rad-symm}.

In either case, we conjecture that, at any time, an optimally controlled process $X^{\sigma^\star}$ should follow either radial motion or tangential motion, depending only on the current radial position of the process. At the switching points $s_i$, there is a reflection phenomenon. We conjecture that the behaviour of the process $X^{\sigma^\star}$ on these boundaries is similar to that of \emph{sticky Brownian motion} --- see Warren \cite{warren1997branching} --- and that $X^{\sigma^\star}$ cannot be defined as a strong solution of an SDE. However, in this paper, we only work with the process $X^{\sigma^\star}$ defined in the weak sense given above and do not investigate these properties further.

 We present a simulated trajectory of such a controlled process for an example with two switching points in \Cref{fig:switching}. In \Cref{prop:radial-symmetric-value}, we will prove that the control $\sigma^\star$ is optimal for the weak value function $v^W$.

\begin{figure}[h]
	\centering
	\includegraphics[width = 0.4\textwidth]{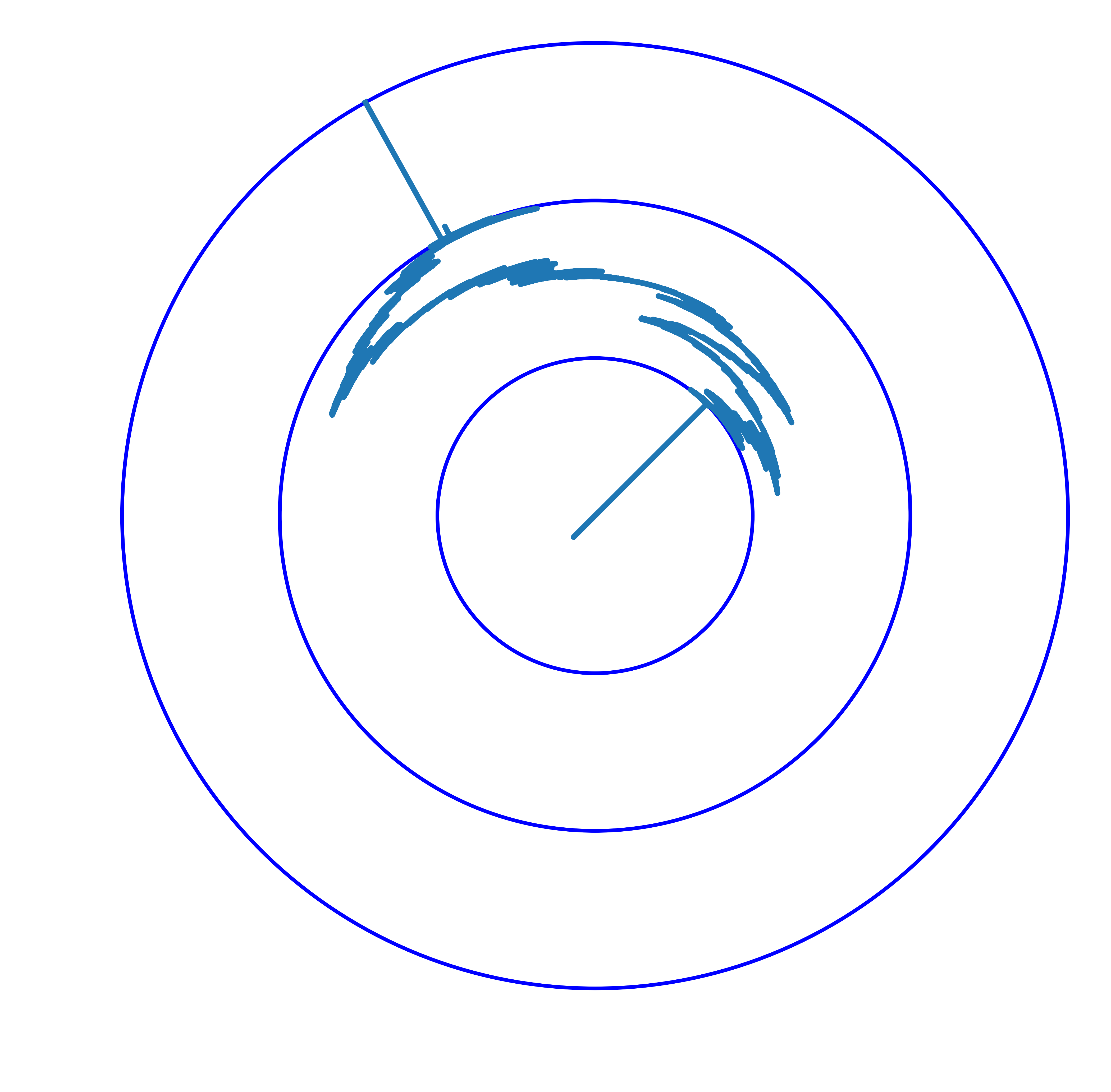}
	\caption[A sample path of an optimal controlled process in a case with two switching points]{A sample path of an optimal controlled process in a case with two switching points}
	\label{fig:switching}
\end{figure}

\subsection{Reduction to a switching problem}\label{sec:switching-problem}

Before beginning to construct a candidate for the value function, we give the following justification for our conjecture that switching between radial motion and tangential motion is optimal. We will work with the radius of the controlled process in this section. We now derive an SDE for the radius process under some simplifying assumptions.

\begin{prop}\label{prop:radius-sde}
	Let $\sigma \in \mathcal{U}$ be of the form  $\sigma_t =
		\begin{bmatrix}
			\overline{\sigma}_t ; & 0 ; & \dotsc ; & 0
		\end{bmatrix}$,
	where $\overline{\sigma}_t \in \RR^d$ with $\abs{\overline{\sigma}_t} = 1$, for $t \geq 0$. Let $x \in D \setminus \{0\}$ and suppose that $X^\sigma$ solves the SDE
	\begin{equation}
		\D X^\sigma_t = \sigma_t \D B_t = \overline{\sigma}_t \D W_t; \quad X^\sigma_0 = x,
	\end{equation}
	where $W$ is the first component of $B$. Set $r_0 := \abs{x}$, let $\varepsilon \in (0, r)$, and define $\tau_\varepsilon := \inf \left\{t > 0 \colon \abs{R^\lambda_t - r_0} = \varepsilon \right\}$. Then there exists a $[0, 1]$-valued process $\lambda$ such that $\abs{X^\sigma_t} = R^\lambda_t$, where $R^\lambda$ solves the SDE
	\begin{equation}
		\D R^\lambda_t = \lambda_t \D W_t + \frac{1 - \lambda_t^2}{2 R^\lambda_t} \D t, \quad t \in [0, \tau_\varepsilon]; \quad R^\lambda_0 = r_0.
	\end{equation}
\end{prop}

\begin{proof}
	Let $t \in [0, \tau_\varepsilon]$, so that $X^\sigma_t \neq 0$. Then we can apply It\^o's formula to find that the radius of $X^\sigma$ satisfies the SDE
	\begin{equation}\label{eq:ito-radius}
		\D \abs{X^\sigma_t} = \abs{X^\sigma_t}^{-1}(X^\sigma_t)^\top \overline{\sigma}_t\D W_t + \frac{1}{2}\abs{X^\sigma_t}^{-3} \trace\left(\left[\abs{X^\sigma_t}^2 I - X^\sigma_t(X^\sigma_t)^\top\right]\overline{\sigma}_t\overline{\sigma}_t^\top\right)\D t.
	\end{equation}
	
	Now let $(X^\sigma_t)^\perp$ denote the vector with norm $\abs{(X^\sigma_t)^\perp} = \abs{X^\sigma_t}$ that is orthogonal to the vector $X^\sigma_t$ and satisfies
	\begin{equation}\label{eq:orthogonal}
		\overline{\sigma}_t = \abs{X^\sigma_t}^{-1}\left(\lambda_t X^\sigma_t + \mu_t (X^\sigma_t)^\perp \right),
	\end{equation}
	for some $\lambda_t, \mu_t \in \RR$. Using the condition $\abs{\overline{\sigma}_t} = 1$, we see that
	\begin{equation}
		1 = \lambda_t^2 + \mu_t^2,
	\end{equation}
	and so $\lambda_t \in [0, 1]$ and $\mu_t = \sqrt{1 - \lambda_t^2}$.
	
	Substituting the expression \eqref{eq:orthogonal} for $\overline{\sigma}$ back into the SDE \eqref{eq:ito-radius} for $\abs{X^\sigma}$, and repeatedly using the identities $(X^\sigma_t)^\top X^\sigma_t = \abs{X^\sigma_t}^2$ and $(X^\sigma_t)^\top (X^\sigma_t)^\perp = 0$, we have
	\begin{equation}
		\D \abs{X^\sigma_t} = \lambda_t\D W_t + \frac{1}{2}\abs{X^\sigma_t}^{-1} \left(1 - \lambda_t^2 \right)\D t.
	\end{equation}
	Therefore, writing $R^\lambda_t = \abs{X^\sigma_t}$, where $\lambda_t$ is defined via \eqref{eq:orthogonal}, we arrive at the desired form of the SDE.
\end{proof}

Now, suppose further that the process $\lambda$ in the proof of \Cref{prop:radius-sde} takes the form
\begin{equation}
	\lambda_t = \lambda(R^\lambda_t), \quad t \geq 0.
\end{equation}
Then we can write down the infinitesimal generator $\mathcal{L}^\lambda$ for the process $R^\lambda$ as
\begin{equation}
	\mathcal{L}^\lambda u(r) = - \frac{1}{2}\lambda^2(r) u^{\prime\prime}(r) - \frac{1 - \lambda^2(r)}{2 r}u^\prime(r),
\end{equation}
for $r \in (r_0 - \varepsilon, r_0 + \varepsilon)$ and any smooth function $u \in C^2((r_0 - \varepsilon, r_0 + \varepsilon), \RR)$.

Consider the following simplification of the control problem. Restrict the control set to contain only those controls that give rise to a process $\lambda$ of the form specified above. Let $v^R:D \to \RR$ be the value function of this simplified problem. By radial symmetry, we can write
\begin{equation}
	v^R(x) = \tilde{v}^R(\abs{x}),
\end{equation}
for some $\tilde{v}^R: [0, R) \to \RR$. Supposing that $\tilde{v}^R$ is twice continuously differentiable, we expect $\tilde{v}^R$ to be a classical solution of a Hamilton-Jacobi-Bellman equation. By the results of Section 3.3 of \cite{touzi_optimal_2013}, $\tilde{v}^R$ solves
\begin{equation}
	\inf_{\lambda}\mathcal{L}^\lambda \tilde{v}^R = \tilde{f},
\end{equation}
in the interval $(r_0 - \varepsilon, r_0 + \varepsilon)$, where the infimum is taken over functions $\lambda:(r_0 - \varepsilon, r_0 + \varepsilon) \to [0, 1]$.

Note that we can rewrite the generator as
\begin{equation}
	\mathcal{L}^\lambda \tilde{v}^R(r) = - \frac{1}{2r} (\tilde{v}^R)^\prime(r) - \frac{r}{2}\lambda^2(r) \left[\frac{1}{r}(\tilde{v}^R)^\prime(r)\right]^\prime.
\end{equation}
Hence, at points $r$ such that $\left[\frac{1}{r}(\tilde{v}^R)^\prime(r)\right]^\prime > 0$, the infimum is attained for $\lambda(r) = 1$, while at points $r$ such that $\left[\frac{1}{r}(\tilde{v}^R)^\prime(r)\right]^\prime < 0$, the infimum is attained for $\lambda(r) = 0$. At a point $r$ such that $\left[\frac{1}{r}(\tilde{v}^R)^\prime(r)\right]^\prime = 0$, the infimum is attained for any value $\lambda(r) \in [0, 1]$.

Returning to the expression \eqref{eq:orthogonal} for $\overline{\sigma}$ in terms of $\lambda$, we see that setting $\lambda_t = 1$ gives $\overline{\sigma}_t = \frac{X^\sigma_t}{\abs{X^\sigma_t}}$, with generator $\mathcal{L}^1$ given by
\begin{equation}\label{eq:generator-radial}
	\mathcal{L}^1 u(r) = -\frac{1}{2}u^{\prime \prime}(r).
\end{equation}
Note that, away from the origin, a controlled process following this control has the same behaviour as radial motion, as defined in \Cref{def:radial}. On the other hand, $\lambda = 0$ corresponds to tangential motion, as defined in \Cref{def:tangential}, with generator
\begin{equation}\label{eq:generator-tangential}
	\mathcal{L}^0 u(r) = -\frac{1}{2r} u^\prime(r).	
\end{equation}
Therefore the above calculations support our claim that the optimal strategy should be to switch between these two behaviour regimes.

\begin{remark}\label{rem:general-control}
	The above heuristic reasoning also applies to control problems of the form \eqref{eq:control-mot} from \Cref{sec:mot}. Therefore we expect a solution $u(x) = \tilde u(|x|)$ of the HJB equation \eqref{eq:hjb-mot} to satisfy
	\begin{equation}
		\inf_{\lambda \in [0, 1]} \left\{\mathcal L^\lambda \tilde u(r) + \lambda^2 r^2 \tilde g(r)\right\} = \tilde h(r),
	\end{equation}
	where $\mathcal L^\lambda$ is the generator of $R^\sigma$, with the relationship between $\sigma$ and $\lambda$ given by \eqref{eq:orthogonal}. Continuing to follow the above reasoning, we expect that switching between radial and tangential motion is also optimal for \eqref{eq:control-mot}.
\end{remark}

We note that, in the above discussion, we restricted the control set and made the strong assumption that the value function is twice continuously differentiable. In order to prove that the behaviour described above is optimal without these restrictions, we will need to refer to the theory of viscosity solutions for HJB equations that we summarise in \Cref{app:dpp-comparison}.

We now identify the conjectured optimal switching points and construct a candidate for the value function, before proving optimality in \Cref{prop:radial-symmetric-value}.

\subsection{Optimal switching points}\label{sec:switching-points}

With the justification of the previous section, we make the ansatz that there exists an optimal strategy of the form described in either \eqref{eq:conj-optimal-control-incr} or \eqref{eq:conj-optimal-control-decr}. We now seek the optimal switching points $r_i$ and $s_i$.

We will find that we require continuous fit and a condition on the first derivative to fix the points $r_i$, and we will need to impose smooth fit and a condition on the second derivative to fix the points $s_i$. It is interesting to note that smooth fit also holds at the points $r_i$, although we do not enforce it.

Under the conjectured optimal behaviour, the value function is of the form
\begin{equation}
	v(x) = \tilde{v}(\abs{x}), \quad x \in D,
\end{equation}
for some $\tilde{v} : [0, R) \to \RR$. To identify the optimal switching points, we will assume that $\tilde{v}$ is differentiable in the interval $(0, R)$ and satisfies the boundary condition $\tilde{v}(R) = g$. Then, for any $r \in (0, R)$, we have
\begin{equation}\label{eq:value-integral}
	\tilde{v}(r) = - \int_r^R \tilde{v}^\prime(s) \D s.
\end{equation}
When we verify our candidate for the value function in \Cref{prop:radial-symmetric-value}, we will show that $v$ is in fact continuously differentiable in $D$ and attains the boundary condition $v = 0$ on $\partial D$.

By definition of the value function, the expected cost associated to any admissible control at some radius $r \in (0, R)$ is greater than the value $\tilde{v}(r)$. Therefore the derivative of such an expected cost at some $r \in (0, R)$ must be less than the derivative of the value function $\tilde{v}^\prime(r)$. We will use this observation to determine the optimal switching points.

Let $\tilde{V}: [0, R] \to \RR$ and define a candidate value function $V: D \to \RR$ by $V(x) = \tilde{V}(\abs{x})$ for $x \in D$. The first step in constructing this function $V$ is to find the optimal switching points, as follows.

Suppose that there exists some $i \geq 1$ such that $0 < s_{i - 1} < r_i < R$. Then we expect that the optimal control switches from tangential motion to radial motion at the point $s_{i - 1}$. In some interval $(s, s_{i - 1})$, we set $\tilde{V} = w_{i - 1}$, where $w_{i - 1}$ solves the ODE
\begin{equation}
	\mathcal{L}^0w_{i - 1}(r) = - 2 r \tilde{f}(r),
\end{equation}
and $\mathcal{L}^0$ is the generator associated to tangential motion that is defined in \eqref{eq:generator-tangential}. This ODE is equivalent to the first order ODE
\begin{equation}
	w_{i - 1}^\prime(r) = - 2 r\tilde{f}(r).
\end{equation}
In the interval $(s_{i - 1}, r_i)$, we set $\tilde{V} = u_i$, where $u_i$ solves the ODE
\begin{equation}
	\mathcal{L}^1 u_i(r) = \tilde{f}(r),
\end{equation}
and $\mathcal{L}^1$ is the generator associated to radial motion that is defined in \eqref{eq:generator-radial}. We can write this ODE as
\begin{equation}\label{eq:switching-ode-radial}
	u_i^{\prime \prime}(r) = - 2 \tilde{f}(r).
\end{equation}
We fix the boundary conditions
\begin{equation}
	u_i(s_{i - 1}) = w_{i - 1}(s_{i - 1}), \qandq {u_i^\prime}_+(s_{i - 1}) = w_{i - 1}^\prime(s_{i - 1}) = - 2 s_{i - 1}\tilde{f}(s_{i - 1}),
\end{equation}
to define $u_i$ uniquely.

Now, in the interval $(r_i, s_i \wedge R)$, we suppose that tangential motion is optimal and set $\tilde{V} = w_i$, where $w_i$ solves the first order ODE
\begin{equation}
	w_i^\prime(r) = - 2 r\tilde{f}(r).
\end{equation}
We then have the following free boundary problem:
\begin{equation}\label{eq:free-boundary-1}
	\begin{cases}
		\tilde{V}^{\prime \prime}(r) = - 2 \tilde{f}(r), & r \in (s_{i - 1}, r_i),\\
		\tilde{V}^\prime(r) = - 2 r \tilde{f}(r), & r \in (r_i, s_i \wedge R),\\
		\tilde{V}(r_i +) = \tilde{V}(r_i -),
	\end{cases}
\end{equation}
where the point $r_i$ is to be found. Note that we require the continuous fit condition at $r_i$ in order to solve the first order ODE in $(r_i, s_i \wedge R)$.

As noted above, we determine the switching point by comparing the derivatives of $u_i$ and $w_i$. The point $r_i$ should be the first point at which $w_i^\prime(r) = - 2 r \tilde{f}(r)$ is greater than the first derivative of $u_i$. Therefore we define $r_i$ by
\begin{equation}
	r_i := \inf\left\{r > s_{i - 1} \colon s_{i - 1}\tilde{f}(s_{i - 1}) + \int_{s_{i - 1}}^r \tilde{f}(s) \D s > r \tilde{f}(r)\right\}.
\end{equation}
That is the first point after $s_{i - 1}$ at which the running average of the cost function becomes greater than its current value. Note that this point cannot be in a region where $\tilde{f}$ is increasing and so $r_i$ is greater than or equal to the first point of decrease of $\tilde{f}$ after $s_{i - 1}$.

In \Cref{subfig:switching-example-first-derivatives}, we show an example of choosing the switching point $r_1$ by comparing derivatives. We see in \Cref{subfig:switching-example-cost} that, for this example, the switching point $r_1$ is strictly greater than the turning point at which the cost function starts to decrease. Also note that, although we have only imposed continuous fit at the point $r_1$, we can see in \Cref{subfig:switching-example-first-derivatives} that the derivatives are equal at $r_1$. For any continuous cost function, this smooth fit property arises in the same way; we will discuss this in detail in \Cref{sec:smooth-fit}.

\begin{figure}[h]
	\centering
	\begin{subfigure}{0.6\textwidth}
		\centering
		\includegraphics[width = \textwidth]{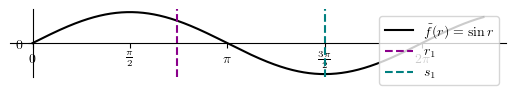}
		\caption{Radial part of the cost function $\tilde{f}(r) = \sin r$}
		\label{subfig:switching-example-cost}
	\end{subfigure}\\
	\begin{subfigure}{0.6\textwidth}
		\centering
		\includegraphics[width = \textwidth]{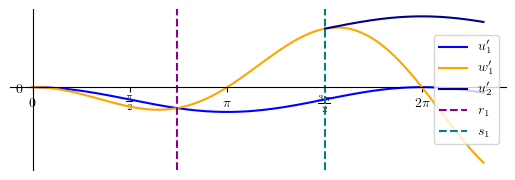}
		\caption{First derivatives of the expected costs $u_1$, $w_1$, and $u_2$}
		\label{subfig:switching-example-first-derivatives}
	\end{subfigure}\\
	\begin{subfigure}{0.6\textwidth}
		\centering
		\includegraphics[width = \textwidth]{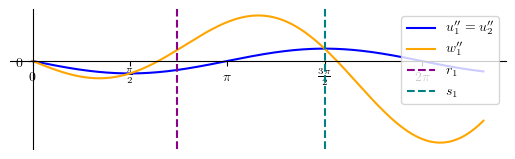}	
		\caption{Second derivatives of the expected costs $u_1$, $w_1$, and $u_2$}
		\label{subfig:switching-example-second-derivatives}
	\end{subfigure}
	\caption[Switching points for the cost function $f(x) = \sin(\abs{x})$]{The first two switching points $r_1$, $s_1$ are shown for the cost function $f(x) = \sin \abs{x}$. The switching point $r_1$ is the first point at which $w_1^\prime(r) = - 2 r \tilde{f}(r)$ exceeds $u_1^\prime$, where $u_1$ solves $u_1^{\prime \prime}(r) = - 2 \tilde{f}(r)$, with ${u_1}^\prime_+(0) = 0$, as shown in (\textsc{b}). The switching point $s_1$ is the first point after $r_1$ at which $u_2^{\prime \prime} = - 2 \tilde{f}$ exceeds $w_1^{\prime \prime}$, as shown in (\textsc{c}). Fixing $u_2^\prime(s_1) = w_1^\prime(s_1)$, we see in (\textsc{b}) that $s_1$ is chosen such that $u_2^\prime$ remains greater than $w_1^\prime$ over an interval of positive length.}
	\label{fig:switching-example}
\end{figure}

Let us now suppose that $s_i < R$. We suppose that, in the interval $(s_i, r_{i + 1} \wedge R)$, radial motion is once again optimal, and we set $\tilde{V} = u_{i + 1}$, where $u_{i + 1}$ solves the second order ODE
\begin{equation}
	u_{i + 1}^{\prime \prime}(r) = - 2 \tilde{f}(r).	
\end{equation}
Then we have a second free boundary problem
\begin{equation}\label{eq:free-boundary-2}
	\begin{cases}
		\tilde{V}^\prime(r) = - 2 r \tilde{f}(r), & r \in (r_i, s_i),\\
		\tilde{V}^{\prime \prime}(r) = - 2 \tilde{f}(r), & r \in (s_i, r_{i + 1} \wedge R),\\
		\tilde{V}(s_i +) = \tilde{V}(s_i -),\\
		\tilde{V}^\prime_+(s_i) = \tilde{V}^\prime_-(s_i),
	\end{cases}
\end{equation} 
where the point $s_i$ is to be found. Here we require both smooth fit and continuous fit at the point $s_i$ in order to solve the second order ODE in the interval $(s_i, r_{i + 1} \wedge R)$.

Having imposed the smooth fit condition $\tilde{V}^\prime_+(s_i) = \tilde{V}^\prime_-(s_i)$, the first derivatives of solutions of $w_{i}^\prime(r) = - 2 r \tilde{f}(r)$ and $u_{i + 1}^{\prime \prime}(r) = - 2 \tilde{f}(r)$ are equal for any choice of $s_i$. In order to fix the point $s_i$, we require a second order condition. Recall from \Cref{ass:rad-symm} that we assume that the right derivative of $\tilde{f}$ exists everywhere. This allows us to define $s_i$ to be the first point at which $u_{i + 1}^{\prime \prime}(r) = - 2 \tilde{f}(r)$ is greater than the one-sided second derivative from the right of the solution of $w_{i}^\prime(r) = - 2 r \tilde{f}(r)$. Thus there is an interval of positive length on which the first derivatives are in this same order.

We can calculate the one-sided second derivative from the right of $w_i$ as
\begin{equation}
	{w_i^{\prime \prime}}_+(r) = - 2 \tilde{f}(r) - 2 r \tilde{f}^{\prime}_+(r).
\end{equation}
This leads us to define $s_i$ by
\begin{equation}
	s_i := \inf \left\{s > r_i \colon \tilde{f}^\prime_+(s) > 0\right\}.
\end{equation}
In this case, the switching point is exactly the turning point at which $\tilde{f}$ starts to increase. For the example in \Cref{fig:switching-example}, we can see that the switching point $s_1$ does indeed coincide with this turning point. \Cref{subfig:switching-example-second-derivatives} shows how this switching point is chosen by comparing second derivatives, and \Cref{subfig:switching-example-first-derivatives} shows that the first derivatives at this point have the desired properties.

Note that the sixth condition of \Cref{ass:rad-symm} implies that there are finitely many switching points $s_i$ and thus finitely many points $r_i$. Taking the above definitions of $r_i$ and $s_j$ for all values of $i, j$ such that $r_i, s_j < R$, we now solve the ODEs in \eqref{eq:free-boundary-1} and \eqref{eq:free-boundary-2} to construct a candidate for the value function.

\subsection{Construction of the value function}\label{sec:construction}

In this section we construct the candidate function $V$, which we will go on to prove is equal to the value function. We break the construction down into two cases depending on the behaviour of the cost function at the origin, and then into two further sub-cases depending on the behaviour of the cost function at the boundary of the domain.

\subsubsection{Case I: Increasing cost at the origin}

Suppose first that $\tilde{f}$ is increasing on the interval $(0, \eta)$.  We summarise the definition of switching points and the construction of the candidate value function in this case in \Cref{alg:solution-incr}.

\begin{algorithm}[h]
\caption[Construction of the value function I]{Construction of the value function in Case I}\label{alg:solution-incr}
\begin{algorithmic}
	\State{Define $s_0 = 0$.}
	\State{Solve $u_1^{\prime\prime}(r) = - 2 \tilde{f}(r)$, with ${u_1}^\prime_+(0) = 0$, $u_1(0) = \alpha$, for some $\alpha \in \RR$.}
	\State{Define $r_1 := \inf\left\{r > 0 \colon \int_0^r \tilde{f}(s) \D s > r \tilde{f}(r)\right\}$.}
	\State{Set $\tilde{V} = u_1$ on $(0, r_1 \wedge R]$.}
	\If{$r_1 < R$}
		\For{$i \geq 1$}
			\State{Solve $w_i^\prime(r) = -2 r\tilde{f}(r)$, with $w_i(r_i) = u_i(r_i).$}
			\State{Define $s_{i}:=\inf\left\{r > r_i \colon \tilde{f}^\prime_+(s) > 0\right\}$.}
			\State{Set $\tilde{V} = w_i$ on $(r_i, s_i \wedge R]$.}
			\If{$s_i \geq R$}
				\Break
			\EndIf
			\State{Solve $u_{i+1}^{\prime\prime}(r) = -2\tilde{f}(r)$, with $u_{i+1}^\prime(s_{i}+) = -2s_{i}\tilde{f}(s_{i})$ and \newline $u_{i+1}(s_{i})= w_i(s_{i}).$}
			\State{Define $r_{i+1} := \inf\left\{r > s_i \colon s_i \tilde{f}(s_i) + \int_{s_i}^r \tilde{f}(s) \D s > r \tilde{f}(r)\right\}$.}
			\State{Set $\tilde{V} = u_{i + 1}$ on $(s_i , r_{i + 1} \wedge R]$.}
			\If{$r_{i+1} \geq R$}
				\Break
			\EndIf
		\EndFor
	\EndIf
	\State{Fix $\alpha$ such that $\tilde{V}(R) = 0$.}
\end{algorithmic}	
\end{algorithm}

Fix $s_0 = 0$. Since we expect the optimal control to enforce radial motion in the ball $B_\eta(0)$, we solve the second order ODE
\begin{equation}
	u_1^{\prime \prime}(r) = - 2 \tilde{f}(r), \quad r \in (0, R).
\end{equation}
We require two boundary conditions in order to uniquely define the solution $u_1$. We impose the boundary condition ${u_1}^\prime_+(0) = 0$ for the following reasons.

First, from the discussion in the previous section, we recall that we will define the first switching point to be
\begin{equation}
	r_1 = \inf\left\{r > 0 \colon u_1^\prime(r) < - 2 r \tilde{f}(r)\right\},
\end{equation}
since we are seeking to maximise the derivative of the candidate value function. Therefore, for $r \in (0, r_1)$, we must have $u_1^\prime(r) \geq - 2 r \tilde{f}(r)$ and, in particular
\begin{equation}
	{u_1}^\prime_+(0) = \lim_{r \downarrow 0} u_1^\prime(r) \geq - 2 \lim_{r \downarrow 0}r \tilde{f}(r) = 0.
\end{equation}
To get the opposite inequality, fix $\delta \in (0, \eta)$ and $r \in (0, \delta)$ and apply It\^o's formula to $u_1(\delta) = u_1\big(\big\lvert X^{\sigma^1}_{\tau_\delta}\big\rvert\big)$ to see that
\begin{equation}
\begin{split}
	u_1(\delta) - u_1(r) = \frac{1}{2}\EE^r\left[\int_0^{\tau_\delta} u_1^{\prime \prime}\big(\big\lvert X^{\sigma^1}_{\tau_\delta}\big\rvert\big) \D s\right] = - \EE^r \left[\int_0^{\tau_\delta} \tilde{f}\big(\big\lvert X^{\sigma^1}_{\tau_\delta}\big\rvert\big) \D s\right].
\end{split}
\end{equation}
Then, applying dominated convergence to take the limit as $r \downarrow 0$, and using the fact that $\tilde{f}$ is increasing, we have that
\begin{equation}
\begin{split}
	\lim_{r \downarrow 0}\frac{1}{\delta}\left(u_1(\delta) - u_1(r)\right) & = - \frac{1}{\delta}\EE^0 \left[\int_0^{\tau_\delta}\tilde{f}\big(\big\lvert X^{\sigma^1}_{\tau_\delta}\big\rvert\big) \D s\right]\\
	& \leq - \frac{1}{\delta} \tilde{f}(0) \EE^0[\tau_\delta] = - \delta \tilde{f}(0).
\end{split}
\end{equation}
Hence
\begin{equation}
	0 \leq {u_1}^\prime_+(0) \leq - \lim_{\delta \downarrow 0}\delta \tilde{f}(0) = 0.
\end{equation}

As well as imposing the above condition on the first derivative, we also fix an arbitrary value $u_1(0) = \alpha \in \RR$. Having constructed the candidate value function, up to this arbitrary constant, on the whole domain, we will use the external boundary condition $\tilde{V}(R) = 0$ to determine the value of $\alpha$. We now have
\begin{equation}
	u_1(r) = \alpha - 2 \int_0^r \int_0^s \tilde{f}(t) \D t \D s.
\end{equation}
Define
\begin{equation}
	r_1 := \inf\left\{r > 0 \colon \int_0^r \tilde{f}(s) \D s > r\tilde{f}(r)\right\},
\end{equation}
and set $\tilde{V}(r) = u_1(r)$ for $r \in (0, r_1 \wedge R]$.

If $r_1 < R$, we then expect the optimal control to switch to enforcing tangential motion. Therefore we solve the first order ODE
\begin{equation}
	w_1^\prime(r) = - 2 r \tilde{f}(r), \quad r \in (r_1, R).
\end{equation}
In order to uniquely define the solution $w_1$, we impose the continuous fit condition $w_1(r_1) = \tilde{V}(r_1)$. Then we have
\begin{equation}
\begin{split}
	w_1(r) & = \tilde{V}(r_1) - 2 \int_{r_1}^r s \tilde{f}(s) \D s\\
	& = \alpha - 2 \int_{r_1}^r s \tilde{f}(s) \D s - 2 \int_0^{r_1}\int_0^s \tilde{f}(t) \D t \D s.
\end{split}
\end{equation}
Now define
\begin{equation}
	s_1 := \inf\left\{r > r_1 \colon \tilde{f}^\prime_+(r) > 0\right\},
\end{equation}
and set $\tilde{V}(r) = w_1(r)$ for $r \in (r_1, s_1 \wedge R]$.

If $s_1 < R$, then we expect the optimal control to switch back to enforcing radial motion, and so we solve the second order ODE
\begin{equation}
	u_2^{\prime \prime}(r) = -2 \tilde{f}(r), \quad r \in (s_1, R).
\end{equation}
At this point, we impose both the continuous fit condition $u_2(s_1) = \tilde{V}(s_1)$ and the smooth fit condition ${u_2}^\prime_+(s_1) = \tilde{V}^\prime(s_1)$ in order to uniquely define $u_2$. We then find that
\begin{equation}
	u_2^\prime(r) = \tilde{V}^\prime(s_1) - 2 \int_{s_1}^r \tilde{f}(s) \D s,
\end{equation}
and so
\begin{equation}
\begin{split}
	u_2(r) & = \tilde{V}(s_1) + (r - s_1) \tilde{V}^\prime(s_1) - 2 \int_{s_1}^r \int_{s_1}^s \tilde{f}(t) \D t \D s\\
	& = \alpha - 2 \int_{s_1}^r \int_{s_1}^s \tilde{f}(t) \D t \D s - 2 \int_0^{r_1}\int_0^s \tilde{f}(t) \D t \D s - 2 \int_{r_1}^{s_1} s \tilde{f}(s) \D s - 2(r_1 - s_1)r_1 \tilde{f}(r_1).
\end{split}
\end{equation}
Defining
\begin{equation}
	r_2 := \inf \left\{r > s_1 \colon s_1 \tilde{f}(s_1) + \int_{s_1}^r \tilde{f}(s) \D s > r \tilde{f}(r)\right\},
\end{equation}
we set $\tilde{V}(r) = u_2(r)$ for $r \in (s_1, r_2 \wedge R]$.

We continue in this way until reaching the boundary of the domain. For each $i \ge 1$, define recursively the switching points
\begin{equation}\label{eq:caseI-switching}
	\begin{split}
		s_{i} & :=\inf\left\{r > r_i \colon \tilde{f}^\prime_+(s) > 0\right\},\\
		r_{i+1} & := \inf\left\{r > s_i \colon s_i \tilde{f}(s_i) + \int_{s_i}^r \tilde{f}(s) \D s > r \tilde{f}(r)\right\},
	\end{split}
\end{equation}
and set
\begin{equation}
	\tilde{V}(r) =
	\begin{cases}
		u_{i}(r), & r \in (s_{i - 1}, r_i \wedge R],\\
		w_i(r), & r \in (r_i, s_i \wedge R].
	\end{cases}
\end{equation}

In order to determine the value of $\tilde{V}(0) = \alpha$, we use the boundary condition on $\partial D$. Let $K \in \NN$ be such that $R \in (s_{K - 1}, s_K]$. Suppose first that $R \in (s_{K - 1}, r_K]$. Then we expect radial motion to be optimal close to the boundary of the domain, and we have $\tilde{V}(r) = u_{K - 1}(r)$ for $r \in (s_{K}, R]$. Imposing the boundary condition $V(x) = 0$ for $x \in \partial D$, we have $u_{K}(R) = 0$. Now suppose that $R \in (r_K, s_K]$, so that we expect tangential motion to be optimal close to the boundary of the domain. Then we have $\tilde{V}(r) = w_K(r)$ for $r \in (r_K, R]$. Now imposing the boundary condition $V(x) = 0$ for $x \in \partial D$ gives us $w_K(R) = 0$. In either case, the value of $\alpha$ is then specified uniquely.

We state the candidate value function explicitly in \Cref{def:candidate-value} below.

\subsubsection{Case II: Decreasing cost at the origin}

We now turn to the second case where $\tilde{f}$ is decreasing on the interval $(0, \eta)$. We summarise the definition of a sequence of switching points and the construction of the candidate value function in this case in \Cref{alg:solution-decr}.

\begin{algorithm}[h]
\caption[Construction of the value function II]{Construction of the value function in Case II}\label{alg:solution-decr}
\begin{algorithmic}
	\State{Define $r_0 = 0$.}
	\State{Solve $w_0^\prime(r) = - 2 r \tilde{f}(r)$, with $w_0(r) = \alpha$, for some $\alpha \in \RR$.}
	\State{Define $s_0 := \inf\left\{r > 0 \colon \tilde{f}^\prime_+(r) > 0\right\}$}.
	\State{Set $\tilde{V} = w_0$ on $(0, s_0 \wedge R]$.}
	\If{$s_0 < R$}
		\For{$i \geq 0$}
			\State{Solve $u_{i+1}^{\prime\prime}(r) = - 2\tilde{f}(r)$, with $u_{i+1}^\prime(s_{i}+) = -2s_{i}\tilde{f}(s_{i})$ and \newline $u_{i+1}(s_{i})= w_i(s_{i}).$}
			\State{Define $r_{i+1} := \inf\left\{r > s_i \colon s_i \tilde{f}(s_i) + \int_{s_i}^r \tilde{f}(s) \D s > r \tilde{f}(r)\right\}$.}
			\State{Set $\tilde{V} = u_{i + 1}$ on $(s_i, r_{i + 1} \wedge R]$.}
			\If{$r_{i+1} \geq R$}
				\Break
			\EndIf
			\State{Solve $w_{i+1}^\prime(r) = -2 r\tilde{f}(r)$, with $w_{i+1}(r_{i+1}) = u_{i+1}(r_{i+1}).$}
			\State{Define $s_{i+1}:=\inf\left\{r > r_{i + 1} \colon \tilde{f}^\prime_+(r) > 0\right\}$.}
			\State{Set $\tilde{V}(R) = g$ on $(r_{i + 1}, s_{i + 1} \wedge R]$.}
			\If{$s_{i+1} \geq R$}
				\Break
			\EndIf
		\EndFor
	\EndIf
	\State{Fix $\alpha$ such that $\tilde{V}(R) = 0$.}
\end{algorithmic}	
\end{algorithm}

We expect the optimal control to enforce tangential motion in $B_\eta(0) \setminus B_\varepsilon(0)$, for any $\varepsilon \in (0, \eta)$. As we will see in \Cref{sec:exploding-cost}, it will be possible to define a control at the origin whose cost approximates the cost associated to tangential motion. Without further justification here, we fix $r_0 = 0$ and seek the solution $w_0$ to the first order ODE
\begin{equation}
	w_0^\prime(r) = - 2 r \tilde{f}(r), \quad r \in (0, R).
\end{equation}

Note that this ODE fixes the first derivative and, in particular, ${w_1}^\prime_+(0) = 0$. In order to uniquely define $w_1$, we need to impose one boundary condition. As in the previous section, we will fix an arbitrary value $w_1(0) = \alpha \in \RR$, and we will determine the value of $\alpha$ from the external boundary condition $\tilde{V}(R) = 0$, once we have constructed the candidate value function on the whole domain.

The construction of the value function proceeds in the same way as in Case I, here defining switching points by $r_0 = 0$ and, for each $i \ge 0$,
\begin{equation}\label{eq:caseII-switching}
	\begin{split}
		r_{i+1} & := \inf\left\{r > s_i \colon s_i \tilde{f}(s_i) + \int_{s_i}^r \tilde{f}(s) \D s > r \tilde{f}(r)\right\},\\
		s_{i+1} & :=\inf\left\{r > r_{i + 1} \colon \tilde{f}^\prime_+(r) > 0\right\}.
	\end{split}
\end{equation}
 We omit the remaining details in this case. We state the candidate value function in both cases in the following \Cref{def:candidate-value}.

\begin{defn}[Candidate value function]\label{def:candidate-value}
	Let the cost functions $f$ and $g$ be as in \Cref{ass:rad-symm}. For $k \in \NN$ and $i = 0, \dotso, k$, define the constant
	\begin{equation}
		\mathfrak{F}^k_i := 2\sum_{j = i + 1}^k \left[(r_j - s_{j - 1}) s_{j - 1} \tilde{f}(s_{j - 1}) + \int_{s_{j - 1}}^{r_j} \int_{s_{j -1}}^s \tilde{f}(t) \D t \D s + \int_{r_j}^{s_j} s \tilde{f}(s) \D s\right].
	\end{equation}
	Then we define the candidate value function $V: D \to \RR$ as follows.
	
	\paragraph{\emph{Case I}} If $\tilde{f}$ is increasing in $(0, \eta)$, then set $s_0 = 0$, define $(s_i, r_i)$ by \eqref{eq:caseI-switching}, and let $K \in \NN$ be such that $R \in (s_{K - 1}, s_K]$. For $x \in D$, define
	\begin{equation}
	\begin{split}
		V(x) & = - 2 \int_{R \vee r_K}^{s_K} s \tilde{f}(s) \D s - 2(r_K - R \wedge r_K) s_{K - 1} \tilde{f}(s_{K - 1}) - 2 \int_{R \wedge r_K}^{r_K} \int_{s_{K - 1}}^s \tilde{f}(t) \D t \D s\\
		& \quad + 2 \sum_{i = 1}^K \ind{(s_{i - 1}, s_i]}(\abs{x}) \left[(r_i - \abs{x} \wedge r_i) s_{i - 1} \tilde{f}(s_{\i -1 }) + \int_{\abs{x} \wedge r_i}^{r_i} \int_{s_{i - 1}}^s \tilde{f}(t) \D t \D s + \int_{\abs{x} \vee r_i}^{s_i} s \tilde{f}(s) \D s + \mathfrak{F}^K_i\right].
	\end{split}
	\end{equation}
	
	\paragraph{\emph{Case II}} If $\tilde{f}$ is decreasing in $(0, \eta)$, then set $r_0 = 0$, define $(r_i, s_i)$ by \eqref{eq:caseII-switching}, and let $L \in \NN$ be such that $R \in (r_L, r_{L + 1}]$. For $x \in D$, define
	\begin{equation}
	\begin{split}
		V(x) & = - 2 \int_{R \wedge s_L}^{s_L} s \tilde{f}(s) \D s + 2(R \vee s_L - s_L) s_L \tilde{f}(s_L) + 2 \int_{s_L}^{R \vee s_L} \int_{s_L}^s \tilde{f}(t) \D t \D s\\
		& \quad + 2 \sum_{i = 0}^L \ind{(r_i, r_{i + 1}]}(\abs{x}) \left[\int_{\abs{x} \wedge s_i}^{s_i} s \tilde{f}(s) \D s  - (\abs{x} \vee s_i - s_i) s_i \tilde{f}(s_i) - \int_{s_i}^{\abs{x} \vee s_i} \int_{s_i}^s \tilde{f}(t) \D t \D s + \mathfrak{F}^L_i \right].
	\end{split}
	\end{equation}
\end{defn}

Before turning to the rigorous proof of optimality in \Cref{sec:proof-optimality}, we make the following remarks on the candidate value function.

\begin{remark}
	We observe that we can recover the the value functions of \Cref{ex:step-decr} and \Cref{ex:step-incr} from the expressions for $V$ given in \Cref{def:candidate-value}. First consider \Cref{ex:step-decr} where $\tilde{f}(r) = -\mathds{1}_{(\rho, R)}(r)$. This example falls into Case II of \Cref{def:candidate-value}. Since $\tilde{f}$ is always decreasing, there are no switching points. Therefore the value function reduces to $V(x) = 2 \int_{|x|}^R s \tilde f(s) \D s$, $x \in D$, as given in \Cref{prop:step-value}.
	Now take $\tilde f(r) = -\mathds{1}_{[0, \rho)}(r)$, as in \Cref{ex:step-incr}. This example falls into Case I, and once again there are no switching points because $\tilde f$ is monotone. Hence the value function reduces to $V(x) = 2 \int_{|x|}^R \int_0^s \tilde{f}(t) \D t \D s$, $x \in D$, as in \Cref{prop:step-incr}.

	Note, however, that \Cref{ex:step-decr} and \Cref{ex:step-incr} are not covered by \Cref{prop:radial-symmetric-value} since the cost functions $f$ are not continuous.	
\end{remark}

\begin{remark}\label{sec:smooth-fit}
  In the preceding construction, the smooth fit condition is required to fix the switching points $s_i$. It is notable, however, that we do not need to impose smooth fit to uniquely identify the points $r_i$, and it is thus surprising to us that the smooth fit condition is nevertheless satisfied at these switching points. 

  A heuristic argument for the smooth fit condition in diffusion problems typically comes from two competing criteria. The heuristic we look to exploit is that we expect $\tilde{v}(\abs{X_t^\nu}) + \int_0^t \tilde{f}(\abs{X_s^\nu}) \D s$ to be a submartingale for all admissible strategies, and a martingale for the optimal strategy. Now suppose that the function $\tilde{v}$ displays a discontinuity in its first derivative at the boundary of two types of behaviour, at $|X| = r$ say. Then, applying the It\^o-Tanaka formula, we expect a local time term of the form $(\tilde{v}'_+(r)-\tilde{v}'_-(r)) \D L_t^r$ to appear in $\D \tilde{v}(|X_t|)$. Since there is no cancelling term in the time integral component, it follows immediately that if $\tilde{v}'_+(r) < \tilde{v}'_-(r)$, then the process will be a supermartingale for any strategy which has positive local time at $r$. Consequently, we expect $\tilde{v}'_+(r) \ge \tilde{v}'_-(r)$ for all $r \in (0,R)$. Moreover, \emph{if} the optimal strategy accrues local time at $r$, then a similar argument forces $\tilde{v}'_+(r) = \tilde{v}'_+(r)$, and we deduce the smooth fit condition.

  However, this justification breaks down at the switching points $r_i$ described above. Under the conjectured optimal strategy, there is no local time accrued at such switching points, and therefore the heuristic justification for smooth fit fails. However, from our construction of the value function, we see that the smooth fit condition still holds!

  In general, we have no heuristic justification for such a condition. We note that, in optimal stopping problems for L\'evy processes, or more generally jump diffusions, it is common to observe continuous fit conditions where there is no diffusive boundary behaviour \cite{kyprianou_principles_2007}. This is comparable to the behaviour that we observe under the optimal strategy at points $r_i$, which arises from the notable fact that our control process can produce both diffusive and non-diffusive behaviour at boundary points. We are unaware of similar behaviour occurring in a diffusive setting, and we leave further study of this behaviour for future work.
\end{remark}

\subsection{Proof of optimality}	\label{sec:proof-optimality}

We now turn to the proof that the candidate function that we have constructed is indeed the value function.

\begin{prop}\label{prop:radial-symmetric-value}
	Under \Cref{ass:rad-symm}, the value function $v$ is continuously differentiable and takes the form $v = V$, where $V$ is defined in \Cref{def:candidate-value}.
	
	Moreover, there exists an optimal weak control $\sigma^\star$ in the following cases. If $\tilde{f}$ is increasing in $(0, \eta)$, then the weak control $\sigma^\star$ defined via \eqref{eq:conj-optimal-control-incr} is optimal. If $\tilde{f}$ is decreasing in $(0, \eta)$ and the initial condition is $x \in D \setminus \{0\}$, then the weak control $\sigma^\star$ defined via \eqref{eq:conj-optimal-control-decr} is optimal.
\end{prop}

In order to prove this result, we refer to the theory of viscosity solutions for HJB equations that we summarise in \Cref{app:dpp-comparison}. The main result that we require is \Cref{thm:unique-viscosity}, which gives a viscosity solution characterisation of the value function.

In this section, we will prove that the candidate function $V$ is a viscosity solution of the HJB equation
\begin{equation}\label{eq:hjb-2}
	- \frac{1}{2}\inf_{\sigma \in U}\trace\left(D^2V(x)\sigma\sigma^\top\right) = f(x), \quad x \in D,
\end{equation}
with boundary condition $V = 0$ on $\partial D$. We then appeal to Theorem \ref{thm:unique-viscosity} to see that the value function $v$ is a viscosity solution of the same boundary value problem and, moreover, such a solution is unique. From this, we conclude that the function $V$ is equal to the value function $v$.

We first show that $V$ is a classical solution of \eqref{eq:hjb-2} in the regions where we expect radial motion to be optimal.

\begin{lemma}\label{lem:ex-incr-visc}
	For each $i \geq 1$, define $u_i: (s_{i-1}, r_i \wedge R] \to \RR$ by
	\begin{equation}
 		u_i(r) = 2 \int_{r}^{r_i}\int_{s_{i-1}}^s \tilde{f}(t) \D t \D s + 2(r_i - r)s_{i - 1}\tilde{f}(s_{i - 1}) + C^u_i, \quad r \in (s_{i-1}, r_i \wedge R],
 	\end{equation}
 	for an arbitrary constant $C^u_i$, and define the set
 	\begin{equation}
 		D_i := \{x \in D \colon \abs{x} \in (s_{i-1}, r_i \wedge R)\}.
 	\end{equation}
 	Then $U_i: D_i \to \RR$, defined by $U_i(x) = u_i(\abs{x})$, is a classical solution of the PDE \eqref{eq:hjb-2} in the region $D_i$.
\end{lemma}

\begin{proof}
	Fix $i \geq 1$ and let $x \in D_i$. Observe that, by definition of $r_{i}$, 
	\begin{equation}\label{eq:ri-ineq}
		u_i^\prime (\abs{x}) \geq - 2 \abs{x}\tilde{f}(\abs{x}).
	\end{equation}
	We have that $U_i$ is twice continuously differentiable at $x$ and
	\begin{equation}
		D^2U_i(x) = \abs{x}^{-3}\left[\abs{x} u_i^{\prime\prime}(\abs{x}) - u_i^\prime(\abs{x})\right]xx^\top + \abs{x}^{-1}u_i^\prime(\abs{x})I.
	\end{equation}
	Substituting in $u_i^{\prime\prime}(\abs{x}) = -2\tilde{f}(\abs{x})$ and rearranging gives
	\begin{equation}
	\begin{split}
		D^2U_i(x) & = - \abs{x}^{-3}\left[2 \abs{x}\tilde{f}(\abs{x}) + u_i^\prime(\abs{x})\right]xx^\top + \abs{x}^{-1}u_i^\prime(\abs{x})I\\
		& = - 2\tilde{f}(\abs{x})I + \abs{x}^{-3}\left[2\abs{x}\tilde{f}(\abs{x}) + u_i^\prime(\abs{x})\right]\left[\abs{x}^2I - xx^\top\right].
	\end{split}
	\end{equation}
	Hence, for any $\sigma \in U$,
	\begin{equation}
	\begin{split}
		\trace\left(D^2U_i(x) \sigma\sigma^\top\right) & = - 2 \tilde{f}(\abs{x})\trace(\sigma\sigma^\top) + \abs{x}^{-3}\left[2\abs{x}\tilde{f}(\abs{x}) + u_i^\prime(\abs{x})\right]\trace\left(\left[\abs{x}^2 I - xx^\top\right] \sigma\sigma^\top\right).
	\end{split}
	\end{equation}
	Noting that $\abs{x}^2I - xx^\top$ is positive semi-definite, and using \eqref{eq:ri-ineq}, we have
	\begin{equation}
		\trace\left(D^2U_i(x)\sigma\sigma^\top\right) \geq -2 \tilde{f}(\abs{x})\trace(\sigma\sigma^\top) = - 2 f(x),
	\end{equation}
	for any $\sigma \in U$.
	
	Taking $\sigma = \sigma^1(x)$, where $\sigma^1:D \to \RR$ is the function defined in \eqref{eq:radial-function}, we see that
	\begin{equation}
		\trace\left(\left[\abs{x}^2 I - xx^\top\right] \sigma^1(x){\sigma^1(x)}^\top\right) = 0,
	\end{equation}
	and so
	\begin{equation}
		\trace\left(D^2U_i(x)\sigma^1(x){\sigma^1(x)}^\top\right) = - 2 f(x).
	\end{equation}
	Hence $U_i$ is a classical solution of the PDE \eqref{eq:hjb-2} in the the region $D_i$.
\end{proof}

We next show that $V$ is a viscosity solution of \eqref{eq:hjb-2} in the regions where we expect tangential motion to be optimal.

\begin{lemma}\label{lem:ex-decr-visc}
	For each $i \geq 0$, define $w_i: (r_i, s_i \wedge R] \to \RR$ by
	\begin{equation}
 		w_i(r) = 2 \int_{r}^{s_i} s \tilde{f}(s) \D s + C^w_i, \quad r \in (r_i, s_i \wedge R],
 	\end{equation}
 	for an arbitrary constant $C^w_i$, and define the set
 	\begin{equation}
		\overline{D}_i := \{x \in D \colon \abs{x} \in (r_i, s_i \wedge R)\}.
	\end{equation}
	Then $W_i: \overline{D}_i \to \RR$, defined by $W_i(x) = w_i(\abs{x})$, is a viscosity solution of the PDE \eqref{eq:hjb-2} in the region $\overline{D}_i$.
\end{lemma}

Note that $w_i$ is twice continuously differentiable if and only if $\tilde{f}$ is continuously differentiable. We first suppose that this is the case and prove the following lemma.

\begin{lemma}\label{lem:ex-decr-smooth}
	Fix $i \geq 0$ and suppose that $\tilde{f}$ is continuously differentiable in the interval $(r_i, s_i \wedge R)$. Then $W_i$ defined in \Cref{lem:ex-decr-visc} is a classical solution of the PDE \eqref{eq:hjb-2} in the region $\overline{D}_i$.
\end{lemma}

\begin{proof}
	Let $x \in \overline{D}_i$ and observe that, by definition of $s_i$,
	\begin{equation}\label{eq:si-ineq}
		w_{i+1}^{\prime\prime}(\abs{x}) \geq -2\tilde{f}(\abs{x}).
	\end{equation}
	Since $\tilde{f}$ is assumed to be continuously differentiable, we have that $w_i$ and $W_i$ are both twice continuously differentiable, and
	\begin{equation}
		D^2W_i(x) = \abs{x}^{-3}\left[\abs{x} w_i^{\prime\prime}(\abs{x}) - w_i^\prime(\abs{x})\right]xx^\top + \abs{x}^{-1}w_i^\prime(\abs{x})I.
	\end{equation}
	Substituting in $w_i^\prime(\abs{x}) = -2 \abs{x}\tilde{f}(\abs{x})$, we have
	\begin{equation}
		D^2W_i(x) = \abs{x}^{-2}\left[w_i^{\prime\prime}(\abs{x}) + 2\tilde{f}(\abs{x})\right]xx^\top - 2\tilde{f}(\abs{x})I.
	\end{equation}
	Hence, for any $\sigma \in U$,
	\begin{equation}
	\begin{split}
		\trace\left(D^2W_i(x)\sigma\sigma^\top\right) & = \abs{x}^{-2}\left[w_i^{\prime\prime}(\abs{x}) + 2\tilde{f}(\abs{x})\right]\trace(xx^\top\sigma\sigma^\top) - 2\tilde{f}(\abs{x})\trace(\sigma\sigma^\top)\\
		& \geq - 2 \tilde{f}(\abs{x})\trace(\sigma\sigma^\top) = - 2 f(x),
	\end{split}
	\end{equation}
	using the inequality \eqref{eq:si-ineq}.
	
	Let $\sigma^\star(x) := \frac{1}{|x|}\begin{bmatrix}
		x^\perp; & 0; & \dots; & 0
\end{bmatrix} \in U$, for some $x^\perp \in \RR^d \setminus \{0\}$ satisfying $x^\top x^\perp = 0$ and $|x| = |x^\perp|$. Then we see that
	\begin{equation}
		\trace\left(xx^\top \sigma^\star(x){\sigma^\star(x)}^\top\right) = 0,
	\end{equation}
	and so
	\begin{equation}
		\trace\left(D^2W_i(x)\sigma^\star(x){\sigma^\star(x)}^\top\right) = - 2 f(x).
	\end{equation}
Hence $W_i$ is a classical solution of the PDE \eqref{eq:hjb-2} in the region $\overline{D}_i$.
\end{proof}

We can now prove \Cref{lem:ex-decr-visc}, by using smooth approximations to the continuous function $\tilde{f}$ and applying a standard stability result for viscosity solutions, which can be found, for example, in Lemma 6.2 of \cite[Chapter II]{fleming_controlled_2006}.

\begin{proof}[Proof of \Cref{lem:ex-decr-visc}]
	Fix $i \geq 1$. Since $\tilde{f}$ is continuous on $[r_i, s_i \wedge R]$, we can approximate $\tilde{f}$ uniformly by polynomials $(\tilde{f}^k)_{k \in \NN}$ (see e.g.\ Theorem 7.26 of \cite{rudin_principles_1976}). Let $k \in \NN$ and define $W_i^k: \overline{D}_i \to \RR$ by
	\begin{equation}
		W_i^k(x) := -2\int_{r_i}^{\abs{x}} \tilde{f}^k(s) s \D s + C^w_i.
	\end{equation}
	Define $f^k: \overline{D}_i \to \RR$ by $f^k(x) = \tilde{f}^k(\abs{x})$, and define $F^k: \overline{D}_i \times \RR^{d,d} \to \RR$ by
	\begin{equation}
		F^k(x, X) = - \frac{1}{2}\inf_{\sigma \in U}\trace(X \sigma \sigma^\top) - f^k(x).
	\end{equation}
	Then, since $\tilde{f}^k$ is continuously differentiable, we can apply Lemma \ref{lem:ex-decr-smooth} to see that $W_i^k$ is a classical solution, and therefore a viscosity solution, of
	\begin{equation}
		F^k(x, D^2W_i^k(x)) = 0 \quad \text{for} \quad x \in \overline{D}_i.
	\end{equation}
	We now show that $F^k$ converges uniformly to $F: \overline{D}_i \times \RR^{d,d} \to \RR$, defined by
	\begin{equation}
		F(x, X) = - \frac{1}{2}\inf_{\sigma \in U}\trace(X \sigma \sigma^\top) - f(x),
	\end{equation}
	and that $W_i^k$ converges uniformly to $W_i$.
	
	Let $\varepsilon > 0$. Then, by uniform convergence of $(\tilde{f}^k)_{k \in \NN}$, there exists $N \in \NN$ such that
	\begin{equation}
		\abs{\tilde{f}(r) - \tilde{f}^k(r)} < \varepsilon, \quad \text{for all} \quad r \in [r_0, R] \quad \text{and} \quad k \geq N.
	\end{equation}
	Let $k \geq N$, $x \in \overline{D}_i$ and $X \in \RR^{d,d}$. Then $\abs{x} \in [r_i, s_i \wedge R]$, and so
	\begin{equation}
	\begin{split}
			\abs{F(x, X) - F^k(x, X)} & = \abs{f(x) - f^k(x)} = \abs{\tilde{f}(\abs{x}) - \tilde{f}^k(\abs{x})} < \varepsilon.
	\end{split}
	\end{equation}
	Therefore $F^k \to F$ uniformly on $\overline{D}_i \times \RR^{d,d}$.
	
	Now choose $M \in \NN$ such that
	\begin{equation}
		\abs{\tilde{f}(r) - \tilde{f}^k(r)} < \frac{\varepsilon}{2 s_i(s_i - r_i)}, \quad \text{for all} \quad r \in [r_i, s_i] \quad \text{and} \quad k \geq M.
	\end{equation}
	Let $k \geq M$ and $x \in \overline{D}_i$. Then $\abs{x} \in [r_i, s_i \wedge R]$, and so
	\begin{equation}
	\begin{split}
		\abs{W_i(x) - W_i^k(x)} = 2 \abs{\int_{r_i}^{\abs{x}}\left(\tilde{f}(s) - \tilde{f}^k(s)\right)s \D s} & \leq 2 \int_{r_i}^{s_i} \abs{\tilde{f}(s) - \tilde{f}^k(s)}\abs{s} \D s\\
		& \leq 2 (s_i - r_i) \frac{\varepsilon}{2 s_i(s_i - r_i)} s_i = \varepsilon.
	\end{split}
	\end{equation}
	Hence $W_i^k \to W_i$ uniformly on $\overline{D}_i$.
	
	We can now apply the stability result given in Lemma 6.2 of \cite[Chapter II]{fleming_controlled_2006}, to conclude that $W_i$ is a viscosity solution of
	\begin{equation}
		F(x, D^2W_i(x)) = 0 \quad \text{for} \quad x \in \overline{D}_i;
	\end{equation}
	i.e.\ $W_i$ is a viscosity solution of the PDE \eqref{eq:hjb-2} in the region $\overline{D}_i$.
\end{proof}

We now combine the above lemmas to prove that $V$ is the value function.

\begin{proof}[Proof of Proposition \ref{prop:radial-symmetric-value}]
	We divide the domain $D$ into disjoint regions and prove first that $V$ is a viscosity solution of \eqref{eq:hjb-2} in the interior of each region.\\
	
	\paragraph{\emph{Step 1:}}
	Fix $i \geq 1$ such that $s_{i - 1} \leq R$, if such a point exists. In the region $D_i = \{x \in D \colon \abs{x} \in (s_{i - 1}, r_i \wedge R)\}$, we have $V = U_i$, for a particular choice of constant $C^u_i$. So by Lemma \ref{lem:ex-incr-visc}, $V$ is a viscosity solution of \eqref{eq:hjb-2} in this region.
	
	Now fix $i \geq 0$ such that $r_i \leq R$, if such a point exists. In the region $\overline{D}_i = \{x \in D \colon \abs{x} \in (r_i, s_i \wedge R)\}$, we have $V = W_i$ for a particular choice of constant $C^w_i$, and so $V$ is a viscosity solution of \eqref{eq:hjb-2} in this region, by Lemma \ref{lem:ex-decr-visc}.\\
	
	\paragraph{\emph{Step 2:}} We next prove that $V$ is a viscosity solution of \eqref{eq:hjb-2} on each of the internal boundaries between the regions.
	
	Let $i \geq 0$ be such that $r_i < R$, if such a point exists. Consider $x_i \in D$ such that $\abs{x_i} = r_i$. Note that
	\begin{equation}\label{eq:ri-hessian}
	\begin{split}
		\lim_{\abs{x} \to r_i -}D^2V(x) & = \lim_{\abs{x} \to r_i -}D^2U_i(x)\\
		& = -\lim_{\abs{x} \to r_i -} \left[2\tilde{f}(\abs{x})I + \abs{x}^{-3}\left(2\abs{x}\tilde{f}(\abs{x}) + u_i^\prime(\abs{x})\right)\left[\abs{x}^2 I - x x^\top\right]\right] = - 2\tilde{f}(r_i)I,
	\end{split}
	\end{equation}
	since $2r_i\tilde{f}(r_i) + {u_i}_-^\prime(r_i) = 0$, by definition of $r_i$ and continuity of $\tilde{f}$.	
	
	 To show that $V$ is a viscosity subsolution at $x_i$, let $x_i \in \arg\min(\phi - V)$, for some $\phi \in C^\infty(D)$. Since $V \in C^1(D)$, it must be the case that $D\phi(x_i) = DV(x_i)$, and that the Hessian of $\phi$ satisfies
	\begin{equation}
		D^2\phi(x_i) \geq \lim_{\abs{x} \to r_i -}D^2V(x) = - 2\tilde{f}(r_i)I,
	\end{equation}
	as calculated in \eqref{eq:ri-hessian}. Hence, for any $\sigma \in U$,
	\begin{equation}
		\trace\left(D^2\phi(x_i)\sigma\sigma^\top\right) \geq - 2 \tilde{f}(r_i) \trace(\sigma\sigma^\top) = - 2 f(x_i),
	\end{equation}
	and so
	\begin{equation}
		- \frac{1}{2}\inf_{\sigma \in U}\trace\left(D^2\phi(x_i)\sigma\sigma^\top\right) \leq f(x_i),
	\end{equation}
	as required.
	
	To show the supersolution property, let $x_i \in \arg\max( \psi - V)$, for some $\psi \in C^\infty(D)$. Then by a similar argument to the one above, we have
	\begin{equation}
		D^2\psi(x_i) \leq - 2 \tilde{f}(r_i) I,
	\end{equation}
	and so
	\begin{equation}
		\trace\left(D^2\psi(x_i)\sigma\sigma^\top\right) \leq - 2f(x_i),
	\end{equation}
	for any $\sigma \in U$, which implies that
	\begin{equation}
		-\frac{1}{2}\inf_{\sigma \in U}\trace\left(D^2\psi(x_i)\sigma\sigma^\top\right)\geq f(x_i).
	\end{equation}
	
	Now let $i \geq 0$ be such that $s_i < R$, if such a point exists, and consider $x_i \in D$ such that $\abs{x_i} = s_i$. Here, note that
	\begin{equation}\label{eq:si-hessian}
	\begin{split}
		\lim_{\abs{x} \to s_i+}D^2V(x) & = \lim_{\abs{x} \to s_i+}D^2U_{i +1}(x)\\
		& = -\lim_{\abs{x} \to s_i+} \left[2\tilde{f}(\abs{x})I + \abs{x}^{-3}\left(2\abs{x}\tilde{f}(\abs{x}) + u_{i+1}^\prime(\abs{x})\right)\left[\abs{x}^2 I - x x^\top\right]\right] = - 2\tilde{f}(s_i)I,
	\end{split}
	\end{equation}
	using the fact that $2s_i \tilde{f}(s_i) + {u_{i+1}}_+^\prime(s_i) = 0$, by definition of $s_i$ and the smooth fit property.
	
	To show that $V$ is a viscosity solution at points of radius $s_i$, we follow the same reasoning as we did for points of radius $r_i$. For $x_i \in \arg \min (\phi - V)$ and $\phi \in C^\infty(D)$, we have that
	\begin{equation}
		D^2\phi(x_i) \geq \lim_{\abs{x} \to s_i+} D^2V(x) = - 2\tilde{f}(s_i)I,
	\end{equation}
	using \eqref{eq:si-hessian}. So, for any $\sigma \in U$,
	\begin{equation}
		\trace\left(D^2\phi(x_i)\sigma\sigma^\top\right) \geq - 2f(x_i),
	\end{equation}
	which implies that the subsolution property holds.
	
	Similarly, for $x_i \in \arg \max(\psi - V)$ and $\psi \in C^\infty(D)$, we have
	\begin{equation}
		D^2\psi(x_i) \leq -2\tilde{f}(s_i)I,
	\end{equation}
	and so, for any $\sigma \in U$,
	\begin{equation}
		\trace\left(D^2\psi(x_i)\sigma\sigma^\top\right) \leq -2f(x_i),
	\end{equation}
	which implies the supersolution property.\\
	
	\paragraph{\emph{Step 3:}} We have shown that $V$ is a viscosity solution of \eqref{eq:hjb-2} in $D\setminus \{0\}$. We now consider the behaviour at the origin. Recall from \Cref{ass:rad-symm} that we have assumed that $\tilde{f}$ is monotone on some interval $(0, \eta)$.\\
	
	\subparagraph{\emph{Case I:}} Suppose that $\tilde{f}$ is strictly increasing on $(0, \eta)$. Then $V = U_1$ in some neighbourhood of the origin. We see that $r_1 \geq \eta$, and so $V = U_1$ in $B_\eta(0)(0)$. Let $x \in B_\eta(0)$ and consider
	\begin{equation}
		D^2V(x) = - 2\tilde{f}(\abs{x})I + \abs{x}^{-3}\left(2\abs{x}\tilde{f}(\abs{x}) + u_1^\prime(\abs{x})\right)\left[\abs{x}^2 I - x x^\top\right].
	\end{equation}
	Since $\abs{x} < r_1$, we have
	\begin{equation}\label{eq:D2-bound-1}
		2\abs{x}\tilde{f}(\abs{x}) + u_1^\prime(\abs{x}) > 0.
	\end{equation}
	Substituting in the value of $u_1^\prime$ and considering a first order Taylor expansion around $0$, we find that there exists $C > 0$ such that
	\begin{equation}\label{eq:D2-bound-2}
	\begin{split}
		2 \abs{x} \tilde{f}(\abs{x}) + u_1^\prime(\abs{x}) = -2\int_0^{\abs{x}}\tilde{f}(s)\D s & = 2\abs{x}\left(\tilde{f}(\abs{x}) - \tilde{f}(0)\right) + o(\abs{x})\\
		& \leq 2\abs{x}\left(\tilde{f}(\abs{x}) - \tilde{f}(0)\right) + C\abs{x}^2.
	\end{split}
	\end{equation}
	Hence, for $j, k \in \{1, \dotso, d\}$, 
	\begin{equation}
	\begin{split}
		0 \leq \abs{x}^{-3}\left(2\abs{x}\tilde{f}(\abs{x}) + u_i^\prime(\abs{x})\right)\abs{\left[\abs{x}^2 I - x x^\top\right]_{jk}} & \leq \abs{x}^{-1}\left(2\abs{x}\tilde{f}(\abs{x}) + u_i^\prime(\abs{x})\right)\\
		& \leq 2 \left(\tilde{f}(\abs{x}) - \tilde{f}(0)\right) C\abs{x}.
	\end{split}
	\end{equation}
	Taking the limit as $\abs{x} \to 0+$, by continuity of $\tilde{f}$, we have
	\begin{equation}
		\lim_{x \to 0} D^2V(x) = - 2 \tilde{f}(0)I.
	\end{equation}
	It is then easy to see that $V$ is a viscosity solution of \eqref{eq:hjb-2} at the origin.\\
	
	\subparagraph{\emph{Case II:}} On the other hand, if $\tilde{f}$ is decreasing in $(0, \eta)$, we have that $V = W_1$ in $B_{\eta}(0)$. Recall from \Cref{ass:rad-symm} that $\tilde{f}$ is continuously differentiable on some interval $(0, \delta)$, and consider $x \in D$ such that $\abs{x} < \delta \wedge \eta$. Then
	\begin{equation}
	\begin{split}
		D^2 V(x) & = \abs{x}^{-2}\left[w_1^{\prime\prime}(\abs{x}) + 2\tilde{f}(\abs{x})\right]xx^\top - 2\tilde{f}(\abs{x})I\\
		& = 2\abs{x}^{-2}\left[-\abs{x}\tilde{f}^\prime(\abs{x}) - \tilde{f}(\abs{x}) + \tilde{f}(\abs{x})\right]xx^\top - 2\tilde{f}(\abs{x})I\\
		& = -2 \abs{x}^{-1}\tilde{f}^\prime(\abs{x})xx^\top - 2 \tilde{f}(\abs{x})I.
	\end{split}
	\end{equation}
	
	Since $\tilde{f}^\prime(\abs{x}) \leq 0$, we get the following bound. For $j, k \in\{1, \dotso, d\}$,
	\begin{equation}
		0 \leq - 2 \abs{x}^{-1} \tilde{f}^\prime(\abs{x})\abs{x_jx_k} \leq -2\abs{x}\tilde{f}^\prime(\abs{x}) \to 0, \quad \text{as} \quad \abs{x}\to 0+,
	\end{equation}
	where the limit is given by the fifth statement of \Cref{ass:rad-symm}.
	
	Therefore $\lim_{x\to0}D^2V(x) = -2 \tilde{f}(0)I$, as in Case I, and so $V$ is a viscosity solution of \eqref{eq:hjb-2} at the origin.\\
	
	\paragraph{\emph{Step 4:}} By construction of the function $V$, the boundary condition $V = 0$ on $\partial D$ is satisfied. We conclude, by Theorem \ref{thm:unique-viscosity}, that the function $V$ is equal to the value function $v$. Also, by the construction of $V$, we have that the value function $v$ is continuously differentiable in $D$.\\
	
	\paragraph{\emph{Step 5:}} Finally, we turn to the proof that the control $\sigma^\star$ is optimal for the weak value function. It is sufficient to show that
	\begin{equation}
		t \mapsto V(X^{\sigma^\star}_t) + \int_0^t f(X^{\sigma^\star}_s) \D s
	\end{equation}
	is a martingale. We will work with the squared radius of the process $X^{\sigma^\star}$, writing $Z^{\sigma^\star}_t = \abs{X^{\sigma^\star}_t}^2$, for $t \geq 0$. We also let $\overline{V}: [0, R^2) \to \RR$ be such that $V(x) = \overline{V}(\abs{x}^2)$ for all $x \in D$.
	
	Suppose that $\tilde{f}$ is increasing on the interval $(0, \eta)$. Then $\sigma^\star$ is the weak control defined via \eqref{eq:conj-optimal-control-incr}. Letting $W$ be the first component of the Brownian motion $B$, \Cref{lem:squared-radius} tells us that $Z^{\sigma^\star}$ satisfies the SDE
	\begin{equation}
		\D Z^{\sigma^{\star}}_t = \D t + 2 \left(\sum_i \ind{Z^{\sigma^\star}_t \in (s_i^2, r_{i+1}^2 \wedge R^2)} + \ind{Z^{\sigma^\star}_t \in [0, r_1^2 \wedge R^2)}\right)\sqrt{Z^{\sigma^\star}_t}\D W_t,
	\end{equation}
	where the index $i$ runs from $1$ to $\inf\{k \in \NN : \; r_{k + 1} \ge R\}$.
	
	In each interval $[r_i^2, s_i^2]$, there is a constant $C$ such that
	\begin{equation}
		\overline{V}(z) = 2 \int_{\sqrt{z}}^{s_i} s \tilde{f}(s) \D s + C.
	\end{equation}
	Therefore, since $\D Z^{\sigma^\star}_t = \D t$ when $Z^{\sigma^\star}_t \in [r_i^2, s_i^2]$, we can make a change of variables to find that
	\begin{equation}\label{eq:optimal-control-tangential}
		\ind{Z^{\sigma^{\star}}_t \in [r_i^2, s_i^2]}\D \overline{V}(Z^{\sigma^\star}_t) = - \ind{Z^{\sigma^{\star}}_t \in [r_i^2, s_i^2]} \tilde{f}\big(\sqrt{Z^{\sigma^\star}_t}\big) \D t.
	\end{equation}
	
	Now, in each interval $(s_i^2, r_{i + 1}^2)$, there is a constant $C$ such that
	\begin{equation}
		\overline{V}(z) = 2 \int_{\sqrt{z}}^{r_{i + 1}}\int_{s_i}^s \tilde{f}(t) \D t \D s + 2 (r_{i + 1} - \sqrt{z}) s_i \tilde{f}(s_i) + C.
	\end{equation}
	We see that $V$ is twice continuously differentiable in such an interval, and so we can apply It\^o's formula to $\overline{V}(Z^{\sigma^\star})$. We calculate the derivatives
	\begin{equation}
		\overline{V}^\prime(z) = - z^{- \frac{1}{2}} \int_{s_i}^{\sqrt{z}} \tilde{f}(s) \D s - z^{- \frac{1}{2}} s_i \tilde{f}(s_i),
	\end{equation}
	and
	\begin{equation}
		\overline{V}^{\prime \prime}(z) = \frac{1}{2} z^{- \frac{3}{2}} \int_{s_i}^{\sqrt{z}} \tilde{f}(s) \D s - \frac{1}{2}Z^{-1} \tilde{f}(\sqrt{z}) + \frac{1}{2} z^{-\frac{3}{2}} s_i \tilde{f}(s_i).
	\end{equation}
	Then, by It\^o's formula, we find that
	\begin{equation}
	\begin{split}
		\ind{Z^{\sigma^\star}_t \in (s_i^2, r_{i + 1}^2)} \D \overline{V}(Z^{\sigma^\star}_t) & = - \ind{Z^{\sigma^\star}_t \in (s_i^2, r_{i + 1}^2)} \tilde{f}\big(\sqrt{Z^{\sigma^\star}_t}\big) \D t + 2 \ind{Z^{\sigma^\star}_t \in (s_i^2, r_{i + 1}^2)} \overline{V}^\prime(Z^{\sigma^\star}_t)\sqrt{Z^{\sigma^\star}_t} \D W_t.
	\end{split}
	\end{equation}
	We have a similar expression for the interval $[0, r_1^2)$, and so combining this with \eqref{eq:optimal-control-tangential}, we have
	\begin{equation}
	\begin{split}
		V(X^{\sigma^\star}_t) - V(X^{\sigma^\star}_0) & = - \int_0^t f(X^{\sigma^\star}_s) \D s + 2\int_0^t \left(\sum_i \ind{Z^{\sigma^\star}_s \in (s_i^2, r_{i + 1}^2)} + \ind{Z^{\sigma^\star}_s \in [0, r_1^2)} \right)\sqrt{Z^{\sigma^\star}_s}\D W_s,
	\end{split}
	\end{equation}
	for any $t \geq 0$. This shows that the required martingale property holds, and so $\sigma^\star$ is an optimal control.
		
	Now suppose that $\tilde{f}$ is decreasing on the interval $(0, \eta)$, and let $X^{\sigma^\star}_0 = x$, for some $x \in D \setminus \{0\}$. In this case $\sigma^\star$ is the weak control defined via \eqref{eq:conj-optimal-control-decr}, and $Z^{\sigma^\star}$ satisfies
	\begin{equation}
		\D Z^{\sigma^{\star}}_t = \D t + 2 \sum_i \ind{Z^{\sigma^\star}_t \in (s_i^2, r_{i+1}^2 \wedge R^2)}\sqrt{Z^{\sigma^\star}_t}\D W_t,
	\end{equation}
	where now the index $i$ runs from $0$ to $\inf\{k \in \NN : \; r_{k + 1} \ge R\}$. We see that $Z^{\sigma^\star}$ never hits the origin.
	
	We can make the same calculations as above to find that, for any $t \geq 0$,
	\begin{equation}
		V(X^{\sigma^\star}_t) - V(X^{\sigma^\star}_0) = - \int_0^t f(X^{\sigma^\star}_s) \D s + 2\int_0^t \sum_i \ind{Z^{\sigma^\star}_s \in (s_i^2, r_{i + 1}^2)}\sqrt{Z^{\sigma^\star}_s}\D W_s,
	\end{equation}
	and so the required martingale property holds once again. We conclude that $\sigma^\star$ is optimal for the weak value function.
\end{proof}

We required the smoothness conditions on the running cost $f$ in \Cref{ass:rad-symm} in order to show that the candidate value function is a viscosity solution at the origin. In \Cref{sec:exploding-cost}, we will relax these assumptions and extend the above result to include cost functions that have an infinite discontinuity at the origin. In this case, we cannot define a viscosity solution of the HJB equation \eqref{eq:hjb-2} at the origin, and so \Cref{thm:unique-viscosity} will no longer be applicable.

\section{Infinite cost at the origin}\label{sec:exploding-cost}
	We now extend \Cref{prop:radial-symmetric-value} by considering the case where the cost function is continuous on the whole domain, except at the origin where it may become infinite. We will show that the value function takes the same form as we saw in \Cref{prop:radial-symmetric-value}. We will also find growth conditions on the cost function under which the value function becomes infinite. We note that, in allowing the cost function to become infinite at the origin, we must take care to check that we still have equality between the strong value function $v^S$ and the weak value function $v^W$, as we showed in \Cref{prop:weak-strong} for the case of continuous cost functions. In a particular growth regime, we cannot prove the equality $v^S(0) = v^W(0)$ in dimension $d = 2$ using the tools of this section. We prove this equality in a parallel work \cite{CoRo22}, using the theory of Brownian filtrations.
	
	We relax the regularity conditions on the cost function $f$ from \Cref{ass:rad-symm}, as follows.
	
	\begin{ass}\label{ass:relaxed}
		We assume that
		\begin{enumerate}
		\item The domain is $D = B_R(0) \subset \RR^d$, for some $R > 0$ and $d \geq 2$;
		\item The cost function $f$ is radially symmetric; i.e.\ $f(x) = \tilde{f}(\abs{x})$, for some function $\tilde{f}: [0, R) \to \RR \cup\{\pm\infty\}$;
		\item The boundary cost $g$ is constant --- without loss of generality, we suppose that $g \equiv 0$;
		\item The cost function $f$ is continuous on $D \setminus \{0\}$;
		\item There exists $\eta > 0$ such that the cost function $\tilde{f}$ is monotone on the interval $(0, \eta)$;
		\item The one-sided derivative $\tilde{f}^\prime_+(r)$ exists for all $r > 0$ and changes sign only finitely many times.
	\end{enumerate}

	\end{ass}
	
	Note that we retain the fifth statement in this assumption to ensure that the cost function does not oscillate as it approaches the origin, and we retain the sixth statement so that there are finitely many switching points and these are well-defined.
	
	Having relaxed the conditions on the cost function $f$, we can no longer use the theory of viscosity solutions. To prove the following results, we once again treat the cases of increasing and decreasing costs separately, and we distinguish between regimes of slow and fast growth at the origin. The different growth regimes will be determined by the convergence of the integrals
	\begin{equation}
		\int_0^r \tilde{f}(s) \D s \qandq \int_0^r s \tilde{f}(s) \D s.
	\end{equation}
	
	\subsection{Cost functions increasing at the origin}
	
	We first consider cost functions that are increasing in some neighbourhood around the origin. In this case, we will find that radial motion, as defined in \Cref{def:radial}, is optimal close to the origin.
	
	\begin{prop}\label{prop:incr-slow-growth}
		Suppose that \Cref{ass:relaxed} holds and there exists $\eta > 0$ such that $\tilde{f}$ is negative and increasing on the interval $(0, \eta)$. Then the strong and weak value functions defined in \Cref{sec:problem-formulation} are equal, and we can write $v = v^S = v^W$. Moreover, for the candidate value $V$ defined in \Cref{def:candidate-value},
		\begin{equation}
			v = \begin{cases}
				V \in (-\infty, \infty), & \text{if} \quad \int_0^r \tilde{f}(s) \D s > - \infty, \quad \text{for any} \quad r > 0,\\
				- \infty, & \text{if} \quad \int_0^r \tilde{f}(s) \D s = - \infty, \quad \text{for any} \quad r > 0.
			\end{cases}
		\end{equation}
	\end{prop}
	
	\begin{remark}
		Note that, since $\tilde{f}$ is increasing on $(0, \eta)$, the function $V$ is defined in Case I of \Cref{def:candidate-value}, with $s_0 = 0$ and $r_1 = \inf\left\{r > 0 \colon \int_0^r \tilde{f}(s) \D s > r \tilde{f}(r)\right\}$. When $\int_0^r \tilde{f}(s) \D s > - \infty$ for any $r > 0$, the switching point $r_1$ is well-defined.
	\end{remark}
	
	\begin{proof}[Proof of \Cref{prop:incr-slow-growth}]
		First suppose that, for any $r > 0$, 
		\begin{equation}
			\int_0^r \tilde{f}(s) \D s > - \infty.	
		\end{equation}
		For $N \in \NN$, define an approximating sequence of functions $\tilde{f}_N:[0, R) \to \RR$ by
		\begin{equation}
			\tilde{f}_N(r) =
			\begin{cases}
				\tilde{f}(\frac{1}{N}), & r \leq \frac{1}{N},\\
				\tilde{f}(r), & r > \frac{1}{N},
			\end{cases}
		\end{equation}
		and define $f_N: D \to \RR$ by $f_N(x) = \tilde{f}_N(\abs{x})$ for $x \in D$. Then $f_N$ is continuous and bounded. Moreover, fixing $N > \frac{1}{\eta}$, we have the bound $f_N \geq f$. Now define $v^S_N: D \to \RR$ by
		\begin{equation}
			v^S_N(x) := \inf_{\sigma \in \mathcal{U}} \EE^x \left[\int_0^\tau f_N(X^\sigma_s) \D s\right], \quad x \in D,
		\end{equation}
		using the same notation as in the definition of the strong value function $v^S$ in \Cref{sec:problem-formulation}. Note that $v^S_N \geq v^S$.
		
		Let $V_N$ denote the candidate value function defined in Case I of \Cref{def:candidate-value} with the function $\tilde{f}$ replaced by $\tilde{f}_N$. Since \Cref{ass:rad-symm} is satisfied for the value function $v^S_N$, we can apply \Cref{prop:radial-symmetric-value} to see that $v^S_N = V_N$. We can also see that, for any $x \in D$, $\lim_{N \to \infty} V_N(x) = V(x)$ and $V(x)$ is finite, since $\int_0^r \tilde{f}(s) \D s > - \infty$ for any $r > 0$. We will show that $\lim_{N \to \infty}v_N^S(x) = v^S(x)$ and conclude that $v^S(x) = V(x)$.
		
		Fix $\sigma \in \mathcal{U}$ and $x \in D$. We have
		\begin{equation}\label{eq:approx-incr}
		\begin{split}
			\EE^x \left[\int_0^\tau \tilde{f}_N(\abs{X^\sigma_s}) \D s\right] & = \EE^x \left[\int_0^\tau \tilde{f}(\abs{X^\sigma_s})\mathds{1}_{\left\{\abs{X^\sigma_s} \in (\frac{1}{N}, R)\right\}} \D s\right] + \tilde{f}\left(\frac{1}{N}\right) \EE^x \left[\int_0^\tau \mathds{1}_{\left\{\abs{X^\sigma_s} \leq \frac{1}{N}\right\}}\right].
		\end{split}
		\end{equation}
		Define $K: = \sup \{f(x) \colon x \in D\}$ and note that $K < \infty$ by continuity of $f$ in $D \setminus \{0\}$. Then the sequence
		\begin{equation}
			\left(\int_0^\tau \tilde{f}(\abs{X^\sigma_s})\mathds{1}_{\left\{\abs{X^\sigma_s} \in (\frac{1}{N}, R)\right\}} \D s\right)_{N \in \NN}
		\end{equation}
		is decreasing for $N > \frac{1}{\eta}$ and bounded above by $\tau K$. Since $\tau$ has finite expectation by \Cref{rem:exit-time}, we can apply monotone convergence (see e.g.\ Theorem 1 of \cite[Chapter II, \S 6]{shiryaev_probability_1996}) to show that
		\begin{equation}
		\begin{split}
			\lim_{N \to \infty} \EE^x \left[\int_0^\tau \tilde{f}(\abs{X^\sigma_s})\mathds{1}_{\left\{\abs{X^\sigma_s} \in (\frac{1}{N}, R)\right\}} \D s\right] & = \EE^x \left[\lim_{N \to \infty}\int_0^\tau \tilde{f}(\abs{X^\sigma_s})\mathds{1}_{\left\{\abs{X^\sigma_s} \in (\frac{1}{N}, R)\right\}} \D s\right] = \EE^x \left[\int_0^\tau \tilde{f}(\abs{X^\sigma_s})\D s\right].
		\end{split}
		\end{equation}
		
		We will show that the second term of \eqref{eq:approx-incr} vanishes as $N \to \infty$ by referring to \Cref{prop:step-incr} on the control problem for a step cost function. Note that $\tilde{f}(\frac{1}{N}) < 0$. For $x \neq 0$, we can choose $N > \frac{1}{\abs{x}}$, so that, by \Cref{prop:step-incr},
		\begin{equation}
		\begin{split}
			0 > \tilde{f}\left(\frac{1}{N}\right) \EE^x \left[\int_0^\tau \mathds{1}_{\left\{\abs{X^\sigma_s} \leq \frac{1}{N}\right\}}\right] & = - \tilde{f}\left(\frac{1}{N}\right)\EE^x \left[\int_0^\tau -\mathds{1}_{\left\{\abs{X^\sigma_s} \leq \frac{1}{N}\right\}}\right]\\
			& \geq - \frac{2}{N}\tilde{f}\left(\frac{1}{N}\right)\left(R - \abs{x}\right) \xrightarrow{N \to \infty} 0,
		\end{split}
		\end{equation}
		using the condition that $\int_0^r\tilde{f}(s) \D s > - \infty$ to find the limit.
		
		For $x = 0$, \Cref{prop:step-incr} gives us
		\begin{equation}
		\begin{split}
			0 > \tilde{f}\left(\frac{1}{N}\right)\EE^0\left[\int_0^\tau \mathds{1}_{\left\{\abs{X^\sigma} \leq \frac{1}{N}\right\}}\D s\right] & = - \tilde{f}\left(\frac{1}{N}\right)\EE^0\left[-\int_0^\tau \mathds{1}_{\left\{\abs{X^\sigma} \leq \frac{1}{N}\right\}}\D s\right]\\
			& \geq \tilde{f}\left(\frac{1}{N}\right)\left(\frac{2R}{N} - \frac{1}{N^2}\right) \xrightarrow{N \to \infty} 0.
		\end{split}
		\end{equation}
		Hence
		\begin{equation}
			\lim_{N \to \infty}\EE^x \left[\int_0^\tau \tilde{f}_N(\abs{X^\sigma_s}) \D s\right] = \EE^x \left[\int_0^\tau \tilde{f}(\abs{X^\sigma_s}) \D s\right],
		\end{equation}
		for any $\sigma \in \mathcal{U}$, $x \in D$.
		
		Now fix $x \in D$ and $\varepsilon > 0$ and choose $\sigma^\varepsilon$ to be an $\varepsilon$-optimal strategy for the cost function $f$; i.e.
		\begin{equation}
			\EE^x \left[\int_0^\tau f(X^{\sigma^\varepsilon}_s) \D s\right] \leq v^S(x) + \varepsilon.
		\end{equation}
		Then
		\begin{equation}
		\begin{split}
			v^S(x) + \varepsilon & \geq \EE^x \left[\int_0^\tau f(X^{\sigma^\varepsilon}_s) \D s\right] = \lim_{N \to \infty}\EE^x \left[\int_0^\tau f_N(X^{\sigma^\varepsilon}_s) \D s\right] \geq \lim_{N \to \infty}v^S_N(x) \geq v^S(x).
		\end{split}
		\end{equation}
		Taking the limit as $\varepsilon \downarrow 0$, we see that
		\begin{equation}
			v^S(x) = \lim_{N \to \infty}v_N^S(x),
		\end{equation}
		and by uniqueness of the limit, we have that $v^S(x) = V(x)$.
		
		As in \Cref{prop:weak-strong}, we can apply Theorem 4.5 of \cite{el_karoui_capacities_2013-1} to see that $v^S = v^W$. Since $f$ is continuous in $D \setminus \{0\}$, upper semicontinuous at $0$, and bounded above by a constant, we can deduce that the conditions of Theorem 4.5 of \cite{el_karoui_capacities_2013-1} are met in the same way as in the proof of \Cref{prop:weak-strong}. Hence $v^W = v^S = V$.
		
		Now suppose that, for any $r > 0$,
		\begin{equation}
			\int_0^r \tilde{f}(s) \D s = - \infty.
		\end{equation}
		We will show that radial motion is an optimal strategy and that this strategy gives a negative infinite cost. Let the control $\sigma^1$ be as defined in \Cref{def:radial}, and define $X^{\sigma^1}$ by
		\begin{equation}
			X^{\sigma^1}_t = x + \int_0^t \sigma^1_s \D B_s, \quad t \geq 0.
		\end{equation}
		Let $W$ be the first component of the Brownian motion $B$. First suppose that $x \neq 0$. Then, for any $t \geq 0$,
		\begin{equation}
			X^{\sigma^1}_t = x + \frac{x}{\abs{x}}W_t,
		\end{equation}
		and so
		\begin{equation}
		\begin{split}
			\EE^x \left[\int_0^\tau \tilde{f}\big(\big\lvert X^{\sigma^1}_s \big \rvert\big)\D s\right] = \EE^{\abs{x}}\left[\int_0^\tau \tilde{f}(W_s) \ind{W_s \geq 0}\D s \right] + \EE^{\abs{x}}\left[\int_0^\tau \tilde{f}(-W_s) \ind{W_s < 0} \D s\right].
		\end{split}
		\end{equation}
		We can now use the Green's function $G$ for the one-dimensional Brownian motion $W$ on the interval $(-R, R)$, as calculated in \Cref{ex:step-incr}. By Corollary 3.8 of \cite[Chapter VII]{revuz_continuous_1999}, we see that
		\begin{equation}
		\begin{split}
			\EE^x \left[\int_0^\tau \tilde{f}(X^{\sigma^1}_s)\D s\right] & = 2 \int_0^R G(\abs{x}, y) \tilde{f}(y) \D y + 2 \int_{-R}^0 G(\abs{x}, y) \tilde{f}(-y) \D y\\
			& = \frac{\abs{x} + R}{R}\int_{\abs{x}}^R (R - y) \tilde{f}(y) \D y + \frac{R - \abs{x}}{R}\int_0^{\abs{x}}(y + R) \tilde{f}(y) \D y + \frac{R - \abs{x}}{R}\int_{-R}^0 (y + R) \tilde{f}(-y) \D y.
		\end{split}
		\end{equation}
		Making a change of variables $y \mapsto -y$ in the last integral gives
		\begin{equation}
		\begin{split}
			\EE^x \left[\int_0^\tau \tilde{f}(X^{\sigma^1}_s)\D s\right] & = 2\int_{\abs{x}}^R(R - y) \tilde{f}(y) \D y + 2 (R - \abs{x})\int_0^{\abs{x}} \tilde{f}(y) \D y.
		\end{split}
		\end{equation}
		Since $f$ is bounded above and $\int_0^{\abs{x}} \tilde{f}(y) \D y = - \infty$, we have
		\begin{equation}
			\EE^x \left[\int_0^\tau \tilde{f}(X^{\sigma^1}_s)\D s\right] = - \infty.
		\end{equation}
		
		Now let $x = 0$. Then, $X^{\sigma^1}_t = e_1 W_t$, for $t \geq 0$. Using the symmetry of the Green's function $G$ for $W$ about zero, together with the growth condition on $f$, we have
		\begin{equation}\label{eq:green-incr-cost}
		\begin{split}
			\EE^0\left[\int_0^\tau \tilde{f}\big(\big \lvert X^{\sigma^1}_s \big \rvert \big)\D s\right] & = 2\EE^0\left[\int_0^\tau \tilde{f}(W_s)\ind{W_s \geq 0}\D s\right]\\
			& = 4\int_0^R G(0, y) \tilde{f}(y)\D y = 2\int_0^R (R - y) \tilde{f}(y) \D y = - \infty.
		\end{split}
		\end{equation}
		
		We conclude that
		\begin{equation}
		\begin{split}
			v^W(x) \leq v^S(x) \leq \EE^x \left[\int_0^\tau \tilde{f}(X^{\sigma^1}_s)\D s\right] = - \infty. \qedhere
		\end{split}
		\end{equation}
	\end{proof}
	
	We have shown that, for cost functions increasing at the origin, there is a dichotomy depending on the convergence of $\int_0^r \tilde{f}(s) \D s$. When $\int_0^r \tilde{f}(s) \D s > - \infty$ for any $r > 0$, the value function is finite and equal to $V$, and when $\int_0^r \tilde{f}(s) \D s = - \infty$ for any $r > 0$, the value is identically equal to negative infinity.
	
	\subsection{Cost functions decreasing at the origin}\label{sec:decreasing}
	
	We now consider cost functions that are decreasing in some neighbourhood around the origin. Excluding the origin from this neighbourhood, an optimal strategy is tangential motion, as defined in \Cref{def:tangential}. We will first show that, away from the origin, the form of the value function is unchanged from the value function in \Cref{prop:radial-symmetric-value}.
	
	\begin{prop}\label{prop:decr-away-from-origin}
		Suppose that \Cref{ass:relaxed} holds and there exists $\eta > 0$ such that $\tilde{f}$ is positive and decreasing on the interval $(0, \eta)$. Then, for $x \in D \setminus \{0\}$, $v(x) = v^S(x) = v^W(x) = V(x) \in (-\infty, \infty)$, where $V$ is the candidate value function defined in \Cref{def:candidate-value}.
	\end{prop}
	
	\begin{remark}
		In this case, since $\tilde{f}$ is decreasing on $(0, \eta)$, $V$ is defined in Case II of \Cref{def:candidate-value}.
	\end{remark}
	
	\begin{proof}[Proof of \Cref{prop:decr-away-from-origin}]
		For $N \in \NN$, define $\tilde{f}_N$, $f_N$ and $v^S_N$ as in the proof of \Cref{prop:incr-slow-growth}. Also define $v_N^W: D \to \RR$ by
		\begin{equation}
			v_N^W(x) := \inf_{\PP \in \mathcal{P}_x}\EE^\PP\left[\int_0^\tau f_N(X_s) \D s\right], \quad x \in D,
		\end{equation}
		using the same notation as in the definition of the weak value function $v^W$ in \Cref{sec:problem-formulation}. Now, for $N > \frac{1}{\eta}$, we have $\tilde{f}_N \leq \tilde{f}$, $f_N \leq f$, $v_N^S \leq v^S$ and $v_N^W \leq v^W$. Recall that, by \Cref{prop:radial-symmetric-value}, $v^W_N = v^S_N = V_N$, where $V_N$ is the candidate value function defined in Case II of \Cref{def:candidate-value} with the cost function $\tilde{f}$ replaced by $\tilde{f}_N$.
		
		Fix $x \in D \setminus\{0\}$ and $N > \frac{1}{\abs{x}} \vee \frac{1}{\eta}$. Then we can see that $V_N(x) = V(x) \in (-\infty, \infty)$, so $v^S_N(x) = v^W_N(x) = V(x)$. We will now show that $v^W(x) = v^S(x) = V(x)$.
		
		Let $\sigma^\star$ be the weak control defined via \eqref{eq:conj-optimal-control-decr}. Since $\tilde{f}^N$ is decreasing in the interval $(0, \eta)$, \Cref{prop:radial-symmetric-value} shows that $\sigma^\star$ is optimal for $v^W_N$, and so
		\begin{equation}\label{eq:approx-decr}
		\begin{split}
			v_N^W(x) & = \EE^x \left[\int_0^\tau \tilde{f}_N\big(\big\lvert X^{\sigma^\star}_s\big\rvert\big)\D s\right]\\
			& = \tilde{f}\left(\frac{1}{N}\right)\EE^x \left[\int_0^\tau \mathds{1}_{\left\{\abs{X^{\sigma^\star}_s} \leq \frac{1}{N}\right\}}\D s\right] + \EE^x \left[\int_0^\tau \tilde{f}\big(\big\lvert X^{\sigma^\star}_s\big\rvert\big)\mathds{1}_{\left\{\abs{X^{\sigma^\star}_s} \in (\frac{1}{N}, R)\right\}}\D s\right].
		\end{split}
		\end{equation}
		When $\abs{X^{\sigma^\star}_t} \in (0, \eta)$, the radius process $t \mapsto \abs{X_t^{\sigma^\star}}$ is deterministically increasing, by \Cref{lem:squared-radius}. Therefore, $\mathds{1}_{\left\{\abs{X^{\sigma^\star}_t} \leq \frac{1}{N}\right\}} = 0$, for all $t \geq 0$, since $\abs{x} > \frac{1}{N}$. Hence, by \eqref{eq:approx-decr} and the definition of $v^W$, we have
		\begin{equation}
			v^W(x) \geq v_N^W(x) = \EE^x \left[\int_0^\tau \tilde{f}\big(\big\lvert X^{\sigma^\star}_s\big\rvert\big)\D s\right] \geq v^W(x),
		\end{equation}
		and so $v^W(x) = v^W_N(x) = V_N(x) = V(x)$.
		
		Now fix $\varepsilon > 0$ and let $\sigma^\varepsilon \in \mathcal{U}$ be $\varepsilon$-suboptimal for $v^S_N(x)$. By \Cref{prop:weak-strong}, $v_N^S(x) = v_N^W(x)$. Therefore
		\begin{equation}
			V(x) = v^W_N(x) = v^S_N(x) \leq v^S(x) \leq \EE^x \left[\int_0^\tau \tilde f(|X^{\sigma^\varepsilon}_t|) \D t\right] \leq \liminf_{N \to \infty} \EE^x \left[\int_0^\tau \tilde f_N(|X^{\sigma^\varepsilon}_t|) \D t\right] \leq V(x) + \varepsilon,
		\end{equation}
		and we conclude that $v^W(x)  = v^S(x) = V(x)$.
		\end{proof}
	
	At the origin, we have not shown that there exists an optimal control. The function $\sigma^0$ introduced in \Cref{def:tangential} is not defined at the origin, and so we require an approximation to tangential motion. We consider different growth rates separately, as we did for increasing costs.
		
	\begin{prop}\label{prop:decr-slow-growth}
		Suppose that \Cref{ass:relaxed} holds and there exists $\eta > 0$ such that $\tilde{f}$ is positive and decreasing on the interval $(0, \eta)$. Suppose further that, for any $r > 0$,
		\begin{equation}
			\int_0^r \tilde{f}(s) \D s < \infty.
		\end{equation}
		Then $v(0) = v^S(0) = v^W(0) = V(0) \in (-\infty, \infty)$, where $V$ is the candidate value defined in \Cref{def:candidate-value}.
	\end{prop}
	
	\begin{proof}
		For $N \in \NN$, define $\tilde{f}_N$, $f_N$, $v^W_N$ and $v^S_N$ as in the proof of \Cref{prop:incr-slow-growth}. Letting $V_N$ be the candidate value function in Case II of \Cref{def:candidate-value} with $\tilde{f}$ replaced by $\tilde{f}_N$, we have $v^W_N(0) = v_N^S(0) = V_N(0)$, by \Cref{prop:radial-symmetric-value}. We also see that $\lim_{N \to \infty}V_N(0) = V(0)$, and the value $V(0)$ is finite due to the growth condition on $\tilde{f}$. We will show that $v^W(0) = \lim_{N \to \infty} v^W_N(0)$ and conclude that $v^W(0) = V(0)$.
	
		Fix $\delta \in (0, \eta)$ and $N > \frac{1}{\delta}$. Denote by $e_1$ the unit vector in the first coordinate direction. Let $X^{\sigma^N}$ be a weak solution of the SDE $\D X_t = \sigma^N_t(X) \D B_t$, where
		\begin{equation}
			\sigma^{N}_t(X) =
			\begin{cases}
				\begin{bmatrix}
				e_1; & 0; & \dotsc; & 0
			\end{bmatrix},
			& \text{for} \quad \abs{X_t} < \frac{1}{N},\\
			 \sigma^0(X_t), &\text{for} \quad \abs{X_t} \in \big[\frac{1}{N}, \eta\big),
			\end{cases}
		\end{equation}
		and the function $\sigma^0$ is defined as in \Cref{def:tangential}.
		By \Cref{rem:weak-control}, $\sigma^N$ defines a weak control. Since $\tilde{f}_N$ is constant on $(0, \frac{1}{N})$ and decreasing on $(\frac{1}{N}, \eta)$, by \Cref{prop:radial-symmetric-value}, we have
		\begin{equation}
			v^S_N(0) = v^W_N(0) = \EE^0\left[\int_0^\tau \tilde{f}_N\big(\big\lvert X^{\sigma^{N}}_s\big\rvert\big)\D s\right].
		\end{equation}
		Also define $\sigma^\delta$ to coincide with $\sigma^{N}$ except that we set $\sigma^\delta_t(X) =
			\begin{bmatrix}
				e_1; & 0; & \dotsc; & 0
			\end{bmatrix}$, for $\abs{X_t} \in [\frac{1}{N}, \delta)$.
		
		Under either control $\sigma^N$ or $\sigma^\delta$, the process $t \mapsto \abs{X_t}$ is deterministically increasing on the interval $(\delta, \eta)$, by \Cref{lem:deterministically-increasing}. Therefore, writing $\tau^\delta$ for the exit time from the ball $B_\delta(0)$, the error between the value $v_N^W(0)$ and the expected cost of choosing the control $\sigma^\delta$ with the cost $f$ is
		\begin{equation}
		\begin{split}
			0 \leq E_N(\delta) := & \EE^0\left[\int_0^\tau \tilde{f}\big(\big \lvert X^{\sigma^\delta}_s \big \rvert \big) \D s\right] - v^W_N(0)\\
			= & \EE^0\left[\int_0^{\tau_\delta} \tilde{f}\big(\big\lvert X^{\sigma^\delta}_s\big \rvert\big) \D s\right] - \EE^0\left[\int_0^{\tau_\delta} \tilde{f}_N\big(\big\lvert X^{\sigma^{N}}_s\big\rvert\big) \D s\right].
		\end{split}
		\end{equation}
		In the ball $B_\delta(0)$, the process $X^{\sigma^\delta}$ is equal to a one-dimensional Brownian motion in the direction $e_1$ and so, making a calculation with the Green's function similar to \eqref{eq:green-incr-cost} in the proof of \Cref{prop:incr-slow-growth}, we find that
		\begin{equation}
		\begin{split}
			\EE^0\left[\int_0^{\tau_\delta} \tilde{f}\big( \big\lvert X^{\sigma^\delta}_s\big\rvert\big) \D s\right] & = 2\int_0^\delta \left(\delta - y\right) \tilde{f}(y) \D y.
		\end{split}
		\end{equation}
		We now compute the expected cost under the control $\sigma^{N}$. When $\big\lvert X^{\sigma^N}_t \big\rvert \in (\frac{1}{N}, \delta)$, the process $X^{\sigma^N}$ follows tangential motion, and so we can calculate
		\begin{equation}
			\EE^{\frac{1}{N}}\left[\int_0^{\tau_\delta}\tilde{f}_N\big(\big\lvert X^{\sigma^{N, \varepsilon}}_s\big\rvert\big)\D s\right] = \int_{0}^{\delta^2 - N^{-2}} \tilde{f}\big(\sqrt{N^{-2} + s}\big)\D s = 2 \int_{\frac{1}{N}}^\delta s \tilde{f}(s) \D s.
		\end{equation}
		In the ball $B_{\frac{1}{N}}(0)$, the process $X^{\sigma^{N}}$ is a one-dimensional Brownian motion and so, making another calculation with the Green's function, we can write
		\begin{equation}
		\begin{split}
			\EE^0\left[\int_0^{\tau_\delta} \tilde{f}_N\big(\big\lvert X^{\sigma^{N}}_s\big\rvert\big) \D s\right] & = \EE^0\left[\int_0^{\tau_{\frac{1}{N}}} \tilde{f}_N\big(\big\lvert X^{\sigma^{N}}_s\big\rvert\big) \D s\right] + \EE^{\frac{1}{N}}\left[\int_0^{\tau_\delta}\tilde{f}_N\big(\big\lvert X^{\sigma^{N}}_s\big\rvert\big) \D s\right]\\
			& = \frac{1}{N^2}\tilde{f}\left(\frac{1}{N}\right) + 2\int_{\frac{1}{N}}^\delta y \tilde{f}(y) \D y.
		\end{split}
		\end{equation}
		Therefore the error is
		\begin{equation}
			E_N(\delta) = 2 \int_0^\delta (\delta - y) \tilde{f}(y) \D y - 2 \int_{\frac{1}{N}}^\delta y \tilde{f}(y) \D y - \frac{1}{N^2} \tilde{f}\left(\frac{1}{N}\right).
		\end{equation}
		Since $\int_0^r \tilde{f}(s) \D s < \infty$, for any $r > 0$, we can take the limit as $N \to \infty$ to get
		\begin{equation}
			E(\delta) := \lim_{N \to \infty}E_N(\delta) = 2 \int_0^\delta \left(\delta - 2y\right)\tilde{f}(y) \D y,
		\end{equation}
		and then taking the limit as $\delta \to 0$ gives
		\begin{equation}\label{eq:error-limit}
		\begin{split}
			0 \leq E(\delta) & = 2 \int_0^\delta \left(\delta - 2y\right)\tilde{f}(y) \D y \xrightarrow{\delta \to 0} 0.
		\end{split}
		\end{equation}
		
		Returning to the definition of $E_N(\delta)$, for fixed $\delta \in (0, \eta)$ and $N > \frac{1}{\delta}$, we recall that
		\begin{equation}\label{eq:decr-error}
			v^W_N(0) + E_N(\delta) = \EE^0\left[\int_0^\tau \tilde{f}\big(\big\lvert X^{\sigma^{\delta}}_s\big\rvert\big)  \D s\right].
		\end{equation}
		Since $\tilde{f}_N \leq \tilde{f}$, we have
		\begin{equation}
			v^W(0) + E_N(\delta) \geq v_N^W(0) + E_N(\delta),
		\end{equation}
		and, by the definition of $v^W$,
		\begin{equation}
			\EE^0\left[\int_0^\tau \tilde{f}\big(\big\lvert X^{\sigma^{\delta}}_s\big\rvert\big)  \D s\right] \geq v^W(0).
		\end{equation}
		Combining these inequalities with \eqref{eq:decr-error}, we see that
		\begin{equation}
			v^W(0) + E_N(\delta) \geq v^W_N(0) + E_N(\delta) \geq v^W(0).
		\end{equation}
		Since the sequence $\left(v_N^W(0)\right)_{N \in \NN}$ is monotone, we can take the limit as $N \to \infty$ and find that
		\begin{equation}
			v^W(0) + E(\delta) \geq \lim_{N \to \infty}v^W_N(0) + E(\delta) \geq v^W(0)	.
		\end{equation}
		Having calculated that $\lim_{\delta \to 0}E(\delta) = 0$ in \eqref{eq:error-limit}, we have
		\begin{equation}
			v^W(0) = \lim_{N \to \infty}v^W_N(0) = V(0).
		\end{equation}
		
		Now fix $\varepsilon > 0$ and let $\sigma^\varepsilon \in \mathcal{U}$ be $\varepsilon$-suboptimal for $v^S_N(0)$. Then
		\begin{equation}
			v^W_N(0) = v^S_N(0) \leq v^S(0) \leq \EE^0 \left[\int_0^\tau \tilde f(|X^{\sigma^\varepsilon}_t|) \D t\right] \leq \liminf_{N \to \infty} \EE^0 \left[\int_0^\tau \tilde f_N(|X^{\sigma^\varepsilon}_t|) \D t\right] \leq \liminf_{N \to \infty}v^S_N(0) + \varepsilon = V(0) + \varepsilon.
		\end{equation}
		Since $\lim_{N \to \infty} v^W_N(0) = V(0)$, we conclude that $v^S(0) = v^W(0) = V(0)$.
	\end{proof}
	
	\begin{remark}\label{rem:origin-approx}
		Note that, if the growth rate of $\tilde{f}$ is such that, for any $r > 0$, $\int_0^r \tilde{f}(s) \D s = \infty$, then the error $E(\delta)$ in the proof of \Cref{prop:decr-slow-growth} is infinite for all $\delta$. Therefore the above argument does not generalise to costs with faster growth at the origin.
	\end{remark}
	
	We now consider decreasing costs with faster growth at the origin.
	
	\begin{prop}\label{prop:decr-fast-growth}
		Suppose that \Cref{ass:relaxed} holds and that there exists $\eta > 0$ such that $\tilde{f}$ is positive and decreasing on the interval $(0, \eta)$. If, for any $r > 0$,
		\begin{equation}
			\int_0^r s \tilde{f}(s) \D s = \infty,
		\end{equation}
		then $v^S(0) = v^W(0) = + \infty$.
	\end{prop}

	\begin{proof}
		Once again define $\tilde{f}_N$, $f_N$ and $v^W_N$ as in the proof of \Cref{prop:incr-slow-growth}, for $N \in \NN$. Let $N > \frac{1}{\eta}$ and define the weak control $\sigma^{N}$ as in the proof of \Cref{prop:decr-slow-growth}, so that $\sigma^N$ is optimal for $v^W_N$. Using the calculations of the expected cost under the control $\sigma^{N}$ from the proof of \Cref{prop:decr-slow-growth}, we find that
		\begin{equation}
		\begin{split}
			v_N^W(0) = \EE^0\left[\int_0^\tau \tilde{f}_N\big(\big\lvert X^{\sigma^{N}}_s \big\rvert\big)\D s\right] & = \EE^0\left[\int_0^{\tau_\eta} \tilde{f}_N\big(\big\lvert X^{\sigma^{N}}_s \big\rvert\big)\D s\right] + \EE^0\left[\int_0^\tau \tilde{f}\big(\big\lvert X^{\sigma^{N}}_s \big\rvert\big)\ind{\abs{X^{\sigma^{N}}}\in(\eta, R)}\D s\right]\\
			& \geq \frac{1}{N^2}\tilde{f}\left(\frac{1}{N}\right) + 2\int_{\frac{1}{N}}^\eta y \tilde{f}(y) \D y + (R^2 - \eta^2) \min\left\{\tilde{f}(r) \colon r \in (\eta, R)\right\}.
		\end{split}
		\end{equation}
		By the growth condition on $\tilde{f}$, we have
		\begin{equation}
			\lim_{N \to \infty}\int_{\frac{1}{N}}^\eta y \tilde{f}(y) \D y = + \infty.
		\end{equation}
		Also, since $\tilde{f}$ is continuous on $(0, R)$, we have $\min\left\{\tilde{f}(r) \colon r \in (\eta, R)\right\} > - \infty$,
		and so
		\begin{equation}
			\lim_{N \to \infty} v_N^W(0) = + \infty.
		\end{equation}
		We conclude that
		\begin{equation}
		\begin{split}
			v^S(0) \geq v^W(0) \geq \lim_{N \to \infty}v_N^W(0) = +\infty. \qedhere
		\end{split}
		\end{equation}
	\end{proof}

	We have now fully characterised the value function for any radially symmetric cost, except for the value at the origin when the cost function is decreasing at the origin and grows at such a rate that, for any $r > 0$,
	\begin{equation}
		\int_0^r \tilde{f}(s) \D s = \infty \qandq \int_0^r s \tilde{f}(s) \D s < \infty.
	\end{equation}
	This remaining growth regime has many interesting features, which we study in detail in \cite{CoRo22}. Here we present the following partial result.
	
	\begin{prop}\label{prop:inter-growth}
		Suppose that \Cref{ass:relaxed} holds and that there exists $\eta > 0$ such that $\tilde{f}$ is positive and decreasing on the interval $(0, \eta)$. If, for any $r > 0$,
		\begin{equation}
			\int_0^r \tilde{f}(s) \D s = \infty \qandq \int_0^r s \tilde{f}(s) \D s < \infty,
		\end{equation}
		then, for $V$ equal to the candidate value function defined in \Cref{def:candidate-value},
		\begin{equation}
			v^W(0) = V(0) \in (- \infty, \infty).
		\end{equation}
		Moreover, if $d \geq 3$, then $v(0) = v^S(0) = v^W(0) = V(0)$.
	\end{prop}
	
	Note that we do not make any claim about the strong value function in dimension $d = 2$.
	We split the proof of \Cref{prop:inter-growth} into two lemmas, first proving the result for dimensions $d \geq 3$.
	
	\begin{lemma}\label{lem:inter-growth-high-dim}
		Under the conditions of \Cref{prop:inter-growth} with $d \geq 3$, we have
		\begin{equation}
			v(0) = v^S(0) = v^W(0) = V(0) \in (- \infty, \infty).
		\end{equation}
	\end{lemma}

	\begin{proof}
		In this case, we can follow the same argument as in the proof of \Cref{prop:decr-slow-growth} except that we replace the constant control 
		$\begin{bmatrix}
				e_1; & 0; & \dotsc; & 0
		\end{bmatrix}$ with $\frac{1}{d}I$, where $I$ is the $d$-dimensional identity matrix. Instead of following a one-dimensional Brownian motion at the origin, the controlled processes under $\sigma^N$ and $\sigma^\delta$ follow a scaled $d$-dimensional Brownian motion. We now verify that the approximation arguments in \Cref{prop:decr-slow-growth} hold with this change.
			
		We will use the Green's function for the $d$-dimensional Brownian motion $B$, as defined in Section 3.3 of the book \cite{morters_brownian_2010} of M\"orters and Peres. By Theorem 3.32 and 3.33 of \cite{morters_brownian_2010} and the radial symmetry of $f$, there are constants $C, C^\prime > 0$ such that, for any $\delta \in (0, \eta)$,
		\begin{equation}
		\begin{split}
			\EE^0\left[\int_0^{\tau_\delta} f(B_s) \D s\right] & = C \int_{B_\delta} \abs{y}^{2 - d} f(y) \D y = C^\prime \int_0^\delta r \tilde{f}(r) \D r \xrightarrow{\delta 
			\to 0} 0,
		\end{split}
		\end{equation}
		where the limit follows from the assumption that $\int_0^r s \tilde{f}(s) \D s < \infty$ for any $r > 0$. Hence, following the same arguments as in the proof of \Cref{prop:decr-slow-growth}, we deduce the desired result.
	\end{proof}
	
	Now suppose that $d = 2$. Note that, from the form of the Green's function for $2$-dimensional Brownian motion given in Theorem 3.34 of \cite{morters_brownian_2010}, we can see that the argument used for $d \geq 3$ is no longer valid. In the following lemma, we treat the weak control problem in dimension $d = 2$.
	
	\begin{lemma}\label{lem:inter-growth-weak}
		Under the conditions of \Cref{prop:inter-growth} with $d = 2$, the weak value function is given by
		\begin{equation}
			v^W(0) = V(0) \in (- \infty, \infty).
		\end{equation}
	\end{lemma}
	
	\begin{proof}
		Retaining the notation of the proof of \Cref{prop:decr-fast-growth}, we have that, for any $y \in D$ with $\abs{y} = \eta$,
		\begin{equation}\label{eq:ineq-inter-growth}
		\begin{split}
			v^W(0) & \geq \lim_{N \to \infty}v^W_N(0) = V(0) = 2\int_0^\eta \xi \tilde{f}(y) \D \xi + V(y),
		\end{split}
		\end{equation}
		by \Cref{prop:radial-symmetric-value} and the definition of $V$ in \Cref{def:candidate-value}.
		
		In Theorem 4.3 of \cite{larsson_relative_2020}, Larsson and Ruf prove that, for $d = 2$, there exists a weak solution $X^{\sigma^0}$ of the SDE
		\begin{equation}\label{eq:sde-tan-rem}
			\D X_t = \sigma^0(X_t) \D B_t; \quad X_0 = 0.
		\end{equation}
		The process $X^{\sigma^0}$ follows tangential motion starting from the origin, as defined in \Cref{def:tangential}. By \Cref{lem:deterministically-increasing}, we have $\big\lvert X^{\sigma^0}_t\big\rvert = \sqrt{t}$, for any $t \geq 0$, and so
		\begin{equation}
			\EE^0 \left[\int_0^{\tau_\eta} f(X^{\sigma^0}_s)\D s\right] = \int_0^{\eta^2} \tilde{f}(\sqrt{s}) \D s = 2 \int_0^\eta \xi \tilde{f}(\xi) \D \xi.
		\end{equation}
		
		Note that \Cref{ass:main-strengthened} holds, and so we can apply the dynamic programming principle from \Cref{prop:dpp} to see that, for any $y \in D$ with $\abs{y} = \eta$,
		\begin{equation}
		\begin{split}
			v^W(0) \leq v^S(0) & \leq \EE^0\left[\int_0^{\tau_\eta} f(X^{\sigma^0}_s)\D s + v^S(X^{\sigma^0}_{\tau_\eta})\right] = 2 \int_0^\eta \xi \tilde{f}(\xi) \D \xi + V(y),
		\end{split}
		\end{equation}
		using the result of \Cref{prop:decr-away-from-origin} that $v^S = V$ away from the origin.
		
		Combining the above inequality with \eqref{eq:ineq-inter-growth}, we have
                $v^W(0) = V(0)$, as required.
	\end{proof}
	
	\begin{figure}[h]
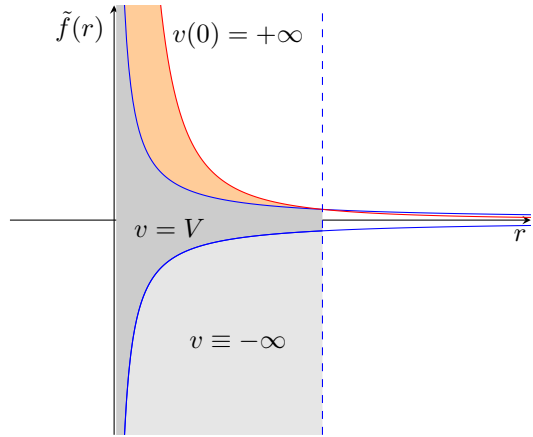

		\centering
		\vspace{1ex}
		\growth
		\vspace{1ex}
		\caption[Growth regimes for the cost function treated in \Cref{thm:value-relaxed}]{Figure showing the distinct growth regimes for the cost function in \Cref{thm:value-relaxed}, highlighting the case where, for any $r > 0$, $\int_0^r \tilde{f}(s) \D s = \infty$ and $\int_0^r s \tilde{f}(s)\D s < \infty$, as in \Cref{prop:inter-growth}.}
	\end{figure}
	
	We summarise the preceding results in the following extension of \Cref{prop:radial-symmetric-value}.
	
	\begin{thm}\label{thm:value-relaxed}
		Suppose that \Cref{ass:relaxed} is satisfied, and let $V: D \to \RR \cup \{\pm\infty\}$ be the candidate value function defined in \Cref{def:candidate-value}. Let $x \in D$ and suppose, moreover, that one of the following conditions holds:
		\begin{enumerate}[label = (\roman*)]
			\item $\tilde{f}$ is increasing on the interval $(0, \eta)$;
			\item $\tilde{f}$ is decreasing on the interval $(0, \eta)$ and $x \in D \setminus \{0\}$;
			\item $\tilde{f}$ is decreasing on the interval $(0, \eta)$, $x = 0$ and, for any $r > 0$, $\int_0^r \tilde{f}(s) < \infty$;
			\item $\tilde{f}$ is decreasing on the interval $(0, \eta)$, $x = 0$ and, for any $r > 0$, $\int_0^r s\tilde{f}(s) = \infty$;
			\item $\tilde{f}$ is decreasing on the interval $(0, \eta)$ and $d \geq 3$.
		\end{enumerate}
		Then the value function is given by
		\begin{equation}
			v(x) = v^S(x) = v^W(x) = V(x).
		\end{equation}
		Furthermore, we can determine when the value function is finite. If $\tilde{f}$ is increasing on the interval $(0, \eta)$, then
		\begin{equation}
			\begin{cases}
				v > - \infty, & \text{if} \quad \int_0^r \tilde{f}(s) \D s > - \infty \quad \text{for any} \quad r > 0,\\
				v \equiv - \infty, & \text{if} \quad \int_0^r \tilde{f}(s) \D s = - \infty \quad \text{for any} \quad r > 0.
			\end{cases}
		\end{equation}
		If $\tilde{f}$ is decreasing on the interval $(0, \eta)$, then $v(x) < \infty$ for $x \in D \setminus \{0\}$, and
		\begin{equation}
			\begin{cases}
				v(0) = \infty, & \text{if} \quad \int_0^r s\tilde{f}(s) \D s = \infty \quad \text{for any} \quad r > 0,\\
				v(0) < \infty, & \text{if} \quad d \geq 3 \quad \text{and} \quad \int_0^r s\tilde{f}(s) \D s < \infty \quad \text{for any} \quad r > 0,\\
				v(0) < \infty, & \text{if} \quad d = 2 \quad \text{and} \quad \int_0^r \tilde{f}(s) \D s < \infty \quad \text{for any} \quad r > 0.
			\end{cases}
		\end{equation}
	\end{thm}

	We now discuss what remains to find the strong value function under the assumptions of \Cref{prop:inter-growth} in the case $d = 2$.
	
	\begin{remark}
		Recall that, in \Cref{prop:weak-strong}, we appealed to Theorem 4.5 of El Karoui and Tan's paper \cite{el_karoui_capacities_2013-1} to show equality between weak and strong value functions, under the assumption that the cost function $f$ is upper semicontinuous and bounded above by a constant.
		
		Under the assumptions of \Cref{prop:inter-growth}, we cannot apply Theorem 4.5 of \cite{el_karoui_capacities_2013-1}, since one of the conditions of that theorem is no longer satisfied. Namely, in our setup, Theorem 4.5 of \cite{el_karoui_capacities_2013-1} is only applicable if the random variable $F_\tau := \int_0^\tau f(X_s) \D s$ is bounded above by some random variable that is uniformly integrable under the family of probability measures $\mathcal{P}_0$ defined in \Cref{sec:problem-formulation}. We show that this condition is not satisfied as follows.
		
		Let $e_1$ be the unit vector in the first coordinate direction and define $X^1$ by $X^1_t = e_1 B_t$, for $t \geq 0$. Then let $\PP^{X^1}$ be the law of the process $X^1$ and define the product measure $\PP := \PP^{X^1} \times \delta_{e_1} \in \mathcal{P}_0$. Following the same Green's function calculation as in \eqref{eq:green-incr-cost} in the proof of \Cref{prop:incr-slow-growth}, we compute that
		\begin{equation}
			\EE^{\PP}\left[\int_0^\tau f(X_s) \D s\right] = \int_0^R (R - r) \tilde{f}(r) \D r = +\infty,
		\end{equation}
		due to the growth condition on $\tilde{f}$ at the origin. Hence there does not exist any uniformly integrable upper bound on $F_\tau$ and Theorem 4.5 of \cite{el_karoui_capacities_2013-1} does not apply.
	\end{remark}

	In \Cref{lem:inter-growth-weak}, we found the weak value function at the origin by using the fact that there exists a weak solution of the SDE \eqref{eq:sde-tan-rem} describing tangential motion started from the origin. However, the SDE \eqref{eq:sde-tan-rem} has no strong solution, as we prove in \cite[Theorem 1.1]{CoRo22}. Since there exists no strong solution, we cannot follow the same argument as in the proof of \Cref{lem:inter-growth-weak} to find the strong value function. Nevertheless, we show in \cite[Theorem 4.1]{CoRo22} that the strong and weak value functions are in fact equal, using the theory of Brownian filtrations.
	
	\appendix
	\section{Dynamic programming and comparison principles}\label{app:dpp-comparison}
	
	In this appendix, we state the main results from the theory of dynamic programming and viscosity solutions that we use in the paper. The proofs of these results are fairly standard, but we have been unable to find versions of these results in the literature that completely cover the conditions required here. Full details can be found in the doctoral thesis \cite{robinson_stochastic_2020}. We require the following strengthening of \Cref{ass:main}.

	\begin{ass}\label{ass:main-strengthened}
		Suppose that \Cref{ass:main} holds and, moreover, the domain $D$ is strictly convex and the value function $v$ satisfies $v(x) > - \infty$, for any $x \in D$.
	\end{ass}
	
	\begin{prop}\label{prop:dpp}
		Suppose that \Cref{ass:main-strengthened} is satisfied. Then $v$ is continuous and the following dynamic programming principle holds. For any $x \in D$ and any stopping time $\theta$ with $\theta \in [0, \tau]$ almost surely, $v$ satisfies
		\begin{equation}\label{eq:dpp}
			v(x) = \inf_{\nu \in \mathcal{U}} \EE^x \left[\int_0^{\theta}f(X_s^{\nu}) \D s + v(X_{\theta}^{\nu})\right].
		\end{equation}
	\end{prop}
	
	\begin{remark}\label{rem:dpp}
		If there exists an optimal control $\sigma^\star \in \mathcal{U}$, then \eqref{eq:dpp} is equivalent to stating that
		\begin{equation}
			v(X^{\sigma}_t) + \int_0^t f(X^{\sigma}_s) \D s \quad \text{is} \quad
			\begin{cases}
				\text{a submartingale,} & \text{for all} \quad \sigma \in \mathcal{U},\\
				\text{a martingale,} & \text{for} \quad \sigma = \sigma^\star.
			\end{cases}
		\end{equation}	
	\end{remark}

	\begin{prop}\label{prop:hjb-comparison}
		Suppose that \Cref{ass:main} holds and that $f: D \to \RR$ is continuous. Then we have the following comparison principle for the HJB equation
		\begin{equation}\label{eq:hjb}
		- \frac{1}{2} \inf_{\sigma \in U} \trace	 \left(D^2 v \sigma \sigma^{\top}\right) - f = 0.
		\end{equation}
		Suppose that $u \in \usc(\overline{D})$ is a viscosity subsolution of \eqref{eq:hjb}, $v \in \lsc(\overline{D})$ is a viscosity supersolution of \eqref{eq:hjb}, and $u \leq v$ on $\partial D$. Then $u \leq v$ on $\overline{D}$.
	\end{prop}
	
	\begin{thm}\label{thm:unique-viscosity}
		Suppose that \Cref{ass:main-strengthened} holds, and suppose further that the domain $D$ is uniformly convex, the running cost $f$ is continuous in $D$, and the boundary cost $g$ is uniformly continuous on $\partial D$.
		
		Then the value function $v: D \to \RR$ defined in \Cref{sec:problem-formulation} extends continuously to $\overline{D}$ and is the unique viscosity solution of the HJB equation \eqref{eq:hjb} in $D$, with boundary condition $v = g$ on $\partial D$.
	\end{thm}

	\bibliographystyle{abbrv}
	\bibliography{bibliography}
\end{document}